\documentclass[11pt,reqno]{amsart}
\usepackage{caption2}
\usepackage{amsfonts}
\usepackage{amsmath,amsthm,amssymb,url,listings,amsfonts,amscd}
\usepackage[alphabetic,initials]{amsrefs}
\usepackage{mathrsfs}
\usepackage{lipsum}
\usepackage{bbding}
\usepackage{graphicx,latexsym}
\usepackage{hyperref}
\usepackage{scrtime}
\usepackage{color}
\usepackage[normalem]{ulem}

\newcommand{\bea}{\begin{eqnarray}}
\newcommand{\eea}{\end{eqnarray}}
\newcommand{\bna}{\begin{eqnarray*}}
\newcommand{\ena}{\end{eqnarray*}}

\numberwithin{equation}{section}

\theoremstyle{plain}
\newtheorem{lemma}{Lemma}[section]
\newtheorem{theorem}[lemma]{Theorem}
\newtheorem{corollary}[lemma]{Corollary}

\theoremstyle{definition}
\newtheorem{definition}[lemma]{Definition}
\newtheorem{remark}{Remark}

\renewcommand{\Re}{\operatorname{Re}}

\begin{document}
	
	\title{Analytic twists of $\rm GL_3\times \rm GL_2$ automorphic forms}

	\author{Yongxiao Lin}
	\date{}
	\address{EPFL SB Mathgeom Tan, Station 8, CH-1015, Lausanne, Switzerland}
	\email{yongxiao.lin@epfl.ch}
	
		\author{Qingfeng Sun}
	\date{}
	\address{School of Mathematics and Statistics, Shandong University, Weihai\\Weihai, Shandong
		264209, China}
	\email{qfsun@sdu.edu.cn}
	
%	\date{\today,\ \thistime}

	\begin{abstract}
Let $\pi$ be a Hecke--Maass cusp form for $\rm SL_3(\mathbb{Z})$
with normalized Hecke eigenvalues $\lambda_{\pi}(n,r)$. Let $f$ be a holomorphic or Maass cusp form for $\rm SL_2(\mathbb{Z})$
with normalized Hecke eigenvalues $\lambda_f(n)$. In this paper,
we are concerned with obtaining nontrivial estimates for the sum
\begin{equation*}
\sum_{r,n\geq 1}\lambda_{\pi}(n,r)\lambda_f(n)e\left(t\,\varphi(r^2n/N)\right)V\left(r^2n/N\right),
\end{equation*}
where $e(x)=e^{2\pi ix}$,
$V(x)\in \mathcal{C}_c^{\infty}(0,\infty)$, $t\geq 1$ is a large parameter and
$\varphi(x)$ is some real-valued smooth function.
As applications, we give an improved subconvexity bound for $\rm GL_3\times \rm GL_2$ $L$-functions
in the $t$-aspect, and under the Ramanujan--Petersson conjecture we derive the following bound for sums of $\rm GL_3\times \rm GL_2$ Fourier coefficients
\begin{equation*}
\sum_{r^2n\leq x}\lambda_{\pi}(r,n)\lambda_f(n)\ll_{\pi, f, \varepsilon} x^{5/7-1/364+\varepsilon}
\end{equation*}
for any $\varepsilon>0$, which breaks for the first time the barrier $O(x^{5/7+\varepsilon})$ in a work by
Friedlander--Iwaniec.

%we aim to obtain cancellation for GL(3)$\times$GL(2) Hecke eigenvalues against
%certain exponential sums. In particular, as an application, the result improves
%estimate of Munshi's $t$-aspect subconvexity bound for GL(3)$\times$GL(2)
%$L$-functions.
	
\end{abstract}
\thanks{Y. L. was partially supported under Philippe Michel's  DFG-SNF lead agency program grant (Grant  200020L\_175755). Q. S. was partially
  supported by the National Natural Science Foundation
  of China (Grant Nos. 11871306 and 12031008)}
	
	
	\maketitle
	
	%\tableofcontents
	\section{Introduction}\label{introduction}

Let $\{\lambda_F(n):n\geq 1\}$ be the Hecke eigenvalues of a $\rm GL_m$ automorphic form $F$
and let $\{e(t\varphi(n)):n\geq 1\}$  be a family of exponential functions, where
$t\geq 1$ is a large parameter and $\varphi(x)$ is some real-valued smooth function.
Here as usual, $e(x)=\exp(2\pi ix)$.
It is widely believed that the Hecke eigenvalues $\{\lambda_F(n):n\geq 1\}$ and
the exponential functions $\{e(t\varphi(n)):n\geq 1\}$ are not correlated, in the following sense
\bea\label{general-sum}
\sum_{n=1}^{\infty}\lambda_F(n)e\left(t\varphi\bigg(\frac{n}{X_m}\bigg)\right)V\bigg(\frac{n}{X_m}\bigg)
\ll_{F, V, A} X_m (\log X_m)^{-A},
\eea
for $X_m\ll t^{m/2+\varepsilon}$ and any $A\geq 1$.
(For the case where $X_m\gg t^{m/2+\varepsilon}$ one can apply the
functional equation of $L(s,F)$ to either produce a nontrivial estimate
or to reduce that case to the present setting.)

In applications, one is more concerned with a power-saving estimate
\bea\label{general-sum2}
\sum_{n=1}^{\infty}\lambda_F(n)e\left(t\varphi\bigg(\frac{n}{X_m}\bigg)\right)V\bigg(\frac{n}{X_m}\bigg)\ll_{F, V} {X_m}^{1-\delta},
\eea
for some $\delta>0$. In general this can be a rather difficult problem. Indeed,
by taking $\varphi(x)=-(\log x)/2\pi$, then this would imply
a subconvexity bound for the $L$-function $L(s,F)$ in the $t$-aspect.
On the other hand, in this setting the square-root cancellation phenomenon should not be expected to hold in general, as the following example by Iwaniec, Luo and Sarnak \cite{ILS} shows
\bea\label{GL2}
\sum_{n=1}^{\infty} \lambda_F(n)e(-2\sqrt{n})V\left(\frac{n}{X}\right)=c_FX^{3/4}+O\big(X^{1/4+\varepsilon}\big),
\eea
where $F=f$ is a $\rm SL_2(\mathbb{Z})$ holomorphic cusp form and $c_F$ is a constant depending on $F$.
Nevertheless, nontrivial results towards \eqref{general-sum2} have been known for several cases.

$\bullet$ If $F=f$ is a $\rm SL_2(\mathbb{Z})$ holomorphic cusp form, it was first proved by Jutila in his influential
Tata lecture notes \cite{Jutila} that one has the following bound
$$\sum_{n=1}^{\infty}\lambda_f(n)e\left(t\varphi\bigg(\frac{n}{X_2}\bigg)\right)V\left(\frac{n}{X_2}\right)\ll_{f, \varphi, V} t^{1/3}X_2^{1/2+\varepsilon}$$
as long as $t^{2/3}\ll X_2\ll t^{4/3}$. Similar results for $F$ a Maass cusp form have been obtained by several other authors later.

$\bullet$ For $F=\pi$ a $\rm SL_3(\mathbb{Z})$ Maass cusp form and for the case where
$\varphi(n)=-(\log n)/2\pi$, it was proved by Munshi \cite{Mun1}  that one has
$$\sum_{n=1}^{\infty}\lambda_\pi(n,1)e\left(t\varphi\bigg(\frac{n}{X_3}\bigg)\right)V\bigg(\frac{n}{X_3}\bigg)\ll_{\pi, V}
\begin{cases}t^{1/2}X_3^{5/8+\varepsilon}, & \quad \text{if} \quad t^{24/17}<X_3<t^{3/2+\varepsilon},\\
t^{11/10}X_3^{1/5+\varepsilon}, &\quad \text{if} \quad t^{11/8}<X_3\leq t^{24/17}.
\end{cases}$$
This was later strengthened to
$$\sum_{n=1}^{\infty}\lambda_\pi(n,1)e\left(t\varphi\bigg(\frac{n}{X_3}\bigg)\right)V\bigg(\frac{n}{X_3}\bigg)\ll_{\pi, V} t^{3/10}X_3^{3/4+\varepsilon},$$
as long as $t^{6/5}\ll X_3$ by Aggarwal \cite{Agg2}.
In the same setting, if we let $X_3=t^{1/\beta}$ and consider
$\varphi(n)=n^\beta$, then it was shown by Kumar, Mallesham, and Singh \cite{KMS19} that a bound similar to Aggarwal's holds
$$\sum_{n=1}^{\infty}\lambda_\pi(n,1)e\left(t\varphi\bigg(\frac{n}{X_3}\bigg)\right)V\bigg(\frac{n}{X_3}\bigg)\ll_{\pi, V} t^{3/10}X_3^{3/4+\varepsilon};$$
see also \cite{Huang} for further extensions.

$\bullet$ For the case where $F=\pi\otimes f$, a Rankin--Selberg convolution of
a $\rm SL_3(\mathbb{Z})$ Maass cusp form $\pi$ and a $\rm SL_2(\mathbb{Z})$ cusp form $f$,
and for $\varphi(n)=-(\log n)/2\pi$, it was recently established by Munshi
\cite{Mun6} that the following bound holds
\bea\label{Munshi's}
\sum_{n=1}^{\infty}\lambda_{\pi}(n,r)\lambda_f(n)n^{-it}V\left(\frac{r^2n}{X_6}\right)\ll_{\pi, f,V}
r^{-1}t^{59/42+3\eta/2}X_6^{1/2+\varepsilon},
\eea
as long as $t^{3-\eta}\ll X_6$.

Motivated by this and with applications in mind, the goal of this paper will be two-fold.
On the one hand, we would like to extend this result to a more general setting,
that the twist $n^{-it}$ is replaced by a more general Archimedean character.
On the other hand, we aim to obtain an improved bound than the one stated.
We will now set up some basic notation for our consideration.

From now on, $\pi$ will always denote a Hecke--Maass cusp form for $\rm SL_3(\mathbb{Z})$
with normalized Hecke eigenvalues $\lambda_{\pi}(n,r)$ and
$f$ will be either a holomorphic or Maass cusp form for $\rm SL_2(\mathbb{Z})$
with normalized Hecke eigenvalues $\lambda_f(n)$. Define
\bea\label{natural-sum}
\mathscr{S}(N)=\sum_{r,n\geq 1}\lambda_{\pi}(n,r)
\lambda_f(n)e\left(t\varphi\bigg(\frac{r^2n}{N}\bigg)\right)V\left(\frac{r^2 n}{N}\right),
\eea
where $t\geq 1$ is a large parameter, $\varphi(x)$ is some non-constant real-valued smooth function, and
$V(x)\in \mathcal{C}_c^{\infty}(1,2)$ is a smooth function with
support contained in $(1,2)$ and with total variation $\text{Var}(V)\ll 1$,
further satisfying the condition
\begin{equation}\label{derivative-of-V}
V^{(j)}(x)\ll_j \triangle^j
\end{equation} for $j\geq 0$ with
$\triangle\ll t^{1-\varepsilon}$. Then our main result states as follows.
\begin{theorem}\label{main-theorem}
Let $\varphi(x)=c\log x$ or $cx^{\beta}$ with $0<\beta\leq 1/3$ and $c\in \mathbb{R}\backslash \{0\}$. Assume $\triangle\ll t^{1/2-\varepsilon}$.
For any $0<\gamma<4/5$, we have
\bna
\mathscr{S}(N) \ll_{\pi,f,\varepsilon}
t^{3/5}N^{3/4+\varepsilon}+t^{-9/10}N^{5/4+\varepsilon}+t^{-9\gamma/14}N^{1+\varepsilon}
\ena
for $\max\{t^{8/5+2\gamma},t^{12/5+\varepsilon}\}<N<t^{16/5}$.
\end{theorem}
\begin{remark}
(1)  The estimate should hold for more general function $\varphi(x)$ satisfying the condition $\varphi'(x)\cdot\left(\varphi(x^3)\right)''\leq 0$. This latter condition arises from
using the stationary phase method at different stages of the proof.
For functions that are not of this form, it is possible to apply a duality
principle (i.e., applying functional equation for the $L$-functions $L(s,\pi\otimes f)$)
to reduce those cases to the present setting and to produce some nontrivial results;
see \cite[Sec. 9]{LMS} for related discussions. On the other hand, the
assumption $\triangle\ll t^{1/2-\varepsilon}$ arises when we use stationary phase
analysis for certain oscillatory integral in our approach; see \eqref{assumption-on-Delta}.

\par
(2) For results that use a functional equation (or equivalently, Voronoi summation formula) to study twists of the form $\sum_{N< n\leq 2N}\lambda_{\pi}(n,r)
\lambda_f(n)e(\alpha\, n^\beta)$, the reader is referred to \cite{Fri-Iwa,Ren-Ye}; see also \cite{KP}.
In the other direction, in \cite{LMS}, sums where the twists are \emph{algebraic} in nature, that is, sums of the form $\sum_{r, n\geq 1}\lambda_{\pi}(n,r)
\lambda_f(n)K(n)V\left(r^2n/N\right)$, where $K$ is the Frobenius trace function of some $\ell$-adic sheaf modulo $q$, are considered.
\end{remark}

Since the test function $V$ in the statement allows oscillations, one can derive a nontrivial upper bound for the sum with a sharp cut. Indeed, one has the following
\begin{corollary}\label{sharp-cut-sum}
Same notation and assumptions as above. We have
\bna
\sum_{N\leq r^2n\leq 2N}\lambda_{\pi}(n,r)
\lambda_f(n)e\left(t\varphi\bigg(\frac{r^2n}{N}\bigg)\right)
\ll_{\pi,f,\varepsilon} t^{3/5}N^{3/4+\varepsilon}+t^{-9/10}N^{5/4+\varepsilon}
+ t^{-9\gamma/14}N^{1+\varepsilon}+ t^{-1/2}N^{1+\varepsilon}
\ena
for  $\max\{t^{8/5+2\gamma}, t^{5/2-\varepsilon}\}<N<t^{16/5}$.
\end{corollary}
Theorem \ref{main-theorem} admits an application in bounding $L$-functions on the critical line. We recall
\bna
L\left(s,\pi\otimes f\right)=\sum_{r,n\geq 1}\frac{\lambda_{\pi}(n,r)\lambda_{f}(n)}{(nr^{2})^{s}}.
\ena
An application of the functional equation implies the bound $L\left(1/2+it,\pi\otimes f\right)\ll t^{3/2+\varepsilon}$, which is commonly known as the convexity or trivial bound. The aforementioned bound \eqref{Munshi's} of Munshi's corresponds to
$$L\left(1/2+it,\pi\otimes f\right)\ll_{\pi,f,\varepsilon} t^{3/2-1/42+\varepsilon}.$$
We obtain an improved bound over Munshi's.
\begin{corollary}\label{subconvexity} We have
\bna
L\left(1/2+it,\pi\otimes f\right)\ll_{\pi,f,\varepsilon} t^{3/2-3/20+\varepsilon}.
\ena
\end{corollary}

\begin{remark}
(1) The possibility of improving the $t$-aspect subconvexity exponent is based on an observation made by Aggarwal.  In \cite{Mun1}, Munshi used Kloosterman's circle method together with a ``conductor-lowering" trick
\bea\label{conductor-lowering}
\frac{1}{K}\int_{\mathbb{R}}V\left(\frac{v}{K}\right)\left(\frac{n}{m}\right)^{iv}\mathrm{d}v
\eea that restricts $|n-m|<t^{\varepsilon}N/K$ (here $n,m\asymp N$), to obtain subconvexity bound for $L(1/2+it,\pi)$ in the $t$-aspect. Aggarwal \cite{Agg2} showed that not only
\eqref{conductor-lowering} is removable, but also the subconvex exponent can be improved. Here we demonstrate once again that the trick \eqref{conductor-lowering} in \cite{Mun6} can also be removed in the $\rm GL_3\times \rm GL_2$ scenario.
%The point is that the
%conductor-decreasing effect has been built in the expression of
%Kloosterman's circle method; more precisely the oscillation of the factor
%$e(-n\zeta/aq)$ in \eqref{Kloosterman's} would play a similar role for us as the integral \eqref{conductor-lowering} did for Munshi.
  \par

 (2) The exponent $3/20$ of saving over the trivial bound is consistent with the exponent $3/40$ in the $\rm{GL}_3$ case in the depth and $t$-aspects \cite{SZ, Agg2}. It would be interesting to explore why such an exponent of saving shows up naturally.
\end{remark}

Now we proceed to describe another application of Theorem \ref{main-theorem}.
Let $$A_m(x)=\sum_{n\leq x}\lambda_F(n),$$ where $\lambda_F(n)$'s are the Fourier--Whittaker coefficients of some fixed $\rm GL_m$ automorphic form $F$ satisfying a functional equation relating $L(s,F)$ and $L(1-s,\widetilde{F})$. Here $\widetilde{F}$ is the dual form of $F$.
In \cite{Fri-Iwa}, a general principle on how to get stronger bounds for the sum $A_m(x)$ was described.

Let $(\mu_{1,F},\ldots, \mu_{m,F})$ be the Satake parameter of $F$ at $\infty$.
Let $$B(x,N)=\sum_{n\leq N}\overline{\lambda_F(n)}\, n^{-\frac{m+1}{2m}}\cos(2\pi m\left(nx\right)^{1/m})$$
denote the ``dual" sum associated to $A_m(x)$.
%In \cite{Fri-Iwa},
%in addition to having obtained some improved results for $m=3$ case for some $\rm GL_3$ Eisenstein series $F$, for a general automorphic form $F$,
Under the Ramanujan--Petersson conjecture $\lambda_F(n)\ll n^{\varepsilon}$ and the Selberg conjecture $\mathrm{Re}(\mu_{j,F})=0$\footnote{A careful reader might have already noticed that the dependence on the Selberg conjecture in Friedlander--Iwaniec's proof is very mild: they only assumed that the $\gamma$-factor $L_\infty(s,F)=\prod_{1\leq j\leq m}\Gamma_\mathbb{R}(s-\mu_{j,F})$ does not have poles for $\mathrm{Re}(s)>\frac{1}{2}+\frac{1}{m}$, when they are shifting contour for an integral involving ratio of $L_\infty(s,F)$. Here $\Gamma_\mathbb{R}(s)=\pi^{-s/2}\Gamma(s/2)$.}\label{Footnote1}, Friedlander and Iwaniec established the following identity which relates $A_m(x)$ to its dual sum $B(x,N)$
\bea\label{FI-Functional-eq}
A_m(x)=\mathrm{Res}_{s=1}\frac{L(s,F)}{s}x+c_F\, x^{\frac{m-1}{2m}}B(x,N)+O\left(N^{-\frac{1}{m}}x^{\frac{m-1}{m}+\varepsilon}\right),
\eea
where $c_F$ is some constant depending on the form $F$ only.
In particular, by estimating the sum $B(x,N)$ trivially and choosing
$N=x^{(m-1)/(m+1)}$, Friedlander and Iwaniec showed that\footnote{Landau's lemma also provides a similar result, provided the coefficients $\lambda_F(n)$'s are non-negative.}
\bea\label{trivial-one}
A_m(x)=\mathrm{Res}_{s=1}\frac{L(s,F)}{s}x+O\left(x^{\frac{m-1}{m+1}+\varepsilon}\right),
\eea
which should be regarded as the convexity bound.

Moreover, if one believed $B(x,N)=O((xN)^{\varepsilon})$, then we are led to the following conjectural optimal bound (\cite[Conjecture 2]{Fri-Iwa}) for the error term
\bea\label{conjectural-one}
A_m(x)=\mathrm{Res}_{s=1}\frac{L(s,F)}{s}x+O\left(x^{\frac{m-1}{2m}+\varepsilon}\right).
\eea
%Particularly, if the $L$-function $L(s,F)$ has no pole at $s=1$, then one should expect the following bound to be valid
%$$A_m(x)=O\left(x^{\frac{m-1}{2m}+\varepsilon}\right).$$

Results better than \eqref{trivial-one}
seem to be available for the divisor functions $\lambda_F(n)=\tau_m(n)$ or divisor-function like coefficients only, where
the forms $F$ are certain Eisenstein series so that the Fourier
coefficients admit certain factorization (e.g,
for $\lambda_F(n)=\tau(n)$ the divisor function or $\lambda_F(n)=\sum_{n_1n_2n_3=n}\chi_1(n_1)
\chi_2(n_2)\chi_3(n_3)$, $\chi_j$ being primitive Dirichlet characters). For results
without assuming the Ramanujan--Petersson conjecture and the Selberg conjecture but with bounds weaker
than \eqref{trivial-one}, we refer the readers to \cite{Lv09} and the references therein.

Our main result can be used to improve \eqref{trivial-one} for $m=6$ in a special case that $F$ admits factorization. Let $F=\pi\otimes f$, Rankin--Selberg convolution of $\rm GL_3$ and $\rm GL_2$ cusp forms, whose $n$-th Fourier coefficient is given by $\lambda_F(n)=\sum_{n_1^2n_2=n}\lambda_\pi(n_1,n_2)\lambda_f(n_2)$. Then \eqref{trivial-one} reads
$$\sum_{r^2n\leq x}\lambda_\pi(r,n)\lambda_f(n)=O\left(x^{5/7+\varepsilon}\right),$$
while the conjectural bound \eqref{conjectural-one} predicts that
$$\sum_{r^2n\leq x}\lambda_\pi(r,n)\lambda_f(n)=O\left(x^{5/12+\varepsilon}\right).$$
Note that in this case we do not have a main term in \eqref{trivial-one} and \eqref{conjectural-one}, since $L(s,\pi\otimes f)$ is holomorphic at $s=1$.

By taking $\varphi(n)=\pm 6 n^{1/6}$ in Theorem \ref{main-theorem} one can derive the following.
\begin{corollary}\label{sum-Fourier-coefficients}
Under the Ramanujan--Petersson conjecture, we have
\bna
\sum_{r^2n\leq x}\lambda_{\pi}(r,n)\lambda_f(n)\ll_{\pi,f,\varepsilon} x^{5/7-1/364+\varepsilon}.
\ena
\end{corollary}
\begin{remark}This bound is independent of the Selberg conjecture
for $\pi$ and $f$. Indeed, as we mentioned, Friedlander--Iwaniec's proof for
the identity \eqref{FI-Functional-eq} only assumed that
$L_\infty(s,F)=\prod_{1\leq j\leq m}\Gamma_\mathbb{R}(s-\mu_{j,F})$ has
no poles when $\Re(s)>\frac{1}{2}+\frac{1}{m}$. For $F=\pi\otimes f$,
$L_\infty(s,\pi\otimes f)=\prod_{1\leq i\leq 2, 1\leq j\leq 3}\Gamma_\mathbb{R}(s-\mu_{i,f}-\mu_{j,\pi})$, where $(\mu_{1,f},\mu_{2,f})=(-\frac{k-1}{2}, -\frac{k+1}{2})$ if $f$
is holomorphic and $|\Re(\mu_{i,f})|\leq \theta_2=7/64$ if $f$ is Maass
and $|\Re(\mu_{j,\pi})|\leq \theta_3=5/14$ by Kim--Sarnak \cite{K}.
Hence $L_\infty(s,\pi\otimes f)$ is holomorphic
for $\Re(s)>\theta_2+\theta_3=209/448$. In particular, it does not have poles when $\Re(s)>\frac{1}{2}+\frac{1}{6}$.
\end{remark}
%{\color{red}\sout{Though this is still far away from the conjectural optimal bound $O(x^{5/12+\varepsilon})$, it seems to be the first time when a result better than \eqref{trivial-one} for $m>3$ has been obtained.}}

\medskip
The paper is organized as follows. In Section \ref{sketch-of-proof}, we
provide a quick sketch and key steps of the proof. In
Section \ref{review-of-cuspform}, we review some basic materials of
automorphic forms on $ \rm GL_2$ and $ \rm GL_3$. Sections \ref{details-of-proof}
and \ref{proofs-of-technical-lemmas} give details of the proof for Theorem \ref{main-theorem}
and in Section \ref{proofs-of-corollaries} we complete the proofs
for Corollaries \ref{sharp-cut-sum}, \ref{subconvexity}, and \ref{sum-Fourier-coefficients}.

\bigskip

\noindent
{\bf Acknowledgements.} The first named author would like to thank Philippe Michel and Chandra Sekhar Raju for many illuminating discussions. We are grateful to Bingrong Huang for many helpful discussions. We also thank Keshav Aggarwal, Guanghua Ji, and Zhi Qi for pointing out several errors in the first version of the manuscript and for some valuable comments. Finally we are very grateful to the referees for their detailed comments, corrections, and suggestions.

\bigskip
\noindent
{\bf Notation.}
Throughout the paper, the letters $q$, $m$ and $n$, with or without subscript,
denote integers. The letters $\varepsilon$ and $A$ denote arbitrarily small and large
positive constant, respectively, not necessarily the same at different occurrences.
We use $A\asymp B$ to mean that $c_1B\leq |A|\leq c_2B$ for some positive constants $c_1$ and $c_2$. The symbol
$\ll_{a,b,c}$ denotes that the implied constant depends at most on $a$, $b$ and $c$, and
$q\sim C$ means $C<q\leq 2C$. As usual, $e(x)=e^{2\pi ix}$. The $\star$ in
$\sideset{}{^\star}\sum\limits_{a\bmod q}$ means that the sum is restricted to $(a,q)=1$.
%such as $\frac{ab}{cd}$ will be written as $ab/cd$ and $a/b+c$ or $c+a/b$ means
%$\frac{a}{b}+c$.

\section{Outline of the proof}\label{sketch-of-proof}
To rigorously complete the proof, the details are quite complicated.
In this section, we provide a quick sketch of the proof to guide
the reader through the essential parts of our argument.
Since the parameter $r$ in the $\mathrm{GL}_3$
Hecke eigenvalue $\lambda_\pi(n,r)$ in our actual proof is fixed, we assume $r=1$ here.
Suppose we are working with the following sum
\bna
S(N)=\sum_{n\sim N}\lambda_{\pi}(n,1)\lambda_f(n)e\left(t\varphi\bigg(\frac{n}{N}\bigg)\right),
\ena
where we have suppressed the smooth test function $V\left(n/N\right)$ which controls the length $N\leq n\leq 2N$ from the notation (and we also suppress test functions in various sums in this section).
Here the amplitude of the phase function is of size
\bna
t\varphi\bigg(\frac{n}{N}\bigg)\asymp t.
\ena

Our first step is to follow \cite{Mun6} and write
\bna
S(N)=\sum_{n\sim N}\lambda_{\pi}(n,1)\sum_{m\sim N}\lambda_f(m)e\left(t\varphi\bigg(\frac{m}{N}\bigg)\right)\delta(m-n,0),
\ena
and to use some form of delta symbol method to detect the
Kronecker delta symbol $\delta(m-n,0)$. We take a similar approach by using the Duke--Friedlander--Iwaniec's delta method \eqref{DFI's} given in \cite{IK}, but unlike \cite{Mun6} we remove the conductor-decreasing trick \eqref{conductor-lowering}, to get
\bea\label{before-voronoi}
S(N)&=&\int_{-N^{\varepsilon}}^{N^{\varepsilon}}\frac{1}{Q} \sum_{1\leq q\leq Q}\,\sideset{}{^\star}\sum_{a\bmod{q}}\frac{1}{q}
\sum_{n\sim N}\lambda_{\pi}(n,1)e\left(-\frac{na}{q}\right)
e\left(-\frac{n\zeta}{qQ}\right)\nonumber\\
&&\sum_{m\sim N}\lambda_f(m)e\left(\frac{ma}{q}\right)
e\left(t\varphi\bigg(\frac{m}{N}\bigg)+\frac{m\zeta}{qQ}\right)\mathrm{d}\zeta.
\eea
(Notice that the weight function $g(q,\zeta)$ in \eqref{DFI's} behaves almost like the constant function $1$ in the generic range.)

We now use Voronoi summations to dualize the $m$- and $n$-sums. By applying the $\mathrm{GL}_2$-Voronoi summation and taking into account of the oscillation of the $\mathrm{GL}_2$-Bessel function in the dual sum, we get
\bea\label{GL2dual}
\sum_{m\sim N}\lambda_f(m)e\left(\frac{ma}{q}\right)e\left(t\varphi\bigg(\frac{m}{N}\bigg)+\frac{m\zeta}{qQ}\right)V\left(\frac{m}{N}\right)\leftrightarrow \frac{N}{qt^{1/2}}\sum_{\pm}\sum_{m\sim q^2t^2/N}\lambda_f(m)e\left(\mp\frac{m\bar{a}}{q}\right)\Phi^{\pm}\left(m,q,\zeta\right),
\eea
where
 \bna
\Phi^{\pm}\left(m,q,\zeta\right)=\int_0^\infty V(y)
e\left(t\varphi(y)+\frac{\zeta Ny}{qQ}\pm\frac{2\sqrt{mNy}}{q}\right)\mathrm{d}y.
\ena
If we assume for example $\varphi'(y)<0$, then by integration by parts, $\Phi^{-}\left(m,q,\zeta\right)\ll N^{-A}$, and we would only need to consider the plus sign contribution.

Similarly, applying the $\mathrm{GL}_3$-Voronoi summation gives
\bea\label{GL3dual}
\sum_{n\sim N}\lambda_{\pi}(n,1)e\left(-\frac{na}{q}\right)
e\left(-\frac{n\zeta}{qQ}\right)U\left(\frac{n}{N}\right)\leftrightarrow \frac{N}{q^2}\frac{q}{N^{1/3}}
\sum_{n}\frac{\lambda_{\pi}(1,n)}{n^{1/3}}S(\bar{a},n;q)\Psi^{+}\left(n,q,\zeta\right),
\eea
where
\bna
\Psi^{+}\left(n,q,\zeta\right)=\int_0^{\infty}U\left(y\right)
e\left(-\frac{\zeta N y}{qQ}+ 3\left(\frac{N n}{q^3}\right)^{1/3}y^{1/3}\right)\mathrm{d}y.
\ena
We perform a stationary phase argument to get
\bna
\Psi^{+}\left(n,q,\zeta\right)\asymp \frac{q^{1/2}}{n^{1/6}N^{1/6}}e\left(2\left(\frac{Qn}{q^2\zeta}\right)^{1/2}\right)U_T\left(\frac{n^{1/2}Q^{3/2}}{\zeta^{3/2}N}\right),
\ena
where $U_T(x)$ is some smooth function with compact support, which also controls
the length of the dual sum to $n\sim \frac{\zeta^3 N^2}{Q^3}$.
Therefore by applying the $\mathrm{GL}_3$-Voronoi summation we get
\bna
\sum_{n\sim N}\lambda_{\pi}(n,1)e\left(-\frac{na}{q}\right)
e\left(-\frac{n\zeta}{qQ}\right)\leftrightarrow \frac{N^{1/2}}{q^{1/2}}
\sum_{n\sim \frac{\zeta^3 N^2}{Q^3}\ll \frac{N^2}{Q^3}}\frac{\lambda_{\pi}(1,n)}{n^{1/2}}S(\bar{a},n;q)e\left(2\left(\frac{Qn}{q^2\zeta}\right)^{1/2}\right)U_T\left(\frac{n^{1/2}Q^{3/2}}{\zeta^{3/2}N}\right).
\ena

Plugging the dual sums \eqref{GL2dual} and \eqref{GL3dual} back into \eqref{before-voronoi} and switching the orders of integration over $\zeta$ and $y$, we roughly get
\bea\label{intermediateS(N)}
\begin{split}
S(N)\approx& \frac{N^{3/2}}{Qt^{1/2}}\sum_{n\sim \frac{N^2}{Q^3}}\frac{\lambda_{\pi}(1,n)}{n^{1/2}}\sum_{1\leq q\leq Q}\frac{1}{q^{5/2}}\sideset{}{^\star}\sum_{a\bmod{q}}S(a,n;q)\sum_{m\sim q^2t^2/N}\lambda_f(m)e\left(-\frac{ma}{q}\right)\\
&\times\int_0^\infty V(y)
e\left(t\varphi(y)+\frac{2\sqrt{mNy}}{q}\right)\mathcal{K}(y;n,q)\, \mathrm{d}y
\end{split}
\eea
where
\bna
\mathcal{K}(y;n,q)=
\int_{-N^{\varepsilon}}^{N^{\varepsilon}}
U_T\left(\frac{n^{1/2}Q^{3/2}}{\zeta^{3/2}N}\right)
e\left(2\left(\frac{Qn}{q^2\zeta}\right)^{1/2}+\frac{\zeta Ny}{qQ}\right)
\mathrm{d}\zeta.
\ena
Here we have pretended the $n$-sum to be supported on the dyadic range $n\sim \frac{N^2}{Q^3}$.

We evaluate the integral $\mathcal{K}(y;n,q)$ using the stationary phase method
\bna
\mathcal{K}(y;n,q)\asymp Q\left(\frac{q^3n}{N^5}\right)^{1/6}e\left(\frac{3(nNy)^{1/3}}{q}\right)\widetilde{U}_T(y),
\ena
for some smooth compactly supported function $\widetilde{U}_T(y)$.

Hence putting things together, \eqref{intermediateS(N)} gets transformed to
\bna
\begin{split}
S(N)\approx& \frac{N^{2/3}}{t^{1/2}}\sum_{n\sim \frac{N^2}{Q^3}}\frac{\lambda_{\pi}(1,n)}{n^{1/3}}\sum_{1\leq q\leq Q}\frac{1}{q^{2}}\sum_{m\sim q^2t^2/N}\lambda_f(m)\sideset{}{^\star}\sum_{a\bmod{q}}S(a,n;q)e\left(-\frac{ma}{q}\right)\\
&\times\int_0^\infty V(y)\widetilde{U}_T(y)
e\left(t\varphi(y)+\frac{2\sqrt{mNy}}{q}+\frac{3(nNy)^{1/3}}{q}\right)\, \mathrm{d}y.
\end{split}
\ena
Note that the sum over $a$ can be evaluated explicitly,
\bna
\sideset{}{^\star}\sum_{a\bmod{q}}S(a,n;q)e\left(-\frac{ma}{q}\right)\approx q\,e\left(\frac{\overline{m}n}{q}\right).
\ena
Therefore, we arrive at
\begin{equation*}
S(N)\approx \frac{N^{2/3}}{t^{1/2}}\sum_{n\sim \frac{N^2}{Q^3}}\frac{\lambda_{\pi}(1,n)}{n^{1/3}}\sum_{1\leq q\leq Q}\frac{1}{q}\sum_{m\sim q^2t^2/N}\lambda_f(m)e\left(\frac{\overline{m}n}{q}\right)\mathfrak{J}^{\sharp}(n,m,q),
\end{equation*}
where
\bea\label{J-integral0}
\mathfrak{J}^{\sharp}(n,m,q)=\int_0^\infty V(y)\widetilde{U}_T(y)
e\left(t\varphi(y)+\frac{2\sqrt{mNy}}{q}+\frac{3(nNy)^{1/3}}{q}\right)\, \mathrm{d}y.
\eea
\begin{remark}If we were to apply the bound
\bna
\mathfrak{J}^{\sharp}(n,m,q)\ll t^{-1/2}
\ena
which follows from the second derivative test and then estimate $S(N)$ trivially, we would have obtained the bound
\bna
S(N)\ll Nt,
\ena
which falls short of $O(tN^{\eta})$ from the target bound $O(N^{1-\eta})$.
\end{remark}

To prepare for an application of the Poisson summation in the $n$-variable, we now apply the Cauchy--Schwarz inequality to smooth the $n$-sum and putting all the other sums inside the absolute value squared,
\begin{equation*}
S(N)\ll \frac{N^{2/3}}{t^{1/2}}\bigg(\sum_{n\sim \frac{N^2}{Q^3}}
\frac{|\lambda_{\pi}(1,n)|^2}{n^{2/3}}\bigg)^{1/2}\bigg(\sum_{n\sim \frac{N^2}{Q^3}}\bigg|\sum_{1\leq q\leq Q}\frac{1}{q}\sum_{m\sim q^2t^2/N}\lambda_f(m)e\left(\frac{\overline{m}n}{q}\right)\mathfrak{J}^{\sharp}(n,m,q)\bigg|^2\bigg)^{1/2}.
\end{equation*}
\begin{remark}
If we open the absolute value squared, the contribution from the diagonal term $(q,m)=(q',m')$ is given by
\begin{equation}\label{S-diagonal}
\begin{split}
S_{\text{diag}}\ll& \frac{N^{2/3}}{t^{1/2}}\bigg(\sum_{n\sim \frac{N^2}{Q^3}}\frac{|\lambda_{\pi}(1,n)|^2}{n^{2/3}}\bigg)^{1/2}\bigg(\sum_{n\sim \frac{N^2}{Q^3}}\sum_{1\leq q\leq Q}\frac{1}{q^2}\sum_{m\sim q^2t^2/N} |\lambda_f(m)|^2 |\mathfrak{J}^{\sharp}(n,m,q)|^2\bigg)^{1/2}\\
\ll& N\left(\frac{N}{Q^3}\right)^{1/2},
\end{split}
\end{equation}
which will be fine for our purpose (i.e., $S_{\text{diag}}=o(N)$) as long as $Q\gg N^{1/3}$.
\end{remark}

We remark before continuing that the oscillation of $\mathfrak{J}^{\sharp}(n,m,q)$ in \eqref{J-integral0}, in the $n$-variable is of size $\frac{3(nNy)^{1/3}}{q}\approx \frac{N}{qQ}$.
Now we open the absolute value squared and apply the Poisson summation in the $n$-variable, getting
\begin{equation*}
\begin{split}
\sum_{n\sim \frac{N^2}{Q^3}}e\left(\frac{\overline{m}n}{q}\right)e\left(-\frac{\overline{m'}n}{q'}\right)\mathfrak{J}^{\sharp}(n,m,q)\mathfrak{J}^{\sharp}(n,m',q')\leftrightarrow \frac{N^2}{Q^3}\sum_{\tilde{n}\ll \frac{qq'\frac{N}{qQ}}{N^2/Q^3}}\mathfrak{K}(\tilde{n})\,\mathfrak{I}\left(\frac{\frac{N^2}{Q^3}\tilde{n}}{qq'}\right),
\end{split}
\end{equation*}
where
\bea\label{correlation-integral}
\mathfrak{I}(x)=\int_{\mathbb{R}}
\mathfrak{J}^{\sharp}\left(N^2\xi/Q^3,m,q\right)
\overline{\mathfrak{J}^{\sharp}\left(N^2\xi/Q^3,m',q'\right)}
\, e\left(-x\xi\right)\mathrm{d}\xi
\eea
and
\bna
\begin{split}
\mathfrak{K}(\tilde{n})=&\frac{1}{qq'}\sum_{\beta(qq')}e\left(\frac{\overline{m}\beta}{q}\right)e\left(-\frac{\overline{m'}\beta}{q'}\right)e\left(\frac{\beta \tilde{n}}{qq'}\right)
\end{split}
\ena
which is $\approx \delta_{m\equiv -q'\overline{\tilde{n}}\bmod{q}}\, \delta_{m'\equiv q\overline{\tilde{n}}\bmod{q'}}$ for $\tilde{n}\neq 0$, and would imply $q=q'$ and $m\equiv m'\bmod{q}$ if $\tilde{n}=0$.

The contribution to $S(N)$ from the zero-frequency $\tilde{n}=0$
will roughly correspond to the diagonal contribution $S_{\text{diag}}$ in \eqref{S-diagonal}.
So we continue to analyze the case $\tilde{n}\neq 0$.

%If $|x|$ is ``small", then we do not have cancellation for the integration over $\xi$ and we could only bound the integral $\mathfrak{I}(x)$ trivially by $O(t^{-1})$. If the size of $|x|$ is in some middle range, then one can use the third derivative test to derive
%\bna
%\mathfrak{I}(x)\ll t^{-1/2}\cdot t^{-1/2}\cdot |x|^{-1/3};
%\ena
%while in the generic case that $|x|$ is ``large", by performing some stationary phase argument, one can show, in addition to square-root cancellation $\ll t^{-1/2}$ for each of the integral $\mathfrak{J}^{\sharp}\left(N^2\xi/Q^3,m,q\right)$ in \eqref{correlation-integral}, the square-root cancellation in the $\xi$-variable also holds
%\bna
%\mathfrak{I}(x)\ll t^{-1/2}\cdot t^{-1/2}\cdot |x|^{-1/2}.
%\ena
%(These are the content of Lemma \ref{integral:lemma}.)

For the triple integral $\mathfrak{I}(x)$ in \eqref{correlation-integral}, we expect, by performing stationary phase analysis, that when $|x|$ is ``large", in addition to square-root cancellation
\bna
\mathfrak{J}^{\sharp}\left(N^2\xi/Q^3,m,q\right)\ll t^{-1/2}
\ena
 for each of the inner integrals, the square-root cancellation in the $\xi$-variable
\bna
\mathfrak{I}(x)\ll t^{-1/2}\cdot t^{-1/2}\cdot |x|^{-1/2}
\ena
should hold.
This bound is not true if $|x|$ is ``small." We can effectively control those exceptional cases by using the bound $\mathfrak{I}(x)\ll t^{-1/2}\cdot t^{-1/2}\cdot |x|^{-1/3}$ which follows from the third derivative test provided that $|x|$ is in some intermediate range, and also by the ``trivial" bound $\mathfrak{I}(x)\ll t^{-1/2}\cdot t^{-1/2}$ in the case $|x|$ is ``very small." (These are the content of Lemma \ref{integral:lemma}.)

By plugging all these analysis back into $S(N)$, it turns out that the non-zero frequencies contribution $\Sigma_{\neq 0}$ from $\tilde{n}\neq 0$ to $S(N)$ is given by
\bna
\Sigma_{\neq 0}=O\left(N^{1/4}Qt\right).
\ena
Hence combining this with the diagonal contribution $S_{\text{diag}}$ in \eqref{S-diagonal}, we get
\bna
S(N)\ll N\left(\frac{N}{Q^3}\right)^{1/2}+N^{1/4}Qt.
\ena
By choosing $Q=N^{1/2}/t^{2/5}$ we obtain $S(N)\ll N^{3/4}t^{3/5}$, which improves over the trivial bound $S(N)\ll N$ as long as $t^{12/5}\ll N$.

\section{A review of automorphic forms}\label{review-of-cuspform}

\subsection{Hecke cusp forms for $\mathrm{GL}_2$}

\subsubsection{Holomorphic cusp forms for $\mathrm{GL}_2$}

Let $f$ be a holomorphic cusp form of weight $\kappa$ for $\rm SL_2(\mathbb{Z})$
with Fourier expansion
\bna
f(z)=\sum_{n=1}^{\infty}\lambda_f(n)n^{(\kappa-1)/2}e(nz)
\ena
for $\mbox{Im}\,z>0$. We assume that $f$ is a normalized Hecke eigenform
so that $\lambda_f(1)=1$.
By the Ramanujan--Petersson conjecture, which is proved by Deligne \cite{Del} for $\kappa\geq 2$ and by Deligne--Serre \cite{Del-Ser} for the case $\kappa=1$,
we have $\lambda_f(n)\ll \tau(n)$ with $\tau(n)$ being the divisor function.
Moreover, we have the Wilton-type bound (see \cite[Theorem 5.3]{Iwaniec})
\bea\label{additive-holo}
\sum_{n\leq X}\lambda_f(n)e(n \alpha)\ll_{f} X^{1/2}\log (2X),
\eea
which holds uniformly for any $\alpha\in \mathbb{R}$. The $\log (2X)$ factor can be removed; see \cite{Jutila1}.

For $h(x)\in \mathcal{C}_c(0,\infty)$, we set
\bea\label{intgeral transform-1}
\Phi_h(x) =2\pi i^{\kappa} \int_0^{\infty} h(y) J_{\kappa-1}(4\pi\sqrt{xy})\mathrm{d}y,
\eea
where $J_{\kappa-1}$ is the usual $J$-Bessel function of order $\kappa-1$.
We have the following Voronoi summation formula (see \cite[Theorem A.4]{KMV}).

\begin{lemma}\label{voronoiGL2-holomorphic}
Let $q\in \mathbb{N}$ and $a\in \mathbb{Z}$ be such
that $(a,q)=1$. For $X>0$, we have
\bna\label{voronoi for holomorphic}
\sum_{n=1}^{\infty}\lambda_f(n)e\left(\frac{an}{q}\right)h\left(\frac{n}{X}\right)
=\frac{X}{q} \sum_{n=1}^{\infty}\lambda_f(n)
e\left(-\frac{\overline{a}n}{q}\right)\Phi_h\left(\frac{nX}{q^2}\right),
\ena
where $\overline{a}$ denotes
the multiplicative inverse of $a$ modulo $q$.
\end{lemma}
It is known that for integer $v$, $J_{v}(x)$ can be expressed as
\bna
J_{v}(x)=x^{-1/2}\left(e^{ix}W_{v,+}(x)+
e^{-ix}W_{v,-}(x)\right)
\ena
where $W_{v,+}$ satisfies $x^jW_{v,\pm}^{(j)}(x)\ll_{v,j}1$ for $x\gg 1$ (see \cite[Sec. 4.5]{HMQ}).
We revisit this and derive an asymptotic formula for
$\Phi_h(x)$ when $x\gg 1$.

\begin{lemma}\label{voronoiGL2-holomorphic-asymptotic}
For any fixed integer $J\geq 1$ and $x\gg 1$, we have
\bna
\Phi_h(x)=x^{-1/4} \int_0^\infty h(y)y^{-1/4}
\sum_{j=0}^{J}
\frac{c_{j} e(2 \sqrt{xy})+d_{j} e(-2 \sqrt{xy})}
{(xy)^{j/2}}\mathrm{d}y
+O_{\kappa,J}\left(x^{-J/2-3/4}\right),
 \ena
where $c_{j}$ and $d_{j}$ are constants depending on $\kappa$.
\end{lemma}
\begin{proof}
 By Section 7.21 in \cite{Wat}, for $x\gg 1$,we have
 \bna
J_{v}(x)&=&\sqrt{\frac{2}{\pi x}}
\cos\left(x-\frac{\pi}{2}v-\frac{\pi}{4}\right)
\left\{\sum_{j=0}^{J-1}\frac{(-1)^j(v,2j)}{(2x)^{2j}}
+O_{n,v}\left(x^{-2J}\right)\right\}\nonumber\\
&&+\sqrt{\frac{2}{\pi x}}\sin\left(x-\frac{\pi}{2}v-\frac{\pi}{4}\right)
\left\{\sum_{j=0}^{J-1}\frac{(-1)^j(v,2j+1)}{(2x)^{2j+1}}
+O_{n,v}\left(x^{-2J-1}\right)\right\},
\ena
 where $(v,0)=1$ and for $j\geq 1$,
\bna
(v,j)=\frac{\Gamma(v+j+1/2)}{j!\Gamma(v-j+1/2)}
=\frac{\{4v^2-1^2\}\{4v^2-3^2\}\cdots \{4v^2-(2j-1)^2\}}{2^{2j}j!}.
\ena
 Thus for $x\gg 1$,
 \bea\label{BoundOfJ}
 J_{v}(x)=x^{-1/2} \sum_{j=0}^{J}
 \frac{c_j e\left(x/2\pi\right)+d_j e\left(-x/2\pi\right)}{x^j}
 +O_{v,J}\left(x^{-3/2-J}\right),
 \eea
 where $c_j=c_j(v)$ and $d_j=d_j(v)$, $j=0,1,2,\ldots,J$ are constants depending only on $v$.
 In particular, $c_0=(2\pi)^{-1/2}e(-(2v+1)/8)$ and $d_0=(2\pi)^{-1/2}e((2v+1)/8)$.
Plugging this formula into \eqref{intgeral transform-1}, we have
 \bna
\Phi_h(x)= \int_0^\infty h(y)(xy)^{-\frac{1}{4}}
\sum_{j=0}^{J}
\frac{c'_{j} e(2 \sqrt{xy})+d'_{j} e(-2 \sqrt{xy})}
{(xy)^{j/2}}\mathrm{d}y
+O_{\kappa,J}\left(x^{-J/2-3/4}\right),
 \ena
 where $c'_{j}$ and $d'_{j}$ are constants depending on $\kappa$. The lemma follows.

\end{proof}

\begin{remark}

For $x\ll 1$, usually the bound $x^jJ_{v}^{(j)}(x)\ll_{v,j}1$ would be sufficient in applications.
\end{remark}

\medskip

\subsubsection{Maass cusp forms for $\mathrm{GL}_2$}

Let $f$ be a Hecke--Maass cusp form for $\rm SL_2(\mathbb{Z})$
with Laplace eigenvalue $1/4+\mu^2$. Then $f$ has a Fourier expansion
$$
f(z)=\sqrt{y}\sum_{n\neq 0}\lambda_f(n)K_{i\mu}(2\pi |n|y)e(nx),
$$
where $K_{i\mu}$ is the modified Bessel function of the third kind.
The Ramanujan--Petersson conjecture, which asserts that
$\lambda_f(n)=O_{\mu,\varepsilon}(n^{\varepsilon})$, is
not known yet.
%, though positive evidence from
%numerical investigations have been made (see \cite{Hej}).
%The best record bound is
%\bea\label{individual bound}
%\lambda_f(n)\ll_{\varepsilon} n^{7/64+\varepsilon}.
%\eea
%By Rankin--Selberg theory,
%\bea\label{GL2: Rankin Selberg}
%\sum_{n\leq X}|\lambda_f(n)|^2\ll_{f} X.
%\eea
We do not need to make use of such an individual bound. Rather, the following
average bound\footnote{Recently Huang had an improvement on the error term; see \cite{Huang}.}
\bea\label{GL2: Rankin Selberg}
\sum_{n\leq X}|\lambda_f(n)|^2= c_{f} X+O\big(X^{3/5}\big)
\eea
which follows from the Rankin--Selberg theory and Landau's lemma (see for instance \cite[Lemma 1]{Murty}),
would be sufficient for our purpose. Actually, in order to make use of \eqref{GL2: Rankin Selberg} in replace of the Ramanujan--Petersson conjecture, later we need to perform some counting arguments more carefully; see Section \ref{counting-part}.

Similar to \eqref{additive-holo}, we recall (see \cite[Theorem 8.1]{Iwaniec1})
\bea\label{additive-Maass}
\sum_{n\leq X}\lambda_f(n)e(n \alpha)\ll_{f} X^{1/2}\log (2X),
\eea
which holds uniformly for any $\alpha\in \mathbb{R}$ (the factor $\log (2X)$
was recently removed by J\"{a}\"{a}saari and
Vesalainen \cite{J-V}).

For $h(x)\in \mathcal{C}_c(0,\infty)$, we define the integral transforms
\bea\label{intgeral transform-2}
\begin{split}
\Phi_h^+(x) =& \frac{-\pi}{\sin(\pi i\mu)} \int_0^\infty h(y)\left(J_{2i\mu}(4\pi\sqrt{xy})
- J_{-2i\mu}(4\pi\sqrt{xy})\right) \mathrm{d}y,\\
\Phi_h^-(x) =& 4\varepsilon_f\cosh(\pi \mu)\int_0^\infty h(y)K_{2i\mu}(4\pi\sqrt{xy}) \mathrm{d}y,
\end{split}\eea
where $\varepsilon_f$ the an eigenvalue under the reflection operator.
We have the following Voronoi summation formula (see \cite[Theorem A.4]{KMV}).

\begin{lemma}\label{voronoiGL2-Maass}
Let $q\in \mathbb{N}$ and $a\in \mathbb{Z}$ be such
that $(a,q)=1$. For $X>0$, we have
\bna\label{voronoi for Maass form}
\sum_{n=1}^{\infty}\lambda_f(n)e\left(\frac{an}{q}\right)h\left(\frac{n}{X}\right)
= \frac{X}{q} \sum_{\pm}\sum_{n=1}^{\infty}\lambda_f(n)
e\left(\mp\frac{\overline{a}n}{q}\right)\Phi_h^{\pm}\left(\frac{nX}{q^2}\right),
\ena
where $a \overline{a} \equiv 1(\bmod q)$.
\end{lemma}

By the integral representation of $K_{v}(z)$ (see \cite[(8.432-8)]{GR}), for
$| \arg z|<\pi$, $\Re(v)>-1/2$, one has
$$
K_{v}(z)=\sqrt{\frac{\pi}{2z}}\frac{e^{-z}}{\Gamma(v+1/2)}
\int_0^{\infty}e^{-u}u^{v-1/2}\left(1+\frac{u}{2z}\right)^{v-1/2}\mathrm{d}u.
$$
Thus for $v\in i\mathbb{R}$ and $z\gg 1$, we have
\bea\label{BoundOfK}
K_{v}(z)\ll_{v} z^{-1/2}e^{-z}.
\eea
For $x\gg 1$, by \eqref{intgeral transform-2} and \eqref{BoundOfK}, it is easily seen that
\bea\label{The $-$ case}
\Phi_h^-(x)\ll_{\mu,A}x^{-A}.
\eea
For $\Phi_h^+(x)$ and $x\gg 1$, we have a similar asymptotic formula as for
$\Phi_h(x)$ in the holomorphic case, namely:
\begin{lemma}\label{voronoiGL2-Maass-asymptotic}
For any fixed integer $J\geq 1$ and $x\gg 1$, we have
\bna
\Phi_h^{+}(x)=x^{-1/4} \int_0^\infty h(y)y^{-1/4}
\sum_{j=0}^{J}
\frac{c_{j} e(2 \sqrt{xy})+d_{j} e(-2 \sqrt{xy})}
{(xy)^{j/2}}\mathrm{d}y
+O_{\mu,J}\left(x^{-J/2-3/4}\right),
 \ena
where $c_{j}$ and $d_{j}$ are some constants depending on $\mu$.
\end{lemma}
\begin{proof}
 The lemma follows from \eqref{BoundOfJ} and \eqref{intgeral transform-2}.
\end{proof}

\subsection{Hecke--Maass cusp forms for $\mathrm{GL}_3$}

Let $\pi$ be a Hecke--Maass cusp form of type $\nu=(\nu_1,\nu_2)$ for $\rm SL_3(\mathbb{Z})$,
which has a Fourier--Whittaker expansion with normalized Fourier coefficients
$\lambda_{\pi}(n_1,n_2)$ (see \cite{Gol}). From the Kim--Sarnak bounds \cite{K}, we have
\bea\label{GL3-Fourier}
\lambda_{\pi}(n_1,n_2)\ll |n_1n_2|^{5/14+\varepsilon}
\eea
for any $\varepsilon>0$.

By Rankin--Selberg theory, we have
\bea\label{GL3-Rankin--Selberg}
\mathop{\sum\sum}_{n_1^2n_2\leq X} \left|\lambda_{\pi}(n_1,n_2)\right|^2=
c_{\pi}X+O\big(X^{4/5}\big),
\eea
where the error term follows from Landau's lemma (see \cite[Lemma 1]{Murty}).

Denote the Langlands parameters of $\pi$ by
\bna
\mu_1=-\nu_1-2\nu_2+1, \qquad \mu_2=-\nu_1+\nu_2,\qquad  \mu_3=2\nu_1+\nu_2-1.
\ena
For $\ell=0,1$ we define
\bna
\gamma_\ell (s)=\frac{1}{2\pi^{3(s+1/2)}}\prod_{j=1}^3
\frac{\Gamma\left((1+s+\mu_j+\ell)/2\right)}
{\Gamma\left((-s-\mu_j+\ell)/2\right)}
\ena
and set $\gamma_{\pm}(s)=\gamma_0(s)\mp i \gamma_1(s)$. Then by Stirling's formula,
for $\sigma\geq -1/2$,
\bea\label{A bound}
\gamma_{\pm}(\sigma+i\tau)\ll_{\pi,\sigma}(1+|\tau|)^{3\left(\sigma+1/2\right)},
\eea
and for $|\tau|\gg N^{\varepsilon}$
(see \cite{Mun1})
\bea\label{Gamma bound-asymptotic}
\gamma_{\pm}\left(-\frac{1}{2}+i\tau\right)
=\left(\frac{|\tau|}{e\pi}\right)^{3i\tau}\Upsilon_{\pm}(\tau),\qquad
\mathrm{where}\qquad
\Upsilon'_{\pm}(\tau)\ll \frac{1}{|\tau|}.
\eea

For $g(x)\in \mathcal{C}_c^\infty(0,\infty)$ we
denote by $\widetilde{g}(s)$
the Mellin transform of $g(x)$.
Let
\bea\label{intgeral transform-3}
\Psi_g^{\pm}\left(x\right)
=\frac{1}{2\pi i}\int_{(\sigma)}x^{-s}
\gamma_{\pm}(s)\widetilde{g}(-s)\mathrm{d}s,
\eea
where $\sigma>\max\limits_{1\leq j\leq 3}\{-1-\mathrm{Re}(\mu_j)\}$.
Then we have the following Voronoi summation formula (see \cite{MS2, GL1}).

\begin{lemma}\label{voronoiGL3}
Let $q\in \mathbb{N}$ and $a\in \mathbb{Z}$ be such
that $(a,q)=1$. Then
\bna
\sum_{n=1}^{\infty}\lambda_{\pi}\left(n,r\right)e\left(\frac{an}{q}\right)
g\left(n\right)= q\sum_{\pm}\sum_{n_{1}|qr}
\sum_{n_{2}=1}^{\infty}\frac{\lambda_{\pi}\left(n_{1},n_{2}\right)}{n_{1}n_{2}}
S\left(r\overline{a},\pm n_{2};\frac{qr}{n_{1}}\right)
\Psi_g^{\pm}\left(\frac{n_{1}^{2}n_{2}}{q^{3}r}\right),
\ena
where $a \overline{a} \equiv 1(\bmod q)$ and $S(m,n;c)$ is the classical Kloosterman sum.
\end{lemma}

For large $x$, just as in the $\mathrm{GL}_2$-Voronoi summation case, we have an asymptotic formula for $\Psi_g^{\pm}(x)$ (see \cite{Li}).

\begin{lemma}\label{voronoiGL3-holomorphic-asymptotic}
Let $g(x)\in \mathcal{C}_c^\infty(X,2X)$.
For any fixed integer $J\geq 1$ and $xX\gg 1$, we have
\bna
\Psi_g^{\pm}\left(x\right)=x\int_0^\infty
g(y)
\sum_{j=1}^{J}\frac{c_{j}e\left(3x^{1/3}y^{1/3}\right)
+d_{j}e\left(-3x^{1/3}y^{1/3}\right)}{\left(xy\right)^{j/3}}\mathrm{d}y
+O\left((xX)^{-J/3+2/3}\right),
\ena
where $c_j$ and $d_j$ are constants depending only on $\mu_i$, $i=1,2,3$, and $\pm$.
\end{lemma}

\begin{remark}\label{decay-of-largeX}
For $x\gg N^{\varepsilon}$, we can choose $J$ sufficiently large so that
the contribution from the $O$-terms in Lemmas \ref{voronoiGL2-holomorphic-asymptotic},
\ref{voronoiGL2-Maass-asymptotic}
and \ref{voronoiGL3-holomorphic-asymptotic}
is negligible. For the main terms
we only need to analyze the leading term $j=1$, as the analysis of the remaining
lower order terms is the same and their contribution is smaller
compared to that of the leading term.
\end{remark}

\section{Proof of the main theorem}\label{details-of-proof}

In this section, we provide the details of the proof for Theorem \ref{main-theorem}.
We write the sum in \eqref{natural-sum} as
\bna
\mathscr{S}(N)=\sum_{r\ll N^{1/2}}\mathscr{S}_r\left(\frac{N}{r^2}\right),
\ena
where
\bea\label{aim-sum}
\mathscr{S}_r(N_0):=\sum_{n=1}^{\infty}\lambda_{\pi}(n,r)
\lambda_f(n)e\left(t\varphi\left(\frac{n}{N_0}\right)\right)V\left(\frac{n}{N_0}\right).
\eea
%If $r\gg N^{\eta}$, then
%\bna
%\begin{split}
%\sum_{N^{\eta}\ll r\ll N^{1/2}}\mathscr{S}_r\left(\frac{N}{r^2}\right)\ll& \sum_{r\gg N^{\eta}}\bigg|\sum_{n\asymp \frac{N}{r^2}}\lambda_{\pi}(n,r)
%\lambda_f(n)e\left(t\varphi\left(\frac{n}{N/r^2}\right)\right)V\left(\frac{n}{N/r^2}\right)\bigg|\\
%\ll& \sum_{r\gg N^{\eta}}\left(\sum_{n\asymp \frac{N}{r^2}}|\lambda_{\pi}(n,r)|^2\right)^{1/2}\left(\sum_{n\asymp \frac{N}{r^2}}|\lambda_f(n)|^2\right)^{1/2}\\
%\ll& N^{1-\eta}
%\end{split}
%\ena
%, whichcan be dominated by the bound in Theorem 1.1 with an appropriate choice of $\eta$
If $r\gg t^{\gamma}$, then
\bea\label{error}
\begin{split}
\sum_{t^{\gamma}\ll r\ll N^{1/2}}\mathscr{S}_r\left(\frac{N}{r^2}\right)
\ll& \bigg(\sum_{t^{\gamma}\ll r\ll N^{1/2}}\sum_{r^2n\asymp  N}|\lambda_{\pi}(n,r)|^2\bigg)^{1/2}
\left(\sum_{t^{\gamma}\ll r\ll N^{1/2}}\sum_{n\asymp \frac{N}{r^2}}|\lambda_f(n)|^2\right)^{1/2}\\
\ll&\, t^{-\gamma/2}N^{1/2}\bigg(\sum_{t^{\gamma}\ll r\ll N^{1/2}}
\sum_{n_1|r^{\infty}}|\lambda_{\pi}(n_1,r)|^2
\sum_{n_2\asymp N/(n_1r^2)}|\lambda_{\pi}(n_2,1)|^2\bigg)^{1/2}\\
\ll&\, t^{-\gamma/2}N^{1/2}\bigg(\sum_{t^{\gamma}\ll r\ll N^{1/2}}
\sum_{n_1|r^{\infty}}(n_1r)^{5/7} \cdot \frac{N}{n_1r^2}\bigg)^{1/2}\\
\ll&\, t^{-9\gamma/14}N^{1+\varepsilon}
\end{split}
\eea
upon applying Cauchy--Schwarz, \eqref{GL3-Fourier} and the Rankin--Selberg estimates \eqref{GL2: Rankin Selberg} and \eqref{GL3-Rankin--Selberg}.

Therefore in the following we restrict the $r$-variable for the sum $\mathscr{S}(N)$ in \eqref{natural-sum} to the range
$1\leq r\leq t^{\gamma}$, at the cost of an error term $O(t^{-9\gamma/14}N^{1+\varepsilon})$.
For $1\leq r\leq t^{\gamma}$ fixed, we will consider and give
a nontrivial bound for the sum $\mathscr{S}_r(N_0)$, defined in \eqref{aim-sum}, with $N_0=Nr^{-2}$.
Without loss of generality, we further assume that the function $\varphi$ in \eqref{aim-sum} satisfies
\bea\label{first-derivative-varphi}
\varphi'(x)<0.
\eea
(The case $\varphi'(x)>0$ can be analyzed analogously.)

\subsection{Applying DFI's circle method}

Define $\delta: \mathbb{Z}\rightarrow \{0,1\}$ with
$\delta(0)=1$ and $\delta(n)=0$ for $n\neq 0$.
As in \cite{Mun6}, we will use a version of the circle method by
Duke, Friedlander and Iwaniec (see \cite[Chapter 20]{IK})
which states that for any $n\in \mathbb{Z}$ and $Q\in \mathbb{R}^+$, we have
\bea\label{DFI's}
\delta(n)=\frac{1}{Q}\sum_{1\leq q\leq Q} \;\frac{1}{q}\;
\sideset{}{^\star}\sum_{a\bmod{q}}e\left(\frac{na}{q}\right)
\int_\mathbb{R}g(q,\zeta) e\left(\frac{n\zeta}{qQ}\right)\mathrm{d}\zeta
\eea
where the $\star$ on the sum indicates
that the sum over $a$ is restricted to $(a,q)=1$.
The function $g$ has the following properties
(see (20.158) and (20.159) of \cite{IK}
\footnote{After correcting a typo in eq. (20.158) there.} and  \cite[Lemma 15]{Huang})
\bea\label{g-h}
g(q,\zeta)\ll |\zeta|^{-A}, \;\;\;\;\;\; g(q,\zeta) =1+h(q,\zeta)\;\;\text{with}\;\;h(q,\zeta)=
O\left(\frac{Q}{q}\left(\frac{q}{Q}+|\zeta|\right)^A\right)
\eea
for any $A>1$ and
\bea\label{g rapid decay}
\frac{\partial^j}{\partial \zeta^j}g(q,\zeta)\ll
|\zeta|^{-j}\min\left(|\zeta|^{-1},\frac{Q}{q}\right)\log Q, \quad j\geq 1.
\eea
In particular the first property in \eqref{g-h} implies that
the effective range of the integration in
\eqref{DFI's} is $[-N_0^\varepsilon, N_0^\varepsilon]$, and
the second property in \eqref{g-h} implies that for
$q\leq Q^{1-\varepsilon}$ and $|\zeta|\ll Q^{-\varepsilon}$,
\bna
g(q,\zeta) =1+O_A\left(Q^{-A}\right)
\ena
for any $A>1$.
 Moreover,
\eqref{g rapid decay} implies that
for $q\geq Q^{1-\varepsilon}$ or $q\leq Q^{1-\varepsilon}$ and
$Q^{-\varepsilon}\ll \zeta\ll Q^{\varepsilon}$, we have
\bna
\frac{\partial^j}{\partial \zeta^j}g(q,\zeta)\ll Q^{\varepsilon}|\zeta|^{-j}, \qquad j\geq 1.
\ena
Therefore, in all cases, we have
\bea\label{rapid decay g}
\frac{\partial^j}{\partial \zeta^j}g(q,\zeta)\ll Q^{\varepsilon}|\zeta|^{-j}, \qquad j\geq 1.
\eea

We write \eqref{aim-sum} as
\bna
\mathscr{S}_r(N_0)=\sum_{n=1}^{\infty}\lambda_{\pi}(n,r)
U\left(\frac{n}{N_0}\right)\sum_{m=1}^{\infty}
\lambda_f(m)e\left(t\varphi\left(\frac{m}{N_0}\right)
\right)V\left(\frac{m}{N_0}\right)\delta(m-n),
\ena
where $U(x)\in \mathcal{C}_c^{\infty}(1/2,5/2)$ is a smooth function satisfying $U(x)=1$
for $x\in [1,2]$ and $U^{(j)}(x)\ll_j 1$. Plugging the identity \eqref{DFI's} for
$\delta(m-n)$ in and
exchanging the order of integration and summations,
we get
\bna
\mathscr{S}_r(N_0)&=&\frac{1}{Q}
\int_{\mathbb{R}} \sum_{1\leq q\leq Q}\frac{g(q,\zeta)}{q}\;
\sideset{}{^\star}\sum_{a\bmod{q}}
\left\{\sum_{n=1}^{\infty}\lambda_{\pi}(n,r)
e\left(-\frac{na}{q}\right)U\left(\frac{n}{N_0}\right)
e\left(-\frac{n\zeta}{qQ}\right)\right\}\nonumber\\
&& \left\{\sum_{m=1}^{\infty}\lambda_{f}(m)e\left(\frac{ma}{q}\right)
V\left(\frac{m}{N_0}\right)
e\left(t\varphi\left(\frac{m}{N_0}\right)+\frac{m\zeta}{qQ}\right)\right\}\mathrm{d}\zeta.
\ena
By the first property in \eqref{g-h}, we write
\bna
\mathscr{S}_r(N_0)&=&\frac{1}{Q}
\int_{\mathbb{R}}W\left(\frac{\zeta}{N_0^{\varepsilon}}\right) \sum_{1\leq q\leq Q}\frac{g(q,\zeta)}{q}\;
\sideset{}{^\star}\sum_{a\bmod{q}}
\left\{\sum_{n=1}^{\infty}\lambda_{\pi}(n,r)
e\left(-\frac{na}{q}\right)U\left(\frac{n}{N_0}\right)
e\left(-\frac{n\zeta}{qQ}\right)\right\}\nonumber\\
&& \left\{\sum_{m=1}^{\infty}\lambda_{f}(m)e\left(\frac{ma}{q}\right)
V\left(\frac{m}{N_0}\right)
e\left(t\varphi\left(\frac{m}{N_0}\right)+\frac{m\zeta}{qQ}\right)\right\}\mathrm{d}\zeta
+O(N_0^{-A}),
\ena
where $W$ is a nonnegative smooth function supported on $[-2,2]$,
with $W(x)=1$ for $x\in [-1,1]$ and satisfying $W^{(j)}(x)\ll_j 1$.

Next we break the $q$-sum $\sum_{1\leq q\leq Q}$ into dyadic segments $q\sim C$ with $1\ll C\ll Q$ and write
\bea\label{C range}
\mathscr{S}_r(N_0)=\sum_{1\ll C\ll Q\atop \text{dyadic}}\mathscr{S}_r(N_0,C)+O(N_0^{-A}),
\eea
with
\bea\label{beforeVoronoi}
\mathscr{S}_r(N_0,C)&=&\frac{1}{Q}
\int_{\mathbb{R}}W\left(\frac{\zeta}{N_0^{\varepsilon}}\right) \sum_{q\sim C}\frac{g(q,\zeta)}{q}\;
\sideset{}{^\star}\sum_{a\bmod{q}}
\left\{\sum_{n=1}^{\infty}\lambda_{\pi}(n,r)
e\left(-\frac{na}{q}\right)U\left(\frac{n}{N_0}\right)
e\left(-\frac{n\zeta}{qQ}\right)\right\}\nonumber\\
&& \left\{\sum_{m=1}^{\infty}\lambda_{f}(m)e\left(\frac{ma}{q}\right)
V\left(\frac{m}{N_0}\right)
e\left(t\varphi\left(\frac{m}{N_0}\right)+\frac{m\zeta}{qQ}\right)\right\}\mathrm{d}\zeta.
\eea
We now proceed to estimate $\mathscr{S}_r(N_0,C)$, for $1\ll C\ll Q$.

\subsection{Dualizing the summations}
In what follows, we dualize the $n$-and $m$-sums in \eqref{beforeVoronoi} using
Voronoi summation formulas.

\subsubsection{Applying $\mathrm{GL}_2$ Voronoi summation}
Depending on whether $f$ is holomorphic or Maass, we apply
Lemma \ref{voronoiGL2-holomorphic}
or Lemma \ref{voronoiGL2-Maass} respectively with $h(y)=V(y)
e\left(t\varphi(y)+\zeta N_0y/qQ\right)$, to transform the $m$-sum in
\eqref{beforeVoronoi} into
\bea\label{after GL2 Voronoi}
\frac{N_0}{q}\sum_{\pm}\sum_{m=1}^{\infty}\lambda_{f}(m)e\left(\mp \frac{m\overline{a}}{q}\right)
\Phi_h^{\pm}\left(\frac{mN_0}{q^2}\right),
\eea
where if $f$ is holomorphic, $\Phi_h^+(x)=\Phi_h(x)$ with $\Phi_h(x)$
given by \eqref{intgeral transform-1} and $\Phi_h^-(x)=0$,
while for $f$ a Hecke--Maass cusp form, $\Phi_h^{\pm}(x)$ is given by \eqref{intgeral transform-2}.

As is typical in applying the delta symbol method, we assume that
\bea\label{assumption 1}
Q<N_0^{1/2-\varepsilon}.
\eea
Then we have $mN_0/q^2\gg N_0^{\varepsilon}$. In particular, by \eqref{The $-$ case}, the contribution from
$\Phi_h^{-}\left(mN_0/q^2\right)$ is negligible.
Next we apply asymptotic expansions for the
Bessel functions in the Hankel transform $\Phi_h^{+}\left(mN_0/q^2\right)$,
which is the content of Lemma \ref{voronoiGL2-holomorphic-asymptotic},
Lemma \ref{voronoiGL2-Maass-asymptotic}
and Remark \ref{decay-of-largeX}.
Then, up to a negligible error term, the sum \eqref{after GL2 Voronoi} is asymptotically equal to
\bna
 \frac{N_0^{3/4}}{q^{1/2}}\sum_{\pm}\sum_{m=1}^{\infty}\frac{\lambda_f(m)}{m^{1/4}}
e\left(-\frac{m\overline{a}}{q}\right)\int_0^\infty V(y)y^{-1/4}
e\left(t\varphi(y)+\frac{\zeta N_0y}{qQ}\pm \frac{2\sqrt{mN_0y}}{q}\right)\mathrm{d}y,
 \ena
and accordingly $\mathscr{S}_r(N_0,C)$ is asymptotically equal to (up to a negligible error)
 \bea\label{GL2}
&&\sum_{\pm}\frac{N_0^{3/4}}{Q}\int_{\mathbb{R}}W\left(\frac{\zeta}{N_0^{\varepsilon}}\right)
 \sum_{q\sim C}\frac{g(q,\zeta)}{q^{3/2}}\;
\sideset{}{^\star}\sum_{a\bmod{q}}
\sum_{m=1}^{\infty}\frac{\lambda_f(m)}{m^{1/4}}e\left(-\frac{m\overline{a}}{q}\right)
\Phi^{\pm}\left(m,q,\zeta\right)\nonumber\\&&\qquad\qquad\qquad\qquad
\times\sum_{n=1}^{\infty}\lambda_{\pi}(n,r)
e\left(-\frac{na}{q}\right)U\left(\frac{n}{N_0}\right)
e\left(-\frac{n\zeta}{qQ}\right)\mathrm{d}\zeta,
 \eea
where
 \bea\label{Phi definition}
\Phi^{\pm}\left(m,q,\zeta\right)=\int_0^\infty V(y)y^{-1/4}
e\left(t\varphi(y)+\frac{\zeta N_0y}{qQ}\pm \frac{2\sqrt{mN_0y}}{q}\right)\mathrm{d}y.
\eea
Notice that for $\triangle<t^{1-\varepsilon}$, defined in \eqref{derivative-of-V}, by Lemma \ref{lem: upper bound},
 the integral $\Phi^{\pm}\left(m,q,\zeta\right)$ is negligibly small if
 $\sqrt{mN_0}/q\gg N_0^{\varepsilon} \max\left\{t,N_0/qQ\right\}$.
Recall $q\sim C$. Thus we only need to consider those
``smaller" $m$'s, that is, in the range $1\leq m\ll N_0^{\varepsilon}\max\{C^2t^2/N_0,N_0/Q^2\}$.
%Making a dyadic subdivision to the sum over $m$,
%we have
%\bna
%\frac{N^{3/4}}{q^{1/2}}
%\sum_{1\ll M\ll N^{\varepsilon}\max\left\{C^2t^2/N,N/Q^2\right\}\atop \text{dyadic}}
%\sum_{m\sim M}\frac{\lambda_f(m)}{m^{1/4}}
%\Phi^{+}\left(m,q,\zeta\right).
%\ena
%Plugging the above expression into \eqref{beforeVoronoi},
Making a smooth partition of unity into dyadic segments to the sum over $m$,
we arrive at
\bna
\mathscr{S}_r(N_0,C)\ll \sum_{\pm}\sum_{1\ll M\ll N_0^{\varepsilon}\max\left\{C^2t^2/N_0,N_0/Q^2\right\}\atop \text{dyadic}}
|\mathscr{T}_r^{\pm}(N_0,C,M)|,
\ena
where
\bea\label{after GL2 Voronoi-2}
\mathscr{T}_r^{\pm}(N_0,C,M)&=&\frac{N_0^{3/4}}{Q}\int_{\mathbb{R}}
W\left(\frac{\zeta}{N_0^{\varepsilon}}\right)
 \sum_{q\sim C}\frac{g(q,\zeta)}{q^{3/2}}\;
\sideset{}{^\star}\sum_{a\bmod{q}}
\sum_{m=1}^{\infty}\frac{\lambda_f(m)}{m^{1/4}}e\left(-\frac{m\overline{a}}{q}\right)
w\left(\frac{m}{M}\right)
\nonumber\\&&\times\Phi^{\pm}\left(m,q,\zeta\right)
\sum_{n=1}^{\infty}\lambda_{\pi}(n,r)
e\left(-\frac{na}{q}\right)U\left(\frac{n}{N_0}\right)
e\left(-\frac{n\zeta}{qQ}\right)\mathrm{d}\zeta.
\eea
Here $w(x)\in \mathcal{C}_c^{\infty}(1,2)$ satisfies $w^{(j)}(x)\ll_j 1$ for $j\geq 0$.

Before proceeding to applying $\rm GL_3$ Voronoi formula to the sum over $n$, we give some preliminary estimates for
$\Phi^{\pm}\left(m,q,\zeta\right)$ as follows.
\begin{lemma}\label{phi:individual estimates}
 Let $q\sim Q$ and $m\sim M$. Assume $\zeta$ is such that $|\zeta|N_0/CQ\ll N_0^{-\varepsilon} \max\{t,\sqrt{N_0M}/C\}$.  Then
$ \Phi^{-}\left(m,q,\zeta\right)$ is negligibly small.

Under the further assumption
\bea\label{varphi condition}
(\varphi(y^3))''>0,
\eea
we have
\bna
\Phi^{+}\left(m,q,\zeta\right)\ll \max\{t,\sqrt{N_0M}/C\}^{-1/2}.
\ena

%(2) For $|\zeta|N_0/CQ\gg N_0^{-\varepsilon} \max\{t,\sqrt{N_0M}/C\}$, we have
%\bna
%\Phi^{-}\left(m,q,\zeta\right)\ll \max\{t,\sqrt{N_0M}/C\}^{-1/2}.
%\ena

%(2) Let $\omega(x)\in C_c^{\infty}(0,\infty)$. We have
%\bea\label{integral111}
%\int_{\mathbb{R}}
%\omega\left(\frac{|\zeta|}{\Xi}\right)
%\left|\Phi^{\pm}\left(m,q,\zeta\right)\right|^2\mathrm{d}\zeta
%\ll N_0^{\varepsilon} \frac{CQ}{N_0}.
%\eea
%Moreover, for $\Xi\ll N_0^{\varepsilon}$, we have
%\bea\label{integral1111}
%\int_{\mathbb{R}}
%\omega\left(\frac{|\zeta|}{\Xi}\right)
%\left|\Phi^{\pm}\left(m,q,\zeta\right)\right|^2
%\mathrm{d}\zeta\ll N_0^{\varepsilon}\max\{t,\sqrt{N_0M}/C\}^{-1}.
%\eea

\end{lemma}

\begin{proof}

Recall $\Phi^{\pm}\left(m,q,\zeta\right)$ is given by the integral in \eqref{Phi definition}. The first derivative of the phase function in the integrand is
 \bna
 \rho'_{\pm}(y)=t\varphi'(y)+\frac{\zeta N_0}{qQ}\pm\frac{1}{q}\sqrt{N_0m/y}.
 \ena
 Since we have assumed in \eqref{first-derivative-varphi} that $\varphi'(y)<0$,
 for the ``$-$" case,
 we have
 $|\rho'_{-}(y)|\gg \max\{t,\sqrt{N_0M}/C\}$ provided
 $|\zeta|N_0/CQ\ll N_0^{-\varepsilon} \max\{t,\sqrt{N_0M}/C\}$.
 Then, by applying integration by parts
 repeatedly, one gets
 \bna
 \Phi^{-}\left(m,q,\zeta\right)\ll_A N_0^{-A}
 \ena for any $A\geq 1$. Thus the first statement of (1) follows.
 For the ``$+$" case, we change variable $y\rightarrow y^3$ to get
\bna
\Phi^{+}\left(m,q,\zeta\right)=3\int_0^\infty V(y^3)y^{5/4}
e\left(t\varphi(y^3)+\frac{\zeta N_0y^3}{qQ}+ \frac{2\sqrt{N_0my^3}}{q}\right)\mathrm{d}y.
\ena
Under the additional assumption \eqref{varphi condition}, by applying the second derivative test in
 Lemma \ref{lem: 2st derivative test, dim 1}, we have
 \bna
 \Phi^{+}\left(m,q,\zeta\right)\ll \max\{t,\sqrt{N_0M}/C\}^{-1/2}.
 \ena
Here we also made use of the assumption that the test function $V$ has bounded variation $\text{Var}(V)$. This proves the second statement of the lemma.

%For (2), we insert the definition \eqref{Phi definition} to get,
%\bna
%&&\int_{\mathbb{R}}
%\omega\left(\frac{|\zeta|}{\Xi}\right)
%\left|\Phi^{\pm}\left(m,q,\zeta\right)\right|^2\mathrm{d}\zeta\\
%&=&\int_0^{\infty}\int_0^{\infty}V(y_1)V(y_2)y_1^{-1/4}y_2^{-1/4}
%e\left(t\varphi(y_1)-t\varphi(y_2)\pm \frac{2\sqrt{mN_0y_1}}{q}\mp
%\frac{2\sqrt{mN_0y_2}}{q} \right)\\&&\qquad\qquad \qquad\qquad \qquad\times
%\left(\int_{\mathbb{R}}\omega\left(\frac{|\zeta|}{\Xi}\right)
%e\left(\frac{\zeta N_0}{qQ}(y_1-y_2)\right)\mathrm{d}\zeta\right)\mathrm{d}y_1\mathrm{d}y_2.
%\ena
%Recall $q\sim C$. Considering the integral over $\zeta$ and using integration by parts, one finds that the above integral is negligibly small
%unless $|y_1-y_2|<CQ/\Xi N_0^{1-\varepsilon}$. Changing variable $\zeta/\Xi\rightarrow \zeta$, we have
%\bna
%\int_{\mathbb{R}}
%\omega\left(\frac{|\zeta|}{\Xi}\right)
%\left|\Phi^{\pm}\left(m,q,\zeta\right)\right|^2\mathrm{d}\zeta
%&\ll& \Xi \mathop{\int\int}_{|y_1-y_2|<CQ/\Xi N_0^{1-\varepsilon}}
%V(y_1)V(y_2)y_1^{-1/4}y_2^{-1/4}\mathrm{d}y_1\mathrm{d}y_2+N_0^{-A}\nonumber\\
%&\ll& N_0^{\varepsilon} \frac{CQ}{N_0}.
%\ena
%This proves \eqref{integral111}.
%
%To prove \eqref{integral1111}, we note that for $N_0/(CQ)\leq N_0^{-\varepsilon} \max\{t,\sqrt{N_0M}/C\}$,  \eqref{integral1111} follows directly from (1).
%For $N_0/(CQ)\geq N_0^{-\varepsilon} \max\{t,\sqrt{N_0M}/C\}$,
%\eqref{integral1111} follows \eqref{integral111}.
%This finishes the proof of the lemma.
\end{proof}

\subsubsection{Applying $\mathrm{GL}_3$ Voronoi summation}
Applying Lemma \ref{voronoiGL3} with
$g(y)=U\left(y/N_0\right)e\left(-\zeta y/qQ\right)$, we transform the
$n$-sum in \eqref{after GL2 Voronoi-2} into
\bea\label{$n$-sum after GL3 Voronoi}
q\sum_{\pm}\sum_{n_1|qr}\sum_{n_2=1}^{\infty}
\frac{\lambda_{\pi}\left(n_1,n_2\right)}{n_1n_2}
S\left(-r\overline{a},\pm n_2;\frac{qr}{n_1}\right)\Psi_r^{\pm}\left(n_1^2n_2,q,\zeta\right),
\eea
where by \eqref{intgeral transform-3},
\bea\label{Psi-integral}
\Psi_r^{\pm}\left(n_1^2n_2,q,\zeta\right)=
\frac{1}{2\pi i}\int_{(\sigma)}\left(\frac{N_0n_1^2n_2}{q^3r}\right)^{-s}
\gamma_{\pm}(s)U^{\dag}\left(\frac{N_0\zeta}{qQ},-s\right)\mathrm{d}s
\eea
with
\bna
U^{\dag}\left(\xi,s\right)=\int_0^\infty U(y)e\left(-\xi y\right)y^{s-1}\mathrm{d}y
\ena
and $\gamma_{\pm}(s)$ being ratio of gamma factors defined right above \eqref{A bound}.

By  \eqref{$U$ bound}, we have
\bna
\begin{split}
U^{\dag}\left(\frac{N_0\zeta}{qQ},-s\right)\ll_j&
\min \left\{ \left(\frac{1+|\mathrm{Im}(s)|}{N_0|\zeta|/qQ} \right)^j ,
\left(\frac{1+N_0|\zeta|/qQ}{|\mathrm{Im}(s)|} \right)^j \right\}\\
\ll_j& \min \left\{1, \left(\frac{N_0^{1+\varepsilon}}{qQ|\mathrm{Im}(s)|} \right)^j \right\}.
\end{split}
\ena
This together with \eqref{A bound} implies
\bna
\Psi_r^{\pm}\left(n_1^2n_2,q,\zeta\right)&\ll_j& \left(\frac{N_0n_1^2n_2}{q^3r}\right)^{-\sigma}
\int_{\mathbb{R}}
(1+|\tau|)^{3(\sigma+1/2)}\min\left\{1,\left(\frac{N_0^{1+\varepsilon}}{qQ|\tau|}\right)^j\right\}\mathrm{d}\tau\\
&\ll&\left(\frac{N_0^{1+\varepsilon}}{CQ}\right)^{5/2}\left(\frac{Q^3n_1^2n_2}{N_0^{2+\varepsilon}r}\right)^{-\sigma},
\ena
upon noting that $q\sim C$ and $|\zeta|\leq N_0^{\varepsilon}$.
%upon choosing $j=3\sigma+5/2$ and noting that $q\sim C$ and $|\zeta|\leq N_0^{\varepsilon}$.

By taking $j=3\sigma+5/2$ and $\sigma$ sufficiently large, one sees that
the $n_1,n_2$-sums in \eqref{$n$-sum after GL3 Voronoi} can be truncated at
$n_1^{2}n_2\leq N_0^{2+\varepsilon}r/Q^3$, at the cost of a negligible error.
Plugging \eqref{$n$-sum after GL3 Voronoi} into \eqref{after GL2 Voronoi-2},
exchanging the orders of summations and integrations, and
further breaking the $(n_1,n_2)$-sums into dyadic segments $n_1^2n_2\sim N_1$
with $1\ll N_1\ll N_0^{2+\varepsilon}r/Q^3$, we obtain
\bea\label{M-N1-range}
\mathscr{S}_r(N_0,C)\ll \sum_{\pm, \pm}\sum_{ 1\ll M\ll N_0^{\varepsilon}\max\left\{C^2t^2/N_0,N_0/Q^2\right\}\atop \text{dyadic}}\sum_{1\ll N_1\ll
N_0^{2+\varepsilon}r/Q^3\atop \text{dyadic}}
\left|\mathscr{T}_r^{\pm,\pm}(N_0,C,M,N_1)\right|,
\eea
where
\bea\label{after GL3 Voronoi}
\mathscr{T}_r^{\pm,\pm}(N_0,C,M,N_1)&=&\frac{N_0^{3/4}}{Q}\sum_{q\sim C}
\frac{1}{q^{1/2}}
\sum_{n_1^2n_2\sim N_1 \atop n_1|qr}
\frac{\lambda_{\pi}\left(n_1,n_2\right)}{n_1n_2}
\sum_{m=1}^{\infty}\frac{\lambda_f(m)}{m^{1/4}}
w\left(\frac{m}{M}\right)
\nonumber\\&&\qquad\qquad\qquad \quad \times
\mathfrak{C}_r(n_1,\pm n_2,m,q)
\mathfrak{J}_r^{\pm,\pm}(m,n_1^2n_2,q),
\eea
with
\bea\label{character sum}
\mathfrak{C}_r(n_1,n_2,m,q)
=\sideset{}{^\star}\sum_{a\bmod q}e\left(\frac{ma}{q}\right)
S\left(ra, n_2;\frac{qr}{n_1}\right)
\eea
and
\bea\label{J-middle}
\mathfrak{J}_r^{\pm,\pm}(m,n_1^2n_2,q)=
\int_{\mathbb{R}}
W\left(\frac{\zeta}{N_0^{\varepsilon}}\right)
g(q,\zeta)
\Phi^{\pm}\left(m,q,\zeta\right)
\Psi_r^{\pm}\left(n_1^2n_2,q,\zeta\right)\mathrm{d}\zeta.
\eea
To study $\mathscr{S}_r(N_0,C)$ in \eqref{beforeVoronoi}, it is therefore sufficient to treat
$\mathscr{T}_r^{\pm,\pm}(N_0,C,M,N_1)$ in \eqref{after GL3 Voronoi}.

In \eqref{J-middle} the behaviours of the integral $\Psi_r^{\pm}\left(n_1^2n_2,q,\zeta\right)$,
which is given explicitly in \eqref{Psi-integral}, depending on whether
$N_0n_1^2n_2/q^3r\ll N_0^{\varepsilon}$
or $N_0n_1^2n_2/q^3r\gg N_0^{\varepsilon}$, that is, whether $N_1\ll C^3r/N_0^{1-\varepsilon}$ or
$N_1\gg C^3r/N_0^{1-\varepsilon}$ (since $n_1^2n_2\sim N_1$ and $q\sim C$), are slightly different.
In what follows, we treat these two cases separately.

%\begin{remark}
%By Weil's bound for Kloosterman sum, we have
%\bna
%\mathfrak{C}_r(n_1,n_2,m,q)\ll q(ra,n_2,qrn_1^{-1})^{1/2}(qrn_1^{-1})^{1/2}\tau(qr/n_1)\ll
%r^{1/2+\varepsilon}q^{3/2+\varepsilon}n_1^{-1/2}(n_2,qrn_1^{-1})^{1/2}
%\ena
%for any $\varepsilon>0$, where $\tau(n)$ is the divisor function.
%Since the Ramanujan sum $S(n,0;q)=\sum_{d|(n,q)}d\mu(q/d)$, the character sum in
%\eqref{character sum} is
%\bea\label{character sum-1}
%\mathfrak{C}_r(n_1,n_2,m,q)
%=\sum_{d | q }d\mu\left(\frac{q}{d}\right)
%\sideset{}{^{\star}}\sum_{\alpha\bmod qr/n_1\atop n_1\alpha \equiv -m \bmod d}
%e\left(\frac{ n_2\overline{\alpha}}{qr/n_1}\right).
%\eea
%Note that the number of
%$\alpha$ is $(n_1,d)qr/(n_1d)$.
%Thus
%\bna
%\mathfrak{C}_r(n_1,n_2,m,q) \ll rq(n_1,q)\tau(q)/n_1.
%\ena
%Thus
%\bea\label{character sum-estimate}
%\mathfrak{C}_r(n_1,n_2,m,q) \ll (rq)^{\varepsilon}\min
%\left\{rq(n_1,q)n_1^{-1},r^{1/2}q^{3/2}n_1^{-1/2}(n_2,qrn_1^{-1})^{1/2}\right\}.
%\eea
%\end{remark}

\subsubsection{The contribution from smaller $N_1$}\label{contribution-smaller}

In this subsection, we consider the contribution from $1\ll N_1\ll C^3r/N_0^{1-\varepsilon}$
to $\mathscr{T}_r^{\pm}(N_0,C,M)$ in \eqref{M-N1-range}, and in Section \ref{contribution-generic} we will consider the case where $N_1\gg C^3r/N_0^{1-\varepsilon}$. We remind the reader that the $\zeta$ in \eqref{J-middle} satisfies $|\zeta|<N_0^{\varepsilon}$.

 \textit{Case 1}. For $\zeta$ with $|\zeta| N_0 /CQ\geq N_0^{\varepsilon}$, we analyze the integral $\Psi_r^{\pm}$
in \eqref{Psi-integral} in some detail. If $|\mathrm{Im}(s)|\ll N_0^{\varepsilon}$ or
$|\mathrm{Im}(s)|\gg N_0^{1+\varepsilon}/CQ$, then by \eqref{$U$ bound},
\bna
\begin{split}
U^{\dag}\left(\frac{\zeta N_0}{qQ},-s\right)\ll_j&
\min \left\{ \left(\frac{1+|\mathrm{Im}(s)|}{|\zeta| N_0/CQ} \right)^j ,
\left(\frac{1+|\zeta| N_0/CQ}{|\mathrm{Im}(s)|} \right)^j \right\}\\
\ll_A&\, N_0^{-A}
\end{split}
\ena
upon recalling $|\zeta|<N_0^{\varepsilon}$ and taking $j$ sufficiently large.
This implies the contribution from $|\mathrm{Im}(s)|\ll N_0^{\varepsilon}$ and
$|\mathrm{Im}(s)|\gg N_0^{1+\varepsilon}/CQ$ to $\Psi_r^{\pm}$ is negligible. Shifting the contour of integration to $\sigma=-1/2$ and using \eqref{Gamma bound-asymptotic},
we rewrite the integral $\Psi_r^{\pm}$ in \eqref{Psi-integral} as
\bna
\Psi_r^{\pm}\left(n_1^2n_2,q,\zeta\right)=\sum_{J\in \mathscr{J}}
\int_0^{\infty}U(y)K_{J}^{\pm}(y)
\left(\frac{N_0n_1^2n_2}{q^3ry}\right)^{1/2}
e\left(-\frac{\zeta N_0y}{qQ}\right)\mathrm{d}y+O_A(N_0^{-A}),
\ena
where
\bna
K_{J}^{\pm}(y)=\frac{1}{2\pi}\int_{\mathbb{R}}
\omega_J\left(|\tau|\right)\Upsilon_{\pm}(\tau)
e\left(-\frac{1}{2\pi}\tau\log\frac{N_0n_1^2n_2y}{q^3r}+\frac{3}{2\pi}\tau\log\frac{|\tau|}{e\pi}\right)
\mathrm{d}\tau.
\ena
Here $\mathscr{J}$ is a collection of $O(\log N_0)$ many real numbers in
the interval $[N_0^{\varepsilon}, N_0^{1+\varepsilon}/CQ]$. For each $J\in \mathscr{J}$,
$\omega_J(x)\in \mathcal{C}_c^{\infty}(J,2J)$ satisfies $\omega_J^{(\ell)}(x)\ll_{\ell} 1$
for $\ell\geq 0$ and $\sum_{J\in \mathscr{J}}\omega_J(x)=1$ for
$x\in [N_0^{\varepsilon}, N_0^{1+\varepsilon}/CQ]$.
Recall that $\Upsilon'_{\pm}(\tau)\ll |\tau|^{-1}$ by \eqref{Gamma bound-asymptotic}. Applying Lemma \ref{lemma:exponentialintegral} (1)
with $X=1,Y=J\log J, Z=J$ and $R=J\log J$, we have
$K_{J}^{\pm}(y)\ll_A N_0^{-A}$.  It follows that
\bna
\Psi_r^{\pm}\left(n_1^2n_2,q,\zeta\right)\ll_A N_0^{-A}.
\ena
Thus the contribution from those $\zeta$'s such that $\zeta N_0 /CQ\gg N_0^{\varepsilon}$ to \eqref{M-N1-range} is negligible.

\textit{Case 2}.  For $\zeta$ with $|\zeta| N_0 /CQ< N_0^{\varepsilon}$,
by Lemma \ref{phi:individual estimates} $\Phi^{-}\left(m,q,\zeta\right)$ is negligibly small, and
we only need to consider the sum \eqref{after GL3 Voronoi} involving
$\Phi^{+}\left(m,q,\zeta\right)$.
Denote temporarily
\bea\label{T}
\mathbb{T}&:=&\frac{N_0^{3/4}}{Q}\sum_{q\sim C}
\frac{1}{q^{1/2}}
\sum_{n_1^2n_2\sim N_1 \atop n_1|qr}
\frac{\lambda_{\pi}\left(n_1,n_2\right)}{n_1n_2}
\sum_{m=1}^{\infty}\frac{\lambda_f(m)}{m^{1/4}}
\mathfrak{C}_r(n_1,\pm n_2,m,q)
\nonumber\\&&\qquad\qquad\qquad \quad \times
w\left(\frac{m}{M}\right)
\mathfrak{J}_r^{+,\pm,\flat}(m,n_1^2n_2,q),
\eea
where
\bea\label{small zeta}
\mathfrak{J}_r^{+,\pm,\flat}(m,n_1^2n_2,q)=
\int_{-CQ/N_0^{1-\varepsilon}}^{CQ/N_0^{1-\varepsilon}}
W\left(\frac{\zeta}{N_0^{\varepsilon}}\right)
g(q,\zeta)
\Phi^{+}\left(m,q,\zeta\right)
\Psi_r^{\pm}\left(n_1^2n_2,q,\zeta\right)\mathrm{d}\zeta.
\eea
Here we recall $\Phi^{+}\left(m,q,\zeta\right)$  and
$\Psi_r^{\pm}\left(n_1^2n_2,q,\zeta\right)$ are defined in \eqref{Phi definition}
and \eqref{Psi-integral}, respectively.

For $\Psi_r^{\pm}$ in \eqref{Psi-integral},
again by \eqref{$U$ bound}, if $\tau=\mathrm{Im}(s)\gg N_0^{\varepsilon}$,
then $U^{\dag}\left(N_0\zeta/qQ,-s\right)\ll N_0^{-A}$.
Hence the contribution from such $\tau$ to $\Psi_r^{\pm}$ is negligible. Next
by shifting the contour of integration to $\sigma=-1/2$ and recalling that $\gamma_{\pm}(-1/2+i\tau)\ll 1$ from \eqref{A bound}, we have
\bea\label{for small n1n2}
\Psi_r^{\pm}\left(n_1^2n_2,q,\zeta\right)&=&\frac{1}{2\pi }\int_{|\tau|\ll N_0^{\varepsilon}}
\left(\frac{N_0n_1^2n_2}{q^3r}\right)^{1/2-i\tau}
\gamma_{\pm}\left(-\frac{1}{2}+i\tau\right)U^{\dag}\left(\frac{N_0\zeta}{qQ},\frac{1}{2}-i\tau\right)
\mathrm{d}\tau+O(N_0^{-A})\nonumber\\
&\ll&\int_{|\tau|\ll N_0^{\varepsilon}}\left(\frac{N_0N_1}{C^3r}\right)^{1/2}\cdot 1\cdot 1\,
\mathrm{d}\tau\nonumber\\
&\ll& N_0^{\varepsilon}
\eea
upon recalling $n_1^2n_2\sim N_1\ll C^3r/N_0^{1-\varepsilon}$, which is our assumption in this subsection.

By Fourier inversion, we can write
\bea\label{expression}
w\left(\frac{m}{M}\right)\mathfrak{J}_r^{+,\pm,\flat}(m,n_1^2n_2,q)
=\int_{\mathbb{R}}\widehat{\mathfrak{J}}(x,n_1^2n_2,q)e(mx)\mathrm{d}x,
\eea
where
\bna
\widehat{\mathfrak{J}}(x,n_1^2n_2,q):=\int_{\mathbb{R}}
w\left(\frac{u}{M}\right)\mathfrak{J}_r^{+,\pm,\flat}(u,n_1^2n_2,q)
e(-xu)\mathrm{d}u.
\ena

We claim that the integral $\widehat{\mathfrak{J}}(x,n_1^2n_2,q)$ is negligibly small (i.e., $\ll N_0^{-A}$) if $|x|> N_0^{1/2+\varepsilon}C^{-1}M^{-1/2}$. To see this, plugging \eqref{small zeta} in, we first write
\bna
\widehat{\mathfrak{J}}(x,n_1^2n_2,q)=M\int_{-CQ/N_0^{1-\varepsilon}}^{CQ/N_0^{1-\varepsilon}}
W\left(\frac{\zeta}{N_0^{\varepsilon}}\right)
g(q,\zeta)\left(\int_{\mathbb{R}}
w(u)
\Phi^{+}\left(uM,q,\zeta\right)
e(-xMu)\mathrm{d}u\right)\Psi_r^{\pm}\left(n_1^2n_2,q,\zeta\right)\mathrm{d}\zeta.
\ena
By inserting the definition \eqref{Phi definition} for $\Phi^{+}\left(uM,q,\zeta\right)$ into the inner integral,
\bea\label{intermediate-step}
\begin{split}
\widehat{\mathfrak{J}}(x,n_1^2n_2,q)=&M\int_{-CQ/N_0^{1-\varepsilon}}^{CQ/N_0^{1-\varepsilon}}
W\left(\frac{\zeta}{N_0^{\varepsilon}}\right)
g(q,\zeta)
\int_0^\infty V(y)y^{-1/4}
e\left(t\varphi(y)+\frac{\zeta N_0y}{qQ}\right)\\
&\quad\quad \times\bigg(\int_{\mathbb{R}}
w(u)e\big(2\sqrt{uMN_0y}/q-xMu\big)\mathrm{d}u\bigg)\mathrm{d}y\, \Psi_r^{\pm}\left(n_1^2n_2,q,\zeta\right)\mathrm{d}\zeta.
\end{split}
\eea
By applying integration by parts repeatedly, one sees that the inner most integral satisfies
\bna
\begin{split}
\int_{\mathbb{R}}
w(u)e\big(2\sqrt{uMN_0y}/q-xMu\big)\mathrm{d}u\ll_j& \left(\frac{1}{|x|M}\big(1+\sqrt{MN_0y}/q\big)\right)^j\\
\ll_j& \left(\frac{1}{|x|C}\sqrt{\frac{N_0}{M}}\right)^j
\end{split}
\ena
for any $j\geq 0$, upon noting that $w^{(j)}(u)\ll_j 1$, $q\sim C$, and $y\sim 1$. Finally by bounding the outer integrals over $y$ and $\xi$ trivially, we obtain
\bna
\widehat{\mathfrak{J}}(x,n_1^2n_2,q)\ll M\frac{CQ}{N_0^{1-\varepsilon}}
\left(\frac{1}{|x|C}\sqrt{\frac{N_0}{M}}\right)^j.
\ena
In particular, if $|x|> N_0^{1/2+\varepsilon}C^{-1}M^{-1/2}$, by taking $j$ large enough,
we have $\widehat{\mathfrak{J}}(x,n_1^2n_2,q)\ll_A N_0^{-A}$, justifying the claim.

Therefore we can impose an additional restriction $|x|\leq N_0^{1/2+\varepsilon}C^{-1}M^{-1/2}$
to the integral over $x$ in \eqref{expression}, that is, we can write \eqref{expression} as
\bea\label{expression2}
w\left(\frac{m}{M}\right)\mathfrak{J}_r^{+,\pm,\flat}(m,n_1^2n_2,q)
=\int_{|x|\leq N_0^{1/2+\varepsilon}C^{-1}M^{-1/2}}\widehat{\mathfrak{J}}(x,n_1^2n_2,q)e(mx)\mathrm{d}x+O_A\left(N_0^{-A}\right).
\eea

We now make a change of variable $y\rightarrow y^3$ in \eqref{intermediate-step} to rewrite
\bea\label{another expression}
\widehat{\mathfrak{J}}(x,n_1^2n_2,q)=M\int_{-CQ/N_0^{1-\varepsilon}}^{CQ/N_0^{1-\varepsilon}}
W\left(\frac{\zeta}{N_0^{\varepsilon}}\right)
g(q,\zeta)
\Psi_r^{\pm}\left(n_1^2n_2,q,\zeta\right)R\left(x,q,\zeta\right)\mathrm{d}\zeta,
\eea
where
\bna
R\left(x,q,\zeta\right)=3\int_0^{\infty}\int_0^{\infty}y^{5/4}V(y^3)w(u)
e(G(y,u))\mathrm{d}u\mathrm{d}y
\ena
with
$$
G(y,u)=t\varphi(y^3)+\frac{\zeta N_0y^3}{qQ}+ \frac{2\sqrt{MN_0u}}{q}y^{3/2}-xMu.
$$
Recall that we have assumed in \eqref{varphi condition} that
$(\varphi(y^3))''>0$.
Then calculating partial derivatives, one finds that
\bna
\begin{split}
\frac{\partial^2 G(y,u)}{\partial y^2} =& \,t(\varphi(y^3))''+\frac{6\zeta N_0y}{qQ}+
\frac{3\sqrt{MN_0u}}{2q}y^{-1/2}\gg t,\\
\frac{\partial^2 G(y,u)}{\partial u^2} =&-\frac{\sqrt{MN_0}}{2q}\left(\frac{y}{u}\right)^{3/2}\asymp \sqrt{MN_0}/C,
\end{split}\ena
and
$$\frac{\partial^2 G(y,u)}{\partial y^2}\cdot \frac{\partial^2 G(y,u)}{\partial u^2}
-\left(\frac{\partial^2 G(y,u) }{\partial y\partial u}\right)^2\gg t \sqrt{MN_0}/C.$$
This last inequality holds because
$\frac{\partial^2 G(y,u)}{\partial y^2}\cdot \frac{\partial^2 G(y,u)}{\partial u^2}<0$, due to our additional assumption $(\varphi(y^3))''>0$.
Hence by applying the second derivative test in Lemma \ref{lem: 2nd derivative test, dim 2}
with $\rho_1=t$, $\rho_2=\sqrt{MN_0}/C$ and $\text{Var}=1$, we obtain
\bea\label{two dim}
R\left(x,q,\zeta\right)\ll t^{-1/2}C^{1/2}(MN_0)^{-1/4}.
\eea
Plugging the bounds \eqref{g-h}, \eqref{for small n1n2} and \eqref{two dim} into \eqref{another expression}, we obtain
\bea\label{estimate save}
\widehat{\mathfrak{J}}(x,n_1^2n_2,q)\ll M \frac{C Q}{N_0^{1-\varepsilon}}t^{-1/2}C^{1/2}(MN_0)^{-1/4}.
\eea

In view of \eqref{character sum} and \eqref{expression2}, the sum in \eqref{T} involving the $m$-variable takes the form
\bna
&&\sum_{m\sim M}\frac{\lambda_f(m)}{m^{1/4}}e\left(\frac{ma}{q}\right)
w\left(\frac{m}{M}\right)
\mathfrak{J}_r^{+,\pm,\flat}(m,n_1^2n_2,q)\\
&=&\int_{|x|\leq N_0^{1/2+\varepsilon}C^{-1}M^{-1/2}}\widehat{\mathfrak{J}}(x,n_1^2n_2,q)
\sum_{m\sim M}\frac{\lambda_f(m)}{m^{1/4}}e\left(\left(x+\frac{a}{q}\right)m\right)
\mathrm{d}x+O_A(N_0^{-A}).
\ena
From the Wilton bound \eqref{additive-holo} and \eqref{additive-Maass}, the inner $m$-sum is bounded above by $O(M^{1/4+\varepsilon})$. This together with the upper bound  \eqref{estimate save} infers the following estimate for the above expression
\bna
%\ll M^{1/4}\frac{CQ}{N_0^{1-\varepsilon}t^{1/2}}\left(\frac{MN_0}{C^2}\right)^{1/4}
\ll \frac{N_0^{1/2+\varepsilon}}{CM^{1/2}}\cdot \frac{M C Q}{N_0^{1-\varepsilon}t^{1/2}}\left(\frac{C^2}{MN_0}\right)^{1/4}\cdot M^{1/4}
\ll N_0^{-3/4+\varepsilon}t^{-1/2}M^{1/2}C^{1/2}Q.
\ena
Now we are ready to give our final estimate for $\mathbb{T}$, the contribution from the case where $\zeta$ satisfies $\zeta N_0/CQ\ll N_0^{\varepsilon}$ to $\mathscr{T}_r^{\pm}(N_0,C,M)$ in \eqref{M-N1-range}, as follows. First we recall from \eqref{T}, we have
\bna
\mathbb{T}&\ll&\frac{N_0^{3/4}}{Q}\bigg|\sum_{q\sim C}
\frac{1}{q^{1/2}}
\sum_{n_1^2n_2\sim N_1 \atop n_1|qr}
\frac{\lambda_{\pi}\left(n_1,n_2\right)}{n_1n_2}
\sideset{}{^\star}\sum_{a\bmod q}
S\left(ra, n_2;\frac{qr}{n_1}\right)
\nonumber\\&&\qquad\qquad\qquad \quad \times
\sum_{m\sim M}\frac{\lambda_f(m)}{m^{1/4}}e\left(\frac{ma}{q}\right)
w\left(\frac{m}{M}\right)
\mathfrak{J}_r^{+,\pm,\flat}(m,n_1^2n_2,q)\bigg|.
\ena
By the Weil's bound for Kloosterman sums and the Rankin--Selberg
estimate \eqref{GL3-Rankin--Selberg}, we derive
\bna
\mathbb{T}&\ll&N_0^{\varepsilon}\frac{M^{1/2}C^{1/2}}{t^{1/2}}\sum_{q\sim C}
\frac{1}{q^{1/2}}
\sum_{n_1^2n_2\sim N_1 \atop n_1|qr}
\frac{|\lambda_{\pi}\left(n_1,n_2\right)|}{n_1n_2}\sideset{}{^\star}\sum_{a\bmod q}
\left(ra, n_2,\frac{qr}{n_1}\right)^{1/2}\left(\frac{qr}{n_1}\right)^{1/2}\nonumber\\
&\ll&N_0^{\varepsilon}\frac{r^{1/2}M^{1/2}C^{1/2}Q}{t^{1/2}}
\sum_{q\sim C}\sum_{n_1|qr}n_1^{-3/2}
\sum_{n_1^2n_2\sim N_1}
\frac{|\lambda_{\pi}\left(n_1,n_2\right)|}{n_2}
\left(n_2,\frac{qr}{n_1}\right)^{1/2}\nonumber\\
&\ll&N_0^{\varepsilon}\frac{r^{1/2}M^{1/2}C^{1/2}Q}{t^{1/2}}
\sum_{q\sim C}\sum_{n_1|qr}n_1^{-1/2}
\left(\sum_{n_1^2n_2\sim N_1}
\frac{|\lambda_{\pi}\left(n_1,n_2\right)|^2}{n_1^2n_2}\right)^{1/2}
\left(\sum_{n_1^2n_2\sim N_1}
\frac{\left(n_2,\frac{qr}{n_1}\right)}{n_2}
\right)^{1/2}\nonumber\\
&\ll&N_0^{\varepsilon}\frac{r^{1/2}M^{1/2}C^{3/2}Q}{t^{1/2}}.
\ena
We further assume
\bea\label{assumption: range 0}
Q>(N_0/t)^{1/2}.
\eea
Then for $C\ll Q$, we have $C^2t^2/N_0\ll N_0/Q^2$. Hence for $M$ in \eqref{M-N1-range}, the above expression is
bounded by
\bea\label{small part-1}
N_0^{\varepsilon}r^{1/2}t^{-1/2}Q^{5/2}\left(Q^2t^2/N_0\right)^{1/2}\ll
r^{1/2}N_0^{-1/2+\varepsilon}Q^{7/2}t^{1/2}.
\eea

\begin{remark}\label{second-derivative-0}
If $(\varphi(x^3))''=0$, we may apply the second derivative test in
Lemma \ref{lem: 2nd derivative test, dim 2}
with $\rho_1=\rho_2=\sqrt{MN_0}/C$ to obtain
\bna
R\left(x,q,\zeta\right)\ll C(MN_0)^{-1/2}.
\ena
Then following the same line of proof as above, we would obtain the same estimate
as in \eqref{small part-1}.

\end{remark}

\subsubsection{The generic case: contribution from larger $N_1$}\label{contribution-generic}
Recall in Section \ref{contribution-smaller} we have analyzed the case where $N_1\ll C^3r/N_0^{1-\varepsilon}$ in \eqref{M-N1-range}.
Now we consider the contribution from the complementary range
\bea\label{generic-N1}
N_1\gg C^3r/N_0^{1-\varepsilon}
\eea
to \eqref{M-N1-range}.

To deal with \eqref{after GL3 Voronoi}, we first analyze the integral $\Psi_r^{\pm}\left(n_1^2n_2,q,\zeta\right)$ inside \eqref{J-middle}, which we recall is given by \eqref{Psi-integral}.
Applying Lemma \ref{voronoiGL3-holomorphic-asymptotic} for $\Psi_r^{\pm}$
and taking the $J$ there to be sufficiently large, one sees that the
analysis of $\Psi_r^{\pm}$ is reduced to treating the integral
\bna
\begin{split}
\Psi_{r,\zeta}^{\pm,0}:=&\left(\frac{n_1^2n_2}{q^3r}\right)^{2/3}\int_0^{\infty}y^{-1/3}
U\left(\frac{y}{N_0}\right)
e\left(-\frac{\zeta y}{qQ}\pm 3\left(\frac{n_1^2n_2 y}{q^3r}\right)^{1/3}\right)\mathrm{d}y\\
=&\left(\frac{n_1^2n_2N_0}{q^3r}\right)^{2/3}\int_0^{\infty}y^{-1/3}U\left(y\right)
e\left(-\frac{\zeta N_0 y}{qQ}\pm 3\left(\frac{N_0 n_1^2n_2}{q^3r}\right)^{1/3}y^{1/3}\right)\mathrm{d}y.
\end{split}\ena
Here $|\zeta|<N_0^{\varepsilon}$.
Without loss of generality, we may assume $\zeta\geq 0$.
Denote
\bna
\varrho_\pm(y)=-\frac{2\pi\zeta N_0 y}{qQ}\pm6\pi\left(\frac{N_0 n_1^2n_2}{q^3r}\right)^{1/3}y^{1/3}.
\ena
Then
\bna
\varrho'_\pm(y)&=&-\frac{2\pi\zeta N_0}{qQ}\pm
2\pi\left(\frac{N_0 n_1^2n_2}{q^3r}\right)^{1/3}y^{-2/3}.\\
%\varrho^{(j)}(y)&=&2\pi\left(-\frac{2}{3}\right)\left(-\frac{5}{3}\right)\cdots
%\left(\frac{4}{3}-j\right)\left(\frac{N_0 n_1^2n_2}{q^3r}\right)^{1/3}y^{1/3-j},
%\quad \mbox{for}\,j\geq 2.
\ena
Since $n_1^2n_2\sim N_1\gg C^3r/N_0^{1-\varepsilon}$ and we assumed $\zeta\geq 0$, for the ``$-$" case, $|\varrho'_{-}(y)|\gg \left(N_0 n_1^2n_2/q^3r\right)^{1/3}
\gg N_0^{\varepsilon}$. Then by using Lemma \ref{lemma:exponentialintegral} (1), we obtain $\Psi_{r,\zeta}^{-,0}\ll N_0^{-A}$, whose contribution to \eqref{M-N1-range} is therefore negligible. Hence we only need to concentrate on $\Psi_{r,\zeta}^{+,0}$ (correspondingly, for
the character sum $\mathfrak{C}_r(n_1,\pm n_2,m,q)$ in \eqref{after GL3 Voronoi}, we only need to consider
the ``$+$" case).

For $\Psi_{r,\zeta}^{+,0}$, if $\zeta$ satisfies $|\zeta| N_0/CQ\leq 1$, then similarly we have $\varrho'_{+}(y)\gg \left(N_0 n_1^2n_2/q^3r\right)^{1/3}
\gg N_0^{\varepsilon}$. By using integration by parts repeatedly we again have $\Psi_{r,\zeta}^{+,0}\ll N_0^{-A}$. Hence the contribution from such $\zeta$'s to \eqref{J-middle} is negligible. Consequently, we can further assume the $\zeta$ in \eqref{J-middle} satisfies the condition $CQ/N_0\ll \zeta\ll N_0^{\varepsilon}$. In particular, we can insert a smooth partition of unity for the $\zeta$-integral, and write
\bea\label{J-middle-2}
\mathfrak{J}_r^{\pm,+}(m,n_1^2n_2,q)&=&\sum_{CQ/N_0\ll \Xi\ll N_0^{\varepsilon}\atop \text{dyadic}}
\int_{\mathbb{R}}
g(q,\zeta)W\left(\frac{\zeta}{N_0^{\varepsilon}}\right)\widetilde{W}\left(\frac{\zeta}{\Xi}\right)\nonumber\\
&&\qquad\times \Phi^{\pm}\left(m,q,\zeta\right)
\Psi_{r,\zeta}^{+,0} \,\mathrm{d}\zeta+O(N_0^{-A}).
\eea
Here $\widetilde{W}(x)\in \mathcal{C}_c^{\infty}(1,2)$,
 satisfying $\widetilde{W}^{(j)}(x)\ll_j 1$ for $j\geq 0$.

Now we can proceed to apply the stationary phase method to the integral $\Psi_{r,\zeta}^{+,0}$. Calculating $\varrho'_{+}(y_0)=0$, the stationary point $y_0$ is given by
$y_0=\left(n_1^2n_2Q^3/\zeta^3N_0^2r\right)^{1/2}$.
Applying Lemma \ref{lemma:exponentialintegral} (2) with $X=Z=1$ and
$Y=R=\left(N_0 n_1^2n_2/q^3r\right)^{1/3}\gg N_0^{\varepsilon}$, we obtain
\bea\label{phi-0-middle}
\begin{split}
\Psi_{r,\zeta}^{+,0}=&\left(\frac{n_1^2n_2N_0}{q^3r}\right)^{2/3}
\frac{e^{i\varrho_{+}(y_0)}}{\sqrt{\varrho''_{+}(y_0)}}
F_{\natural}(y_0)+O_A\left(N_0^{-A}\right)\\
= &\left(\frac{-4\pi n_1^2n_2N_0}{3q^3r}\right)^{1/2}
e\left(2\left(\frac{Qn_1^2n_2}{q^2r\zeta}\right)^{1/2}\right)
\widetilde{F}_{\natural}\left(\frac{n_1n_2^{1/2}Q^{3/2}}{N_0r^{1/2}\zeta^{3/2}}\right)+O_A\left(N_0^{-A}\right),
\end{split}\eea
where $F_{\natural}$ is an $1$-inert function (depending on $A$)
supported on $y_0 \asymp 1$, and $\widetilde{F}_{\natural}(y):=y^{5/6}F_{\natural}(y)$.

Plugging \eqref{phi-0-middle} into \eqref{J-middle-2}, for
$ N_1\gg C^3r/N_0^{1-\varepsilon}$, up to a negligible error,
the evaluation of $\mathscr{T}_r^{\pm,\pm}(N_0,C,M,N_1)$
in \eqref{after GL3 Voronoi} is therefore reduced to estimating
\bea\label{Xi:range}
\sum_{\pm}\sum_{CQ/N_0\ll \Xi\ll N_0^{\varepsilon}\atop \text{dyadic}}
\hbar^{\pm}(\Xi)
\mathbf{M}_r^{\pm}(\Xi)
\eea
where $\mathbf{M}_r^{\pm}(\Xi)=\mathbf{M}_r^{\pm}(N_0,C,M,N_1,\Xi)$ is defined by
\bea\label{after GL3 Voronoi:2}
\mathbf{M}_r^{\pm}(\Xi)&=&\frac{N_0^{5/4}}{Qr^{1/2}}\sum_{q\sim C}
\frac{1}{q^2}
\sum_{n_1^2n_2\sim N_1 \atop n_1|qr}
\frac{\lambda_{\pi}\left(n_1,n_2\right)}{\sqrt{n_2}}
\sum_{m\sim M}\frac{\lambda_f(m)}{m^{1/4}}
w\left(\frac{m}{M}\right)
\nonumber\\&&\qquad\qquad\qquad \quad \times
\mathfrak{C}_r(n_1,n_2,m,q)
\mathcal{H}_r^{\pm}(n_1^2n_2,m,q,\Xi)
\eea
with
\bea\label{H-middle-0}
\mathcal{H}_r^{\pm}(n_1^2n_2,m,q,\Xi)&=&
\int_{\mathbb{R}}g(q,\zeta)
W\left(\frac{\zeta}{N_0^{\varepsilon}}\right)\widetilde{W}\left(\frac{\zeta}{\Xi}\right)
\Phi^{\pm}\left(m,q,\zeta\right)\nonumber\\
&&\qquad\qquad\times
\widetilde{F}_{\natural}\left(\frac{n_1n_2^{1/2}Q^{3/2}}{N_0r^{1/2}\zeta^{3/2}}\right)
e\left(2\left(\frac{Qn_1^2n_2}{q^2r\zeta}\right)^{1/2}\right)\mathrm{d}\zeta.
\eea
Here we have inserted an indicator function $\hbar^{\pm}(\Xi)$, which depending on the sign of $\Phi^{\pm}\left(m,q,\zeta\right)$, is defined as
 $\hbar^{+}(\Xi)=1$ and
 $\hbar^{-}(\Xi)=\mathbf{1}_{\Xi N_0/CQ\gg N_0^{-\varepsilon} \max\{t,\sqrt{N_0M}/C\}}$, to take into account the facts from Lemma \ref{phi:individual estimates}.

We first study the integral in \eqref{H-middle-0}. By plugging \eqref{Phi definition} for $\Phi^{\pm}\left(m,q,\zeta\right)$ in and switching the order of integrations,
\bea\label{H-middle}
\mathcal{H}_r^{\pm}(n_1^2n_2,m,q,\Xi)=
\int_0^\infty \mathcal{K}_r(y;n_1^2n_2,q,\Xi)V(y)y^{-1/4}
e\left(t\varphi(y)\pm\frac{2\sqrt{mN_0y}}{q}\right)
\mathrm{d}y
\eea
with
\bna
\mathcal{K}_r(y;n_1^2n_2,q,\Xi)=
\int_{\mathbb{R}}g(q,\zeta)
W\left(\frac{\zeta}{N_0^{\varepsilon}}\right)\widetilde{W}\left(\frac{\zeta}{\Xi}\right)
\widetilde{F}_{\natural}\left(\frac{n_1n_2^{1/2}Q^{3/2}}{N_0r^{1/2}\zeta^{3/2}}\right)
e\left(\frac{\zeta N_0y}{qQ}+2\left(\frac{Qn_1^2n_2}{q^2r\zeta}\right)^{1/2}\right)
\mathrm{d}\zeta.
\ena
We can give an asymptotic expansion for $\mathcal{K}_r(y;n_1^2n_2,q,\Xi)$.
By making a change of variable
$n_1n_2^{1/2}Q^{3/2}/N_0r^{1/2}\zeta^{3/2}\rightarrow \zeta$, we get
\bna
\mathcal{K}_r(y;n_1^2n_2,q,\Xi)
=\frac{2}{3}Q\left(\frac{n_1^2n_{2}}{N_0^{2}r}\right)^{1/3}
\int_{\mathbb{R}}\phi(\zeta)
\exp\left(i\varpi(\zeta)\right)\mathrm{d}\zeta,
\ena
where
\bna
\phi(\zeta):=-\zeta^{-5/3}\widetilde{F}_{\natural}(\zeta)
g\left(q,\frac{n_1^{2/3}n_2^{1/3}Q}{\zeta^{2/3}r^{1/3}N_0^{2/3}}\right)
W\left(\frac{n_1^{2/3}n_2^{1/3}Q}{\zeta^{2/3}r^{1/3}N_0^{2/3+\varepsilon}}\right)
\widetilde{W}\left(\frac{n_1^{2/3}n_2^{1/3}Q}{\zeta^{2/3}r^{1/3}N_0^{2/3}\Xi}\right)
\ena
and the phase function $\varpi(\zeta)$ is given by
\bna
\varpi(\zeta)=2\pi\left(\frac{N_0 n_1^2n_2}{q^3r}\right)^{1/3}
\left(y\zeta^{-2/3}+2\zeta^{1/3}\right).
\ena
Note that
\bea\label{1st deri}
\varpi'(\zeta)=
\frac{4\pi}{3}\left(\frac{N_0n_1^2n_2}{q^3r}\right)^{1/3}\left(-y\zeta^{-5/3}+\zeta^{-2/3}\right),
\eea
and for $j\geq 2$,
\bea\label{jth deri}
\varpi^{(j)}(\zeta)
=\frac{4\pi}{3}\left(-\frac{5}{3}\right)\cdot\cdot\cdot\left(\frac{1}{3}-j\right)
\left(\frac{N_0n_1^2n_2}{q^3r}\right)^{1/3}\left(
-y\zeta^{-2/3-j}+\frac{2}{3j-1}\zeta^{1/3-j}\right).
\eea
From \eqref{1st deri},
there is a stationary point $\zeta_0=y$ and from \eqref{jth deri}
$\varpi^{(j)}(\zeta)\ll_j \left(N_0n_1^2n_2/q^3r\right)^{1/3}$ for $j\geq 2$.
By \eqref{rapid decay g}, we have $\phi^{(j)}(\zeta)\ll_j Q^{\varepsilon}$.
Applying Lemma \ref{lemma:exponentialintegral} (2) with $X=Z=1$ and
$Y=R=\left(N_0 n_1^2n_2/q^3r\right)^{1/3}\gg N_0^{\varepsilon}$, we obtain
\bea\label{K-integral}
\begin{split}
\mathcal{K}_r(y;n_1^2n_2,q,\Xi)
=&Q\left(\frac{n_1^2n_2}{N_0^2r}\right)^{1/3}
\frac{e^{i\varpi(\zeta_0)}}{\sqrt{\varpi''(\zeta_0)}}
G_{\natural}(\zeta_0)+O_A\left(N_0^{-A}\right)\\
=&Q\left(\frac{q^3n_1^2n_2}{N_0^5r}\right)^{1/6}\left(\frac{3}{4\pi}\right)^{1/2}y^{5/6}e\left(\frac{3}{q}\left(\frac{N_0n_1^2n_2y}{r}\right)^{1/3}\right)G_{\natural}(y)+O_A\left(N_0^{-A}\right),
\end{split}
\eea
where $G_{\natural}$ is an inert function (depending on $A$ and $\Xi$)
supported on $\zeta_0 \asymp 1$.
Finally, substituting \eqref{K-integral} into
\eqref{H-middle}, one has
\bea\label{H-middle:0}
&&\mathcal{H}_r^{\pm}(n_1^2n_2,m,q,\Xi)\nonumber\\
&=&c_0Q\left(\frac{q^3n_1^2n_2}{N_0^5r}\right)^{1/6}
\int_0^\infty V(y)G_{\natural}(y)y^{7/12}\nonumber\\
&&e\left(t\varphi(y)\pm\frac{2}{q}(N_0m)^{1/2}y^{1/2}
+\frac{3}{q}\left(\frac{N_0n_1^2n_2}{r}\right)^{1/3}y^{1/3}\right)
\mathrm{d}y+O_A\left(N_0^{-A}\right)
\eea
for some absolute constant $c_0$.

Further substituting \eqref{H-middle:0} into \eqref{after GL3 Voronoi:2}, we arrive at
\bea\label{beforeCauchy}
\mathbf{M}_r^{\pm}(\Xi)&=&\frac{c_0N_0^{5/12}}{r^{2/3}}
\sum_{n_1^2n_2\sim N_1}\left(\frac{n_1}{n_2}\right)^{1/3}
\lambda_{\pi}\left(n_1,n_2\right)
\sum_{q\sim C\atop n_1|qr}
\frac{1}{q^{3/2}}\sum_{m\sim M}\frac{\lambda_f(m)}{m^{1/4}}w\left(\frac{m}{M}\right)\nonumber\\
&&\qquad\qquad\times\mathfrak{C}_r(n_1,n_2,m,q)\,
\mathfrak{J}_r^{\pm}(n_1^2n_2,m,q)+O_A\left(N_0^{-A}\right),
\eea
where
\bea\label{J-definition}
\mathfrak{J}_r^{\pm}(n_1^2n_2,m,q)
=\int_0^\infty \widetilde{V}(y)
e\left(t\varphi(y)\pm\frac{2}{q}(N_0m)^{1/2}y^{1/2}
+\frac{3}{q}\left(\frac{N_0n_1^2n_2}{r}\right)^{1/3}y^{1/3}\right)
\mathrm{d}y.
\eea
Here $\widetilde{V}(y)=V(y)G_{\natural}(y)y^{7/12}$,
satisfying $\widetilde{V}^{(j)}(y)\ll_j \triangle^j$ and $\text{Var}(\widetilde{V})\ll 1$.
Recall $\triangle$ denotes the quantity such that $V^{(j)}(x)\ll \triangle^j$; see \eqref{derivative-of-V}.

\begin{remark}
For later purpose, we derive a direct estimate for $\mathfrak{J}_r^{\pm}(n_1^2n_2,m,q)$. By making a change of variable $y=z^2$,
\bna
\mathfrak{J}_r^{\pm}(n_1^2n_2,m,q)
=2\int_0^\infty z\widetilde{V}(z^2)
e\left(t\varphi(z^2)\pm\frac{2}{q}(N_0m)^{1/2}z
+\frac{3}{q}\left(\frac{N_0n_1^2n_2}{r}\right)^{1/3}z^{2/3}\right)
\mathrm{d}z.
\ena
If we assume $\left(\varphi(z^2)\right)''\neq 0$, equivalently $\varphi(z)\neq cz^{1/2}+c_0$,
then the second derivative of the phase function satisfies
\bna
t\left(\varphi(z^2)\right)''-\frac{2}{3}\left(\frac{N_0n_1^2n_2}{q^3r}\right)^{1/3}z^{-4/3}
\gg \max\{t, \lambda\}
\ena
and by Lemma \ref{lem: 2st derivative test, dim 1}, we have
\bea\label{J minus:second derivative estimate}
\mathfrak{J}_r^{\pm}(n_1^2n_2,m,q)\ll  \max\{t,\lambda\}^{-1/2},
\eea
with
\bea\label{lambda}
\lambda:=\left(N_0N_1/(C^3r)\right)^{1/3},
\eea
except for the case $t\asymp \lambda$. For the later exceptional case, we will use an $L^2$-estimate due to Munshi
(see \cite[Lemma 1]{Mun6})
\bea\label{L square bound}
\int_0^{\infty}\psi(\xi)\left|\mathfrak{J}_r^{\pm}(N_1\xi^3,m,q)\right|^2\mathrm{d}\xi
\ll \max\{t, \lambda\}^{-1},
\eea
where $\psi$ is a bump function.
\end{remark}

\subsection{Application of the Poisson summation}
Applying the Cauchy--Schwarz inequality and using the Rankin--Selberg estimate \eqref{GL3-Rankin--Selberg}, one sees that the
$\mathbf{M}_r^{\pm}(\Xi)$ in
\eqref{beforeCauchy} is bounded by
\bna
\frac{N_0^{5/12}N_1^{1/6}}{r^{2/3}}\left(\sum_{n_1^2n_{2}\sim N_1}n_1^2\bigg|
\sum_{q\sim C\atop n_{1}|qr}q^{-3/2}
\sum_{m\sim M}\frac{\lambda_{f}(m)}{m^{1/4}}w\left(\frac{m}{M}\right)
\mathfrak{C}_r(n_1,n_2,m,q) \mathfrak{J}_r^{\pm}(n_1^2n_2,m,q)\bigg|^2\right)^{1/2}.
\ena
 As in Munshi \cite{Mun6}, we write $q=q_1q_2$ with $q_1|(n_1r)^{\infty}$,
$(q_2,n_1r)=1$ and apply Cauchy--Schwarz again to get that the expression
inside the absolute value being
\bna
&&\sum_{n_1|q_1r\atop q_1|(n_1r)^{\infty}}q_1^{-3/2}
\sum_{q_2\sim C/q_1}q_2^{-3/2}\sum_{m\sim M}\frac{\lambda_{f}(m)}{m^{1/4}}w\left(\frac{m}{M}\right)
\mathfrak{C}_r(n_1,n_2,m,q_1q_2)  \mathfrak{J}_r^{\pm}\left(n_1^2n_2,m,q_1q_2\right)
\nonumber\\
&\ll&N_0^{\varepsilon}\left(
\sum_{n_1|q_1r\atop q_1|(n_1r)^{\infty}}q_1^{-3}
\bigg|\sum_{q_2\sim C/q_1}q_2^{-3/2}
\sum_{m\sim M}\frac{\lambda_{f}(m)}{m^{1/4}}w\left(\frac{m}{M}\right)
\mathfrak{C}_r(n_1,n_2,m,q_1q_2)
\mathfrak{J}_r^{\pm}\left(n_1^2n_2,m,q_1q_2\right)\bigg|^2\right)^{1/2}.
\ena
Therefore we have
\bea\label{Cauchy}
\mathbf{M}_r^{\pm}(\Xi)\ll
\frac{N_0^{5/12+\varepsilon}N_1^{1/6}}{r^{2/3}}
\bigg\{\sum_{n_1\leq Cr}n_1^2\sum_{n_1|q_1r\atop q_1|(n_1r)^{\infty}}q_1^{-3}\times
\mathbf{\Omega}^{\pm}(n_1,q_1,r)\bigg\}^{1/2},
\eea
where
\bea\label{Omega}
\mathbf{\Omega}^{\pm}(n_1,q_1,r)&=&
\sum_{n_{2}}\phi\left(\frac{n_2}{N_1/n_1^2}\right)\bigg|
\sum_{q_2\sim C/q_1}q_2^{-3/2}
\sum_{m\sim M}\frac{\lambda_{f}(m)}{m^{1/4}}w\left(\frac{m}{M}\right) \nonumber\\
&&\quad\quad\quad\quad\quad\quad\times \mathfrak{C}_r(n_1,n_2,m,q_1q_2)
\mathfrak{J}_r^{\pm}\left(n_1^2n_2,m,q_1q_2\right)\bigg|^2.
\eea
Here $\phi$ is a nonnegative smooth function on $(0,+\infty)$, supported on $[2/3,3]$, and such
that $\phi(x)=1$ for $x\in [1,2]$.

Opening the absolute square, we break the $n_2$-sum into congruence classes
modulo $q_1q_2q'_2r/n_1$, and then apply the Poisson summation formula to the sum over $n_2$.
It is therefore sufficient to consider the following sum (with a little abuse of notation)
\bea\label{omega-bound}
\mathbf{\Omega}^{\pm}(n_1,q_1,r)
&=&\sum_{q_2\sim C/q_1}q_2^{-3/2}\sum_{q_2'\sim C/q_1}q_2'^{-3/2}
\sum_{m\sim M}\frac{1}{m^{1/4}}\sum_{m'\sim M}
\frac{|\lambda_{f}(m')|^2}{m'^{1/4}}\nonumber\\
&&\qquad\qquad\qquad\qquad\qquad\times
\frac{N_1}{n_1^2}\sum_{\tilde{n}_2\in \mathbb{Z}}\left|\mathfrak{K}(\tilde{n}_2)\right|
\, \left|\mathfrak{I}^{\pm}\left(\frac{N_1\tilde{n}_2}{q_2q_2'q_1n_1r}\right)\right|,
\eea
where the character sum $\mathfrak{K}(\tilde{n}_2):=\mathfrak{K}(\tilde{n}_2,m,m',q_2,q_2',n_1,q_1,r)$ is
given by
\bea\label{K character sum}
\mathfrak{K}(\tilde{n}_2)=\frac{n_1}{q_2q_2'q_1r}\sum_{\beta\bmod q_2q_2'q_1r/n_1}
\mathfrak{C}_r(n_1,\beta,m,q_1q_2)
\overline{\mathfrak{C}_r(n_1,\beta,m',q_1q_2')}
\, e\left(\frac{\tilde{n}_2\beta}{q_2q_2'q_1r/n_1}\right)
\eea
and the integral $\mathfrak{I}^{\pm}(X)=\mathfrak{I}^{\pm}(X;m,m',q_2,q_2',q_1,r)$ is given by
\bea\label{I-integral}
\mathfrak{I}^{\pm}(X)=\int_{\mathbb{R}}
\phi\left(\xi\right)
\mathfrak{J}_r^{\pm}\left(N_1\xi,m,q_1q_2\right)
\overline{\mathfrak{J}_r^{\pm}\left(N_1\xi,m',q_1q_2'\right)}
\, e\left(-X\xi\right)\mathrm{d}\xi.
\eea
We will finish the treatment of $\mathbf{\Omega}^{\pm}(n_1,q_1,r)$ and thus the targeted
sum $\mathscr{S}_r(N_0)$ in \eqref{aim-sum} upon plugging appropriate estimates for $\mathfrak{K}(\tilde{n}_2)$ and $\mathfrak{I}^{\pm}(X)$ into \eqref{omega-bound}.

\begin{remark}\label{new-length-n2}
We point out, before analyzing the integral $\mathfrak{I}^{\pm}(X)$ in details,
how we can effectively truncate the length of the $\tilde{n}_2$-sum in \eqref{omega-bound}. Note that before applying the Poisson summation formula, the ``arithmetic conductor" in the $n_2$-sum is $q_1q_2q'_2r/n_1$, and the ``analytic conductor" (i.e., the size of oscillation of the weight function in the $n_2$-variable) is of size $\frac{3}{q}\left(\frac{N_0n_1^2n_2}{r}\right)^{1/3}\asymp (\frac{N_0N_1}{C^3r})^{1/3}$; see \eqref{J-definition}. Hence the dual $n_2$-sum can be truncated at $|\tilde{n}_2|\ll \frac{\frac{q_1q_2q'_2r}{n_1}  (\frac{N_0N_1}{C^3r})^{1/3}}{N_1/n_1^2}\asymp \frac{n_1r^{2/3}CN_0^{1/3}}{q_1 N_1^{2/3}}$; see \eqref{N2 range}.
\end{remark}

We have the following estimates for $\mathfrak{K}(\tilde{n}_2)$ and $\mathfrak{I}^{\pm}(X)$, whose proofs we postpone to Section \ref{proofs-of-technical-lemmas}.
\begin{lemma}\label{lem: character sum}
Let $\mathfrak{K}(\tilde{n}_2)$ be as in \eqref{K character sum}. Then
\bea\label{character sum: nonzero}
\mathfrak{K}(\tilde{n}_2)\ll\mathop{\sum\sum}_{\substack{d_1, d_1'|q_1}}d_1d_1'
\mathop{\mathop{\sideset{}{^\star}\sum}_{\substack{\alpha\bmod{q_1r/n_1}\\
n_1\alpha\equiv -m\bmod{d_1}}}\;\mathop{\sideset{}{^\star}\sum}_{\substack{\alpha'\bmod{q_1r/n_1}
\\n_1\alpha'\equiv -m'\bmod{d_1'}}}}_{q_2\bar{\alpha}'-q_2'\bar{\alpha}
\equiv \tilde{n}_2\bmod{q_1r/n_1}}\; \mathop{\sum\sum}_{\substack{d_2|(q_2,q_2'n_1- m\tilde{n}_2)
\\ d_2'|(q_2',q_2n_1+ m'\tilde{n}_2)}}d_2d_2'.
\eea
Moreover, if $\tilde{n}_2=0$, we imply
\bea\label{q2=q2prime}
q_2=q_2',
\eea
 and
\bea\label{character sum: zero}
\mathfrak{K}(0)\ll q_1q_2r\mathop{\sum\sum}_{\substack{d, d'|q_1q_2\\(d,d')|(m-m')}}(d,d').
\eea
\end{lemma}

\begin{lemma}\label{integral:lemma}
Let $\mathfrak{I}^{\pm}(X)$ be as in \eqref{I-integral}. Denote $\lambda=\left(N_0N_1/C^3r\right)^{1/3}$.

(1) We have $\mathfrak{I}^{\pm}(X)=O\left(N_0^{-A}\right)$ if $X\gg \lambda^{1+\varepsilon}$.

(2) For $X\ll \lambda^{3+\varepsilon}\max\left\{t,\sqrt{N_0M}/C\right\}^{-2}$, we have
   $\mathfrak{I}^+(X)\ll \max\left\{t,\sqrt{N_0M}/C\right\}^{-1}$.

(3) Assume $\triangle<t^{1/2-\varepsilon}$ (Here $\triangle$ is such that
$V^{(j)}(x)\ll \triangle^j$, defined in \eqref{derivative-of-V}.) and $C$ satisfies $\frac{N_0^{1+\varepsilon}}{CQ}\ll t^{1-\varepsilon}$. Further assume $\varphi$ satisfies the condition \eqref{phi assumption}. Then for $X$ such that $\lambda^{3+\varepsilon}\max\left\{t,\sqrt{N_0M}/C\right\}^{-2} \ll X\ll \lambda^{2+\varepsilon}
\max\left\{t,\sqrt{N_0M}/C\right\}^{-1}$, we have
$\mathfrak{I}^+(X)\ll\max\left\{t,\sqrt{N_0M}/C\right\}^{-1}|X|^{-1/3}$.

(4) For $X\gg \lambda^{2+\varepsilon}\max\left\{t,\sqrt{N_0M}/C\right\}^{-1}$, we have
$\mathfrak{I}^+(X)\ll\max\left\{t,\sqrt{N_0M}/C\right\}^{-1}|X|^{-1/2}$.

(5) Let $q_2=q_2'$. Then
\bna
\mathfrak{I}^{\pm}(0)\ll N_0^{\varepsilon}\,\min\left\{t^{-1},
C\lambda^{-1}N_0^{-1/2}
\left|m^{1/2}-m'^{1/2}\right|^{-1}\right\}.
\ena
\end{lemma}
\begin{remark} The principle of the proof is similar to \cite[Lemma 5.5]{AHLQ}
and \cite[Lemma 5]{Mun6}. The first bound for $\mathfrak{I}^{\pm}(0)$ in
Lemma \ref{integral:lemma} (5) should be regarded as the ``trivial" estimate, while the second bound will help us to take care of the case $m\neq m'$ when we estimate the sum in \eqref{omega-bound}, whose contribution in fact turns out to be smaller compared to those with $m=m'$; see the proof in Section \ref{zero-frequency}.
\end{remark}

With estimates for $\mathfrak{K}(\tilde{n}_2)$ and $\mathfrak{I}^{\pm}(X)$ ready, we now continue with the treatment of $\mathbf{\Omega}^{\pm}(n_1,q_1,r)$
in \eqref{omega-bound}. Denote $X=\frac{N_1\tilde{n}_2}{q_2q_2'q_1n_1r}$ and
$\lambda=\left(N_0N_1/C^3r\right)^{1/3}$. If $X\gg \lambda^{1+\varepsilon}$, that is,
if
\bea\label{N2 range}
|\tilde{n}_2|\gg \frac{q_2q_2'q_1n_1r}{N_1}\lambda^{1+\varepsilon}\asymp
\frac{C^2n_1r}{N_1q_1}\lambda^{1+\varepsilon}
\asymp \frac{n_1r^{2/3}CN_0^{1/3+\varepsilon}}{q_1 N_1^{2/3}}:=N_2,
\eea
then by Lemma \ref{integral:lemma} (1), the contribution from such terms to $\mathbf{\Omega}^{\pm}(n_1,q_1,r)$ is negligible (cf. Remark \ref{new-length-n2}). We only need to consider the range $0\leq |\tilde{n}_2|\leq N_2$.

We treat the cases where $\tilde{n}_2=0$ and $\tilde{n}_2\neq 0$ separately and denote their contributions to $\mathbf{\Omega}^{\pm}(n_1,q_1,r)$ by $\mathbf{\Omega}^{\pm}_0$ and $\mathbf{\Omega}^{\pm}_{\neq 0}$, respectively.

\subsection{The zero frequency}\label{zero-frequency}
Let $\mathbf{\Sigma}^{\pm}_{0}$ denote the contribution of $\mathbf{\Omega}^{\pm}_0$ to \eqref{Cauchy}.

For $\tilde{n}_2=0$, from \eqref{q2=q2prime} we necessarily have $q_2=q'_2$ in \eqref{omega-bound}.
Splitting the sum over $m$  and  $m'$  according as  $m=m'$ or not,
and by applying \eqref{character sum: zero} and Lemma \ref{integral:lemma} (5), we have
\bna
\mathbf{\Omega}^{\pm}_0
&\ll&\frac{N^{\varepsilon}N_1q_1r}{n_1^2} t^{-1}
\sum_{q_2\sim C/q_1}q_2^{-2}\sum_{m'\sim M}\frac{|\lambda_{f}(m')|^2}{m'^{1/2}}
\mathop{\sum\sum}_{d,d'|q_1q_2}(d,d')\\
&+&\frac{N^{\varepsilon}CN_1q_1r}{\lambda N_0^{1/2}n_1^2}
\sum_{q_2\sim C/q_1}q_2^{-2}\sum_{m\sim M}\frac{1}{m^{1/4}}
\mathop{\sum}_{\substack{m'\sim M\\m'\neq m}}\frac{|\lambda_{f}(m')|^2}{m'^{1/4}}
\mathop{\sum}_{\substack{d, d'|q_1q_2\\(d,d')|(m-m')}}(d,d')\frac{M^{1/2}}{|m-m'|},
\ena
where $\lambda=\left(N_0N_1/C^3r\right)^{1/3}$. Here we have replaced the factor $|m^{1/2}-m'^{1/2}|^{-1}$ by $\frac{m^{1/2}+m'^{1/2}}{|m-m'|}\asymp \frac{M^{1/2}}{|m-m'|}$.
Using \eqref{GL2: Rankin Selberg} and the divisor bound, we derive
\bna
\mathbf{\Omega}^{\pm}_0&\ll& \frac{N^{\varepsilon}N_1M^{1/2}q_1^2r}{n_1^2t}
+\frac{N^{\varepsilon}MN_1q_1^2r}{\lambda N_0^{1/2}n_1^2}
\sum_{q_2\sim C/q_1}q_2^{-1}\sum_{1\leq l<M}l^{-1}
\mathop{\sum}_{\substack{d, d'|q_1q_2\\(d,d')|l}}(d,d')\\
&\ll& N^{\varepsilon}\frac{N_1q_1^2r}{n_1^2}
\left(\frac{M^{1/2}}{t}+\frac{MCr^{1/3}}{N_0^{5/6}N_1^{1/3}}\right).
\ena
This bound when substituted in place of $\mathbf{\Omega}^{\pm}(n_1,q_1,r)$ into \eqref{Cauchy} yields that
\bea\label{estimate-1}
\mathbf{\Sigma}^{\pm}_0&\ll& \frac{N_0^{5/12+\varepsilon}N_1^{1/6}}{r^{2/3}}
\left\{\sum_{n_1\leq Cr}n_1^2\sum_{n_1|q_1r\atop q_1|(n_1r)^{\infty}}q_1^{-3}\times
\frac{N_1q_1^2r}{n_1^2}
\left(\frac{M^{1/2}}{t}+\frac{MCr^{1/3}}{N_0^{5/6}N_1^{1/3}}\right)\right\}^{1/2}\nonumber\\
&\ll&\frac{N_0^{5/12+\varepsilon}}{r^{1/6}}
\left(\frac{M^{1/4}N_1^{2/3}}{t^{1/2}}+
\frac{M^{1/2}C^{1/2}r^{1/6}N_1^{1/2}}{N_0^{5/12}}\right).
\eea
We recall from \eqref{M-N1-range} that $1\ll M\ll N_0^{\varepsilon}\max\left\{C^2t^2/N_0,N_0/Q^2\right\}$ and $1\ll N_1\ll N_0^{2+\varepsilon}r/Q^3$.
For $C\ll Q$, then
\bna
\mathbf{\Sigma}^{\pm}_0
&\ll& N^{\varepsilon}r^{1/2}
\left(\frac{N_0^{7/4}M^{1/4}}{Q^{2}t^{1/2}}+\frac{N_0M^{1/2}}{Q}\right)\\
&\ll& N^{\varepsilon}r^{1/2}
\left(\frac{N_0^{3/2}}{Q^{3/2}}+
\frac{N_0^2}{Q^{5/2}t^{1/2}}+N_0^{1/2}t
+\frac{N_0^{3/2}}{Q^2}\right).
\ena
In particular, if we further assume $Q$ satisfies $(N_0/t)^{1/2}<Q<(N_0/t)^{2/3}$, then we get
\bna
\mathbf{\Sigma}^{\pm}_0
\ll N^{\varepsilon}r^{1/2}\left(\frac{N_0^{3/2}}{Q^{3/2}}+N_0^{1/2}t\right)
\ll \frac{r^{1/2}N_0^{3/2+\varepsilon}}{Q^{3/2}}.
\ena
Hence we have shown the following.
\begin{lemma}\label{lemma:zero}
Assume
\bea\label{assumption: range 1}
(N_0/t)^{1/2}<Q<(N_0/t)^{2/3}.
\eea
We have
\bea\label{Sigma-0-bound}
\mathbf{\Sigma}^{\pm}_0\ll  \frac{r^{1/2}N_0^{3/2+\varepsilon}}{Q^{3/2}}.
\eea
\end{lemma}

\subsection{The non-zero frequencies}\label{counting-part}
Recall $\mathbf{\Omega}^{\pm}_{\neq 0}$ denotes the contribution from the terms with $\tilde{n}_2\neq 0$
to $\mathbf{\Omega}^{\pm}(n_1,q_1,r)$ in
\eqref{omega-bound}. Correspondingly, we denote its contribution to \eqref{Cauchy} by $\mathbf{\Sigma}^{\pm}_{\neq 0}$.

Plugging the bound \eqref{character sum: nonzero} into \eqref{omega-bound}, we get
\bna
\mathbf{\Omega}^{\pm}_{\neq 0}&\ll&
\frac{N_1}{n_1^{2}}\sum_{q_2\sim C/q_1}q_2^{-3/2}\sum_{q_2'\sim C/q_1}q_2'^{-3/2}
\sum_{m\sim M}\frac{1}{m^{1/4}}\sum_{m'\sim M}\frac{|\lambda_{f}(m')|^2}{m'^{1/4}}
\mathop{\sum\sum}_{d,d'|q_1}d_1d_1'\\
&&\times\sum_{0\neq \tilde{n}_2\ll N_2}\mathop{\mathop{\sideset{}{^\star}\sum}_{
\substack{\alpha\bmod{q_1r/n_1}\\
n_1\alpha\equiv -m\bmod{d_1}}}\;\mathop{\sideset{}{^\star}\sum}_{\substack{\alpha'\bmod{q_1r/n_1}
\\n_1\alpha'\equiv -m'\bmod{d_1'}}}}_{q_2\bar{\alpha}'-q_2'\bar{\alpha}
\equiv  \tilde{n}_2\bmod{q_1r/n_1}}\; \mathop{\sum\sum}_{\substack{d_2|(q_2,q_2'n_1- m\tilde{n}_2)
\\ d_2'|(q_2',q_2n_1+ m'\tilde{n}_2)}}d_2d_2'\left|\mathfrak{I}^{\pm}\left(\frac{N_1|\tilde{n}_2|}
{q_2q_2'q_1n_1r}\right)\right|,
\ena
where $N_2$ is defined in \eqref{N2 range}.
%One has
%\bna
%\mathbf{\Omega}_{\neq 0}&\ll&
%\frac{N_1q_1^3}{n_1^2C^3M^{1/2}}
%\mathop{\sum\sum}_{q_2,q_2'\sim C/q_1}
%\mathop{\sum\sum}_{m,m'\sim M}|\lambda_{f}(m')|^2\mathop{\sum\sum}_{d,d'|q_1}d_1d_1'
%\sum_{0\neq n_2\ll N_2}\\
%&\times&\mathop{\mathop{\sideset{}{^\star}\sum}_{\substack{\alpha\bmod{q_1r/n_1}\\
%n_1\alpha\equiv -m\bmod{d_1}}}\;\mathop{\sideset{}{^\star}\sum}_{\substack{\alpha'\bmod{q_1r/n_1}
%\\n_1\alpha'\equiv -m'\bmod{d_1'}}}}_{q_2\bar{\alpha}'-q_2'\bar{\alpha}
%\equiv n_2\bmod{q_1r/n_1}}\; \mathop{\sum\sum}_{\substack{d_2|(q_2,q_2'n_1- mn_2)
%\\ d_2'|(q_2',q_2n_1+m'n_2)}}d_2d_2'\left|\mathfrak{I}
%\left(\frac{N_1|n_2|}{q_2q_2'q_1n_1r}\right)\right|.
%\ena
Writing $q_2d_2$ in place of $q_2$ and $q_2'd_2'$ in place of $q_2'$, and noting that for fixed $(\alpha, d_2, d_2', q_2, q_2', \tilde{n}_2)$ the congruence condition
$d_2q_2\overline{\alpha'}-d_2'q_2'\overline{\alpha}\equiv \tilde{n}_2\bmod{q_1 r/n_1}$
determines $\alpha'$ uniquely, we infer
\bea\label{large-and-small-modulus}
\begin{split}
\mathbf{\Omega}^{\pm}_{\neq 0}\ll& \frac{N_1q_1^3}{n_1^2C^3M^{1/2}}
\mathop{\sideset{}{^\star}{\sum}}_{\substack{\alpha\bmod{q_1r/n_1}}}
\mathop{\sum\sum}_{d_1,d_1'|q_1}\:d_1d_1'
\mathop{\sum}_{\substack {d_2\ll C/q_1\\(d_2,d_1)=1}}
\mathop{\sum}_{\substack {d_2'\ll C/q_1\\(d_2',d_1')=1}}\:d_2d_2'\\
&\qquad\times\mathop{\sum\sum}_{\substack{q_2\sim C/q_1d_2\\q_2'\sim C/q_1d_2'}}
\sum_{0\neq \tilde{n}_2\ll N_2}
\mathop{\sum}_{\substack {m'\sim M\\q_2d_2n_1+m' \tilde{n}_2\equiv 0\bmod d_2'}}|\lambda_{f}(m')|^2\\&\qquad\qquad\qquad
\qquad\qquad
\times
\mathop{\sum}_{\substack {{m\sim M}\\{n_1\alpha \equiv -m\bmod d_1}
\\q_2'd_2'n_1- m \tilde{n}_2\equiv 0\bmod d_2}}\left|\mathfrak{I}^{\pm}\left(\frac{N_1|\tilde{n}_2|}{q_2d_2q_2'd_2'q_1n_1r}\right)\right|.
\end{split}
\eea
We would like to apply the properties in Lemma \ref{integral:lemma} for the integral $\mathfrak{I}^{\pm}(X)$. For this purpose, we split the
modulus $C$ according to $N_0^{1+\varepsilon}/Qt\leq C$ or $C\leq N_0^{1+\varepsilon}/Qt$. We consider the former case first and will treat the later case in Section \ref{The case of small modulus} for which we will apply the simpler bound
\bea\label{J-for-smallC}
\mathfrak{I}^{\pm}(X)\ll  \max\{t,\lambda\}^{-1},
\eea
which follows from \eqref{L square bound} and is valid any $X\geq 0$.

\subsubsection{The case of large modulus}\label{The case of large modulus}
In this subsection, we consider the case
\bea\label{C range-large modulus}
N_0^{1+\varepsilon}/Qt \leq C\ll Q .
\eea
Note that for $C$ in this range and $\Xi\ll N_0^{\varepsilon}$,
we have $\Xi N_0/CQ\ll
N_0^{-\varepsilon} \max\{t,\sqrt{N_0M}/C\}$. From the definition in \eqref{Xi:range},
there is only the term $\mathbf{M}_r^{+}(\Xi)$ to bound, and correspondingly we only need to treat $\mathbf{\Omega}^{+}(n_1,q_1,r)$ in \eqref{omega-bound}.

\begin{lemma}\label{lemma:nonzero}
Assume
\bea\label{assumption: range 2}
Q>\max\{(N_0/t)^{1/2},N_0/t^2\}.
\eea
We have
\bea\label{large bound}
\mathbf{\Sigma}^{+}_{\neq 0}\ll
r^{1/2}N_0^{1/4+\varepsilon}Qt\left(1+\frac{N_0^{5/4}}{Q^{5/2}t}\right).
\eea
\end{lemma}
\begin{proof}For $X=N_1|\tilde{n}_2|/q_2d_2q_2'd_2'q_1n_1r$, in order to apply the bounds for $\mathfrak{I}^{+}(X)$ in Lemma \ref{integral:lemma}, we consider the cases where
$X\ll \lambda^{3+\varepsilon}\max\left\{t,\sqrt{N_0M}/C\right\}^{-2}$,
$\lambda^{3+\varepsilon}\max\left\{t,\sqrt{N_0M}/C\right\}^{-2} \ll X\ll \lambda^{2+\varepsilon}
\max\left\{t,\sqrt{N_0M}/C\right\}^{-1}$,
and $X\gg \lambda^{2+\varepsilon}
\max\left\{t,\sqrt{N_0M}/C\right\}^{-1}$ separately, and we split the sum over $\tilde{n}_2$ accordingly.
Set
\bea\label{N3N4}
N_3:=N^{\varepsilon}\frac{\lambda^3C^2n_1r}{N_1q_1}\max\left\{t,\sqrt{N_0M}/C\right\}^{-2},
\quad
N_4:=N^{\varepsilon}\frac{\lambda^2C^2n_1r}{N_1q_1}\max\left\{t,\sqrt{N_0M}/C\right\}^{-1}.
\eea
Then
\bea\label{Omega-nonzero-split}
\mathbf{\Omega}^{+}_{\neq 0}\ll
\mathbf{\Omega}_{\neq 0}^{a}+\mathbf{\Omega}_{\neq 0}^{b}+\mathbf{\Omega}_{\neq 0}^{c},
\eea
where $\mathbf{\Omega}_{\neq 0}^{a}$ is the contribution from
$0\neq \tilde{n}_2\ll N_3$ for which we will use the bound
$\mathfrak{I}^{+}(X)\ll \max \left\{t,\sqrt{N_0M}/C\right\}^{-1}$ in Lemma \ref{integral:lemma}
(2); $\mathbf{\Omega}_{\neq 0}^{b}$ is the contribution from
$N_3\ll \tilde{n}_2\ll N_4$ for which we use the bound
$\mathfrak{I}^{+}(X)\ll \max \left\{t,\sqrt{N_0M}/C\right\}^{-1}|X|^{-1/3}$
in Lemma \ref{integral:lemma}
(3), and $\mathbf{\Omega}_{\neq 0}^{c}$ is the remaining part for
which we apply the bound
$\mathfrak{I}^{+}(X)\ll \max \left\{t,\sqrt{N_0M}/C\right\}^{-1}|X|^{-1/2}$ in Lemma \ref{integral:lemma}
(4).

Notice that for fixed tuple $(n_1,\alpha,\tilde{n}_2)$,
the congruences
$$
\left\{\begin{array}{l}
n_1\alpha \equiv -m\bmod d_1\\
   q_2'd_2'n_1- m \tilde{n}_2\equiv 0\bmod d_2
\end{array}
\right.
$$
imply that $m$ is uniquely determined modulo
$d_1d_2/(d_2,\tilde{n}_2)$. Therefore the number of $m$ is dominated by
$O\left((d_2,\tilde{n}_2)\left(1+M/d_1d_2\right)\right)$. We conclude that
\bna
\mathbf{\Omega}_{\neq 0}^{a}&\ll&\frac{N_1 q_1^4r}{n_1^3 C^3 M^{1/2}}
\mathrm{max}\left\{t, \frac{\sqrt{N_0M}}{C}\right\}^{-1}
  \mathop{\sum\sum}_{d_1, d_1'|q_1}\:d_1d_1'
  \mathop{\sum}_{d_2\ll C/q_1\atop (d_2, d_1)=1}\:
  \mathop{\sum}_{d_2'\ll C/q_1\atop (d_2', d_1')=1}\:d_2d_2'\\
&&\times \mathop{\sum\sum}_{\substack{q_2\sim C/q_1d_2\\q_2'\sim C/q_1d_2'}}
\mathop{\sum}_{0\neq \tilde{n}_2\ll N_3\atop (d_2,\tilde{n}_2)|q_2'd_2'n_1}\:(d_2,\tilde{n}_2)
  \left(1+\frac{M}{d_1d_2}\right)\mathop{\sum}_{\substack {{m'\sim M}
\\q_2d_2n_1+m' \tilde{n}_2\equiv 0\bmod d_2'}}|\lambda_{f}(m')|^2.
\ena
We would now like to count the sums over $(d_2',m')$ using an argument as in \cite[Sec. 6.1]{LMS}.
We first fix the parameters $(d_2,q_2,n_1,\tilde{n}_2)$.

\textit{Case 1}. If $q_2d_2n_1+ m' \tilde{n}_2\equiv 0\bmod d_2'$ but $q_2d_2n_1+ m' \tilde{n}_2\neq 0$, then
$d_2'$ is a factor of the integer $q_2d_2n_1+ m' \tilde{n}_2$. Therefore by
switching the order of summation, the $d_2'$-sum is bounded above
by $\tau(|q_2d_2n_1+ m' \tilde{n}_2|)=O(N^\varepsilon)$ with $\tau(n)$ being the divisor function, and we get
\bea\label{Omega-nonzero-1-1}
\mathbf{\Omega}_{\neq 0}^{a}
&\ll&\frac{N_1 q_1^3r}{n_1^3 C^2 M^{1/2}} \mathrm{max}\left\{t, \frac{\sqrt{N_0M}}{C}\right\}^{-1}
  \mathop{\sum\sum}_{d_1, d_1'|q_1}\:d_1d_1'
  \mathop{\sum}_{d_2\ll C/q_1\atop (d_2, d_1)=1}\:d_2
  \mathop{\sum}_{q_2\sim C/q_1d_2}\:\nonumber\\
&&\times \mathop{\sum}_{0\neq \tilde{n}_2\ll N_3}\:(d_2,\tilde{n}_2)
  \left(1+\frac{M}{d_1d_2}\right)  \mathop{\sum}_{d_2'\ll C/q_1\atop (d_2', d_1')=1}
  \mathop{\sum}_{\substack {{m'\sim M}
\\q_2d_2n_1+m' \tilde{n}_2\equiv 0\bmod d_2'}}|\lambda_{f}(m')|^2\nonumber\\
%&\ll&\frac{N^{\varepsilon}N_1 q_1^3r}{n_1^3 C^2 M^{1/2}}
%\mathrm{max}\left\{t, \frac{\sqrt{NM}}{C}\right\}^{-1}
%  \mathop{\sum\sum}_{d_1, d_1'|q_1}\:d_1d_1'
%  \mathop{\sum}_{d_2\ll C/q_1\atop (d_2, d_1)=1}\:d_2\:\\
%&&\qquad\qquad\qquad\times\mathop{\sum}_{q_2\sim C/q_1d_2}\mathop{\sum}_{0\neq n_2\ll N_3}\:(d_2,n_2)
%  \left(1+\frac{M}{d_1d_2}\right) M\:\\
  &\ll&\frac{N^{\varepsilon}N_1 q_1^2r}{n_1^3 C M^{1/2}}
  \mathrm{max}\left\{t, \frac{\sqrt{N_0M}}{C}\right\}^{-1}
  \mathop{\sum\sum}_{d_1, d_1'|q_1}\:d_1d_1'
    \mathop{\sum}_{d_2\ll C/q_1\atop (d_2, d_1)=1}\: N_3 \left(1+\frac{M}{d_1d_2}\right) M\:\nonumber\\
    &\ll&\frac{N^{\varepsilon}N_3N_1 q_1^3r}{n_1^3 C }
     \mathrm{max}\left\{t, \frac{\sqrt{N_0M}}{C}\right\}^{-1}(C+M)M^{1/2}.
     \eea
Here we have applied the Rankin--Selberg estimate \eqref{GL2: Rankin Selberg}.

\textit{Case 2}. If on the other hand, $q_2d_2n_1+ m' \tilde{n}_2=0$ identically,
%i.e., $q_2\frac{d_2}{(d_2,\tilde{n}_2)}n_1=-m'\frac{\tilde{n}_2}{(d_2,\tilde{n}_2)}$,
then as long as $m'$ and $\tilde{n}_2$ are fixed, the number of tuples $(q_2,d_2,n_1)$
is bounded above by the ternary divisor function $\tau_3(|m' \tilde{n}_2|)$. Hence such a contribution is dominated by
\bea\label{Omega-nonzero-1-2}
&&\frac{N^{\varepsilon}N_1 q_1^4r}{n_1^3 C^3 M^{1/2}}
\mathrm{max}\left\{t, \frac{\sqrt{N_0M}}{C}\right\}^{-1}
  \mathop{\sum\sum}_{d_1, d_1'|q_1}\:d_1d_1'
  \mathop{\sum}_{d_2'\ll C/q_1\atop (d_2', d_1')=1}d_2'
  \mathop{\sum}_{q_2'\sim C/q_1d_2'}
  \mathop{\sum}_{0\neq \tilde{n}_2\ll N_3}\:\nonumber\\
&&\qquad\times
  \sum_{m'\sim M}|\lambda_{f}(m')|^2  \mathop{\sum}_{d_2\ll C/q_1
  \atop (d_2,\tilde{n}_2)|q_2'd_2'n_1}(d_2,\tilde{n}_2)
   \left(d_2+\frac{M}{d_1}\right)
  \mathop{\sum}_{q_2\sim C/q_1d_2\atop q_2d_2n_1+ m' \tilde{n}_2=0}1\:\nonumber\\
&\ll&\frac{N^{\varepsilon}N_1 q_1^4r}{n_1^3 C^3 M^{1/2}}
\mathrm{max}\left\{t, \frac{\sqrt{N_0M}}{C}\right\}^{-1}
  \mathop{\sum\sum}_{d_1, d_1'|q_1}\:d_1d_1'
  \mathop{\sum}_{d_2'\ll C/q_1\atop (d_2', d_1')=1}d_2'
  \mathop{\sum}_{q_2'\sim C/q_1d_2'}\nonumber\\
&&\qquad\times \sum_{\ell|q_2'd_2'n_1}\ell
\mathop{\sum}_{0\neq n_2^*\ll N_3/\ell}\;\sum_{m'\sim M}|\lambda_{f}(m')|^2
\mathop{\sum}_{d_2\ll C/q_1\atop d_2n_1|m'n_2^*\ell}\:
\:\left(d_2+\frac{M}{d_1}\right)  \:\nonumber\\
&\ll&\frac{N^{\varepsilon}N_1 q_1^2r}{n_1^3 C M^{1/2}}
\mathrm{max}\left\{t, \frac{\sqrt{N_0M}}{C}\right\}^{-1}
  \mathop{\sum\sum}_{d_1, d_1'|q_1}\:d_1d_1'
N_3M
  \left(\frac{C}{q_1}+\frac{M}{d_1}\right) \:\nonumber\\
&\ll&\frac{N^{\varepsilon}N_3N_1 q_1^3r}{n_1^3 C }
\mathrm{max}\left\{t, \frac{\sqrt{N_0M}}{C}\right\}^{-1}
\left(C+M\right) M^{1/2}.
\eea

As for $\mathbf{\Omega}_{\neq 0}^{b}$, similarly we have
\bna
\mathbf{\Omega}_{\neq 0}^{b}&\ll&\frac{N_1 q_1^4r}{n_1^3 C^3 M^{1/2}}
\mathrm{max}\left\{t, \frac{\sqrt{N_0M}}{C}\right\}^{-1}
  \mathop{\sum\sum}_{d_1, d_1'|q_1}\:d_1d_1'
  \mathop{\sum}_{d_2\ll C/q_1\atop (d_2, d_1)=1}\:
  \mathop{\sum}_{d_2'\ll C/q_1\atop (d_2', d_1')=1}\:d_2d_2'
\\
&&\quad \times\mathop{\sum\sum}_{\substack{q_2\sim C/q_1d_2\\q_2'\sim C/q_1d_2'}}
\mathop{\sum}_{N_3\ll \tilde{n}_2\ll N_4 \atop (d_2,\tilde{n}_2)|q_2'd_2'n_1}\:
\left(\frac{N_3\mathrm{max}\left\{t, \sqrt{N_0M}/C\right\}^2}
{\lambda^3|\tilde{n}_2|} \right)^{1/3}
\\
&&\qquad \quad \times(d_2,\tilde{n}_2)
  \left(1+\frac{M}{d_1d_2}\right)\mathop{\sum}_{\substack {{m'\sim M}
\\q_2d_2n_1+m' \tilde{n}_2\equiv 0\bmod d_2'}}|\lambda_{f}(m')|^2.
\ena

As in the above, we consider two cases.

\textit{Case 1}. If $q_2d_2n_1+ m' \tilde{n}_2\equiv 0\bmod d_2'$ but $q_2d_2n_1+ m' \tilde{n}_2\neq 0$, then we have
\bea\label{Omega-nonzero-2-1}
\mathbf{\Omega}_{\neq 0}^{b}
&\ll&\frac{N_1 q_1^3r}{n_1^3 C^2 M^{1/2}\lambda} \mathrm{max}\left\{t, \frac{\sqrt{N_0M}}{C}\right\}^{-1/3}
  \mathop{\sum\sum}_{d_1, d_1'|q_1}\:d_1d_1'
  \mathop{\sum}_{d_2\ll C/q_1\atop (d_2, d_1)=1}\:d_2
  \mathop{\sum}_{q_2\sim C/q_1d_2}\:\nonumber\\
&&\times \mathop{\sum}_{N_3\ll \tilde{n}_2\ll N_4}\:\left(\frac{N_3}{|\tilde{n}_2|}\right)^{1/3}\:(d_2,\tilde{n}_2)
  \left(1+\frac{M}{d_1d_2}\right)  \mathop{\sum}_{d_2'\ll C/q_1\atop (d_2', d_1')=1}
  \mathop{\sum}_{\substack {{m'\sim M}
\\q_2d_2n_1+m' \tilde{n}_2\equiv 0\bmod d_2'}}|\lambda_{f}(m')|^2\nonumber\\
%&\ll&\frac{N^{\varepsilon}N_1 q_1^3r}{n_1^3 C^2 M^{1/2}}
%\mathrm{max}\left\{t, \frac{\sqrt{NM}}{C}\right\}^{-1}
%  \mathop{\sum\sum}_{d_1, d_1'|q_1}\:d_1d_1'
%  \mathop{\sum}_{d_2\ll C/q_1\atop (d_2, d_1)=1}\:d_2
%  \mathop{\sum}_{q_2\sim C/q_1d_2}\:\nonumber\\
%&\times&\mathop{\sum}_{N_3\ll \tilde{n}_2\ll N_2}\:\left(\frac{N_3}{|\tilde{n}_2|}\right)^{1/2}(d_2,\tilde{n}_2)
%  \left(1+\frac{M}{d_1d_2}\right) M\:\nonumber\\
  &\ll&\frac{N^{\varepsilon}N_1 q_1^2r}{n_1^3 C M^{1/2}\lambda}
  \mathrm{max}\left\{t, \frac{\sqrt{N_0M}}{C}\right\}^{-1/3}
  \mathop{\sum\sum}_{d_1, d_1'|q_1}\:d_1d_1'
    \mathop{\sum}_{d_2\ll C/q_1\atop (d_2, d_1)=1}\:
   N_4^{2/3}N_3^{1/3} \left(1+\frac{M}{d_1d_2}\right) M\:\nonumber\\
       &\ll&\frac{N^{\varepsilon}N_4^{2/3}N_3^{1/3}N_1 q_1^3r}{n_1^3 C \lambda}
     \mathrm{max}\left\{t, \frac{\sqrt{N_0M}}{C}\right\}^{-1/3}(C+M)M^{1/2}.
     \eea

 \textit{Case 2}. If $q_2d_2n_1+ m' \tilde{n}_2=0$ identically, then as before,
such a contribution is also dominated by \eqref{Omega-nonzero-2-1}.

For $\mathbf{\Omega}_{\neq 0}^{c}$, we discuss analogously as in the above and obtain
\bea\label{Omega-nonzero-2-2}
\mathbf{\Omega}_{\neq 0}^{c}&\ll&\frac{N_1 q_1^4r}{n_1^3 C^3 M^{1/2}}
\mathrm{max}\left\{t, \frac{\sqrt{N_0M}}{C}\right\}^{-1}
  \mathop{\sum\sum}_{d_1, d_1'|q_1}\:d_1d_1'
  \mathop{\sum}_{d_2\ll C/q_1\atop (d_2, d_1)=1}\:
  \mathop{\sum}_{d_2'\ll C/q_1\atop (d_2', d_1')=1}\:d_2d_2'
\nonumber\\
&&\quad \times\mathop{\sum\sum}_{\substack{q_2\sim C/q_1d_2\\q_2'\sim C/q_1d_2'}}
\mathop{\sum}_{N_4\ll \tilde{n}_2\ll N_2 \atop (d_2,\tilde{n}_2)|q_2'd_2'n_1}\:
\left(\frac{N_3\mathrm{max}\left\{t, \sqrt{N_0M}/C\right\}^2}
{\lambda^3|\tilde{n}_2|} \right)^{1/2}
\nonumber\\
&&\qquad \quad \times(d_2,\tilde{n}_2)
  \left(1+\frac{M}{d_1d_2}\right)\mathop{\sum}_{\substack {{m'\sim M}
\\q_2d_2n_1+m' \tilde{n}_2\equiv 0\bmod d_2'}}|\lambda_{f}(m')|^2\nonumber\\
&\ll&\frac{N^{\varepsilon}N_3^{1/2}N_2^{1/2}N_1 q_1^3r}{n_1^3 C \lambda^{3/2}}(C+M)M^{1/2}.
\eea

Recall  $N_2=N^{\varepsilon}\lambda  C^2n_1r/N_1q_1$ in \eqref{N2 range}, and
$N_3=N^{\varepsilon}\max\left\{t,\sqrt{N_0M}/C\right\}^{-2}\lambda^3C^2n_1r/N_1q_1$,
$N_4=N^{\varepsilon}\max\left\{t,\sqrt{N_0M}/C\right\}^{-1}\lambda^2C^2n_1r/N_1q_1$ in \eqref{N3N4},
with $\lambda=\left(N_0N_1/C^3r\right)^{1/3}$.
By comparing, the bound in \eqref{Omega-nonzero-2-1}
is dominated by that in \eqref{Omega-nonzero-1-2} if $\lambda\geq \max\left\{t,\sqrt{N_0M}/C\right\}^{4/5}$ and by the one in \eqref{Omega-nonzero-2-2} if
$\lambda\leq \max\left\{t,\sqrt{N_0M}/C\right\}^{4/5}$.

Consequently,
from \eqref{Omega-nonzero-1-1}-\eqref{Omega-nonzero-2-2}, we conclude that
\bea\label{compare}
\mathbf{\Omega}^{+}_{\neq 0}
  &\ll&\frac{N^{\varepsilon}N_1 q_1^3r}{n_1^3 C}\left(
    N_3^{1/2}N_2^{1/2}\lambda^{-3/2}+ N_3\max\left\{t,\sqrt{N_0M}/C\right\}^{-1}
     \right)(C+M)M^{1/2}\nonumber\\
&\ll& \frac{N^{\varepsilon}r^{11/6}N_1^{1/6}N_0^{1/6} q_1^2 C^{1/2}}{n_1^2 t}
  \left(1+\frac{1}{t^2}\left(\frac{N_0N_1}{C^3r}\right)^{5/6}\right)(C+M)M^{1/2}.
\eea
Since $1\ll M\ll N^{\varepsilon}\max\left\{C^2t^2/N_0,N_0/Q^2\right\}= N^{\varepsilon}C^2t^2/N_0$
in \eqref{M-N1-range} as $C$ satisfies \eqref{C range-large modulus}, this bound when substituted in place of $\mathbf{\Omega}(n_1,q_1,r)$ in \eqref{Cauchy}
gives that
\bna
\mathbf{\Sigma}^{+}_{\neq 0}
&\ll&\frac{N_0^{5/12+\varepsilon}N_1^{1/6}}{r^{2/3}}
\left(1+\frac{1}{t}\left(\frac{N_0N_1}{C^3r}\right)^{5/12}\right)
r^{11/12} N_1^{1/12}N_0^{1/12}C^{1/4}t^{-1/2}  (C+M)^{1/2}M^{1/4} \\
&\ll&r^{1/4}N_0^{1/2+\varepsilon}N_1^{1/4}C^{1/4}(C+M)^{1/2}M^{1/4}
 t^{-1/2}\left(1+\frac{1}{t}\left(\frac{N_0N_1}{C^3r}\right)^{5/12}\right)\\
&\ll& r^{1/4}N_0^{1/4+\varepsilon}N_1^{1/4}\left(C^{5/4}+\frac{C^{7/4}t}{N_0^{1/2}}\right)
 \left(1+\frac{1}{t}\left(\frac{N_0N_1}{C^3r}\right)^{5/12}\right).
\ena
Recall $1\ll N_1\ll N_0^{2+\varepsilon}r/Q^3$ in \eqref{M-N1-range} and $C\ll Q$, we further imply
\bna
\mathbf{\Sigma}^{+}_{\neq 0}&\ll&\frac{r^{1/2}N_0^{1/4+\varepsilon}}{Q^{3/4}}
\left(C^{5/4}N_0^{1/2}+C^{7/4}t\right)
  \left(1+\frac{1}{t}\left(\frac{N_0}{CQ}\right)^{5/4}\right)\\
  &\ll&r^{1/2}N_0^{1/4+\varepsilon}\left(Q^{1/2}N_0^{1/2}+Qt\right)
  \left(1+\frac{1}{t}\left(\frac{N_0}{Q^2}\right)^{5/4}\right).
\ena
%\bna
%\mathbf{\Sigma}_{\neq 0}&\ll& \frac{r^{1/2}N^{1+\varepsilon}}{t^{1/2}Q^{1/2}}(Q+M)^{1/2}M^{1/4}\\
%&\ll&\frac{r^{1/2}N^{1+\varepsilon}}{t^{1/2}Q^{1/2}}
%  \left(Q^{1/2}\left(\frac{Q^{1/2}t^{1/2}}{N^{1/4}}+\frac{N^{1/4}}{Q^{1/2}}\right)
%  +\frac{Q^{3/2}t^{3/2}}{N^{3/4}}+\frac{N^{3/4}}{Q^{3/2}}\right)\\
%&\ll&\frac{r^{1/2}N^{1+\varepsilon}}{t^{1/2}Q^{1/2}}
%  \left(\frac{Qt^{1/2}}{N^{1/4}}+N^{1/4}
%  +\frac{Q^{3/2}t^{3/2}}{N^{3/4}}+\frac{N^{3/4}}{Q^{3/2}}\right).
%\ena
If we assume $Q>\max\{(N_0/t)^{1/2},N_0/t^2\}$, then
\bna
\mathbf{\Sigma}^{+}_{\neq 0}\ll r^{\frac{1}{2}}N_0^{1/4+\varepsilon}Qt
\left(1+\frac{N_0^{5/4}}{Q^{5/2}t}\right).
\ena
The lemma follows.
\end{proof}

\begin{remark}\label{second-derivative-0-large}
If $(\varphi(x^3))''=0$, Lemma 4.3 holds with the factor $\max\left\{t,\sqrt{N_0M}/C\right\}$
replaced by $\sqrt{N_0M}/C$. By going through the proof, a careful reader will see that the same estimates for $\mathbf{\Sigma}^{\pm}_0$ and $\mathbf{\Sigma}^{+}_{\neq 0}$
in Lemma \ref{lemma:zero} and Lemma \ref{lemma:nonzero} are still valid.
\end{remark}

\subsubsection{The case of small modulus}\label{The case of small modulus}
In this subsection, we deal with the case where $C$ is such that
\bna
1\ll C\leq N_0^{1+\varepsilon}/Qt,
\ena
and accordingly we bound the $\mathbf{\Omega}^{\pm}_{\neq 0}$ in \eqref{large-and-small-modulus}.

% \textit{Case 1: the non-zero frequencies for small modulus.}

The proof for this case is similar to what we have done in
Section \ref{The case of large modulus}, except that we will only use
the bound \eqref{J-for-smallC} for $\mathfrak{I}^{\pm}(X)$ in place of the
statements (2), (3) and (4) of
Lemma \ref{integral:lemma}, and we will be brief.

Recall from \eqref{M-N1-range} $M$ satisfies
$1\ll M\ll N_0^{\varepsilon}\max\left\{C^2t^2/N_0,N_0/Q^2\right\}$, then in this case we have
\bna
1\ll M\ll N_0^{1+\varepsilon}/Q^2.
\ena

By replicating the proof of Lemma \ref{lemma:nonzero} and
using \eqref{J-for-smallC}, one finds that (in comparison with \eqref{compare})
\bna
 \mathbf{\Omega}^{\pm}_{\neq 0}&\ll&\frac{N^{\varepsilon}N_1 q_1^3r}{n_1^3 C }
     \left(N_2\mathrm{max}\left\{t, \lambda\right\}^{-1}\right)(C+M)M^{1/2}\\
 &\ll&  \frac{N^{\varepsilon} r^{5/3}N_1^{1/3}N_0^{1/3}q_1^2}
 {n_1^2\mathrm{max}\left\{t, \lambda\right\}}(C+M)M^{1/2}.
\ena
For $M\ll N_0^{1+\varepsilon}/Q^2$ and $C\ll N_0^{1+\varepsilon}Q^{-1}t^{-1}$, the contribution of $ \mathbf{\Omega}^{\pm}_{\neq 0}$ to \eqref{Cauchy} is bounded as follows
\bna
\mathbf{\Sigma}^{\pm}_{\neq 0}
&\ll&\frac{N_0^{5/12+\varepsilon}N_1^{1/6}}{r^{2/3}t^{1/2}}
r^{5/6} N_1^{1/6}N_0^{1/6}(C^{1/2}+M^{1/2})M^{1/4} \\
&\ll& r^{1/6}N_0^{7/12+\varepsilon}N_1^{1/3}
\left(\frac{N_0^{1/2}}{Q^{1/2}t}+\frac{N_0^{1/2}}{Qt^{1/2}}\right)
 \frac{N_0^{1/4}}{Q^{1/2}}\\
&\ll&r^{1/2}N_0^{1/4+\varepsilon}Qt
\left(\frac{N_0^{7/4}}{Q^{3}t^2}+\frac{N_0^{7/4}}{Q^{7/2}t^{3/2}}\right).
\ena
Under the assumption \eqref{assumption: range 2}, this is dominated by the bound in
\eqref{large bound}.

% \textit{Case 2: the zero frequency for small modulus.}
%
%Recall we denoted by $\mathbf{\Sigma}^{\pm}_{0}$ the contribution from $\tilde{n}_2=0$
%to $\mathbf{M}_r^{\pm}$ in \eqref{Cauchy}. We have shown in \eqref{estimate-1} that
%\bna
%\mathbf{\Sigma}^{\pm}_{0}\ll \frac{N_0^{5/12+\varepsilon}}{r^{1/6}}
%\left(\frac{M^{1/4}N_1^{2/3}}{t^{1/2}}+
%\frac{M^{1/2}C^{1/2}r^{1/6}N_1^{1/2}}{N_0^{5/12}}\right).
%\ena
%Since $M\ll N_0^{1+\varepsilon}/Q^2$ and $1\ll N_1\ll N_0^{2+\varepsilon}r/Q^3$ in \eqref{M-N1-range}, if the modulus $C$ satisfies $C\ll N_0^{1+\varepsilon}/Qt$, then we obtain
%\bna
%\mathbf{\Sigma}^{\pm}_{0}
%\ll N_0^{\varepsilon}r^{1/2}
%\left(\frac{N_0^{7/4}M^{1/4}}{Q^{2}t^{1/2}}+\frac{N_0^{3/2}M^{1/2}}{Q^2t^{1/2}}\right)
%\ll \frac{r^{1/2}N_0^{2+\varepsilon}}{Q^{5/2}t^{1/2}}.
%\ena
%Under the assumption \eqref{assumption: range 2} on $Q$, this is dominated by the bound in
%\eqref{Sigma-0-bound}.

\subsection{Completion of the proof}
For the $\mathbf{M}_r^{\pm}(\Xi)$ in \eqref{Xi:range}, by inserting \eqref{Sigma-0-bound} and \eqref{large bound} into \eqref{Cauchy}, we have shown the following
\bna
\mathbf{M}_r^{\pm}(\Xi)\ll \frac{r^{1/2}N_0^{3/2+\varepsilon}}{Q^{3/2}}+
r^{\frac{1}{2}}N_0^{1/4+\varepsilon}Qt
\ena
under the assumption
\bea
\label{final-assumption-Q}
\max\{(N_0/t)^{1/2},N_0/t^2\}<Q<\min\{N_0^{1/2-\varepsilon},(N_0/t)^{2/3}\}
\eea
which is a combination of \eqref{assumption 1}, \eqref{assumption: range 1} and
\eqref{assumption: range 2}. We set $Q=N_0^{1/2}/t^{2/5}$ to balance the contribution. Then
\bna
\mathbf{M}_r^{\pm}(\Xi)\ll r^{1/2}t^{3/5}N_0^{3/4+\varepsilon}
\ena
provided $t^{8/5}<N_0<t^{16/5}$.

With the choice $Q=N_0^{1/2}/t^{2/5}$, substituting the bound above for $\mathbf{M}_r^{\pm}(\Xi)$ and
\eqref{small part-1} into \eqref{M-N1-range}, we obtain
\bna
\begin{split}
 \mathscr{S}_r(N_0)\ll&\, r^{1/2}t^{3/5}N_0^{3/4+\varepsilon}
+r^{1/2}t^{-9/10}N_0^{5/4+\varepsilon}
\end{split}\ena
for the sum $\mathscr{S}_r(N)$ defined in \eqref{aim-sum}, provided $t^{8/5}<N_0<t^{16/5}$.

%
%\bna
%\begin{split}
% \mathscr{S}_r(N_0)\ll& r^{1/2}t^{3/5}N_0^{3/4+\varepsilon}\\
%&+r^{1/2}t^{-9/10}N_0^{5/4+\varepsilon}+{\color{blue} t^{11/5+\varepsilon}+t^{9/5}N_0^{1/4+\varepsilon}  }
%\end{split}\ena

%If we assume $t^{8/5}<N_0<t^{3+\varepsilon}$,
%then the factor $r^{1/2}t^{-9/10}N_0^{5/4+\varepsilon}$ is dominated by the bound $O(r^{1/2}t^{3/5}N_0^{3/4+\varepsilon})$, so is the last two factors, if we assume $\max\{r^{-2/3}t^{32/15},r^{-1}t^{12/5}\}<N_0$. Hence, we obtain
%\bna
% \mathscr{S}_r(N_0)\ll r^{1/2}t^{3/5}N_0^{3/4+\varepsilon},
%\ena
%under the assumption $\max\{r^{-2/3}t^{32/15},r^{-1}t^{12/5}, t^{8/5}\}<N_0<t^{3+\varepsilon}$.

By replacing $N_0$ with $N/r^2$ and summing over $1\leq r\leq t^{\gamma}$ (with $\gamma<4/5$), we get
\bna
\begin{split}
\mathscr{S}(N) \ll&
\sum_{1\leq r\leq t^{\gamma}}\mathscr{S}_r(N/r^2)
+ t^{-9\gamma/14} N^{1+\varepsilon}\\
\ll&\, t^{3/5}N^{3/4+\varepsilon}+t^{-9/10}N^{5/4+\varepsilon}+ t^{-9\gamma/14} N^{1+\varepsilon}
\end{split}\ena
provided $t^{8/5+2\gamma}<N<t^{16/5}$. Here the error term $O(t^{-9\gamma/14} N^{1+\varepsilon})$ was estimated in \eqref{error}. We also notice that the first factor $t^{3/5}N^{3/4+\varepsilon}$ is better than the trivial bound $\mathscr{S}(N)=O(N)$ if $t^{12/5+\varepsilon}<N$.

Notice that by \eqref{first-derivative-varphi}, \eqref{varphi condition}, Remark \ref{second-derivative-0} and Remark
\ref{second-derivative-0-large} (which take care of the case where $(\varphi(x^3))''=0$), $\varphi(x)$ satisfies $\varphi'(x)\cdot(\varphi(x^3))''\leq 0$, while in using Lemma \ref{integral:lemma} (3), we further assume $\varphi(x)$
is of the form $\varphi(x)=c\log x$ or $cx^{\beta}$ ($\beta\neq 1/2$, $c\in \mathbb{R}\backslash \{0\}$); see \eqref{phi assumption}. These imply that the function $\varphi(x)$ satisfies $\varphi(x)=c\log x$ or $cx^{\beta}$ with $0<\beta\leq 1/3$. The assumption $\triangle<t^{1/2-\varepsilon}$ arises also in applying Lemma \ref{integral:lemma} (3); see \eqref{assumption-on-Delta}.
This completes the proof of Theorem \ref{main-theorem}.

\section{Proofs of Lemma \ref{lem: character sum} and Lemma \ref{integral:lemma}}\label{proofs-of-technical-lemmas}
\subsection{Proof of Lemma \ref{lem: character sum}}
The estimate has in fact been carried out by Munshi in \cite{Mun6}, and we will follow his treatment closely.
\begin{proof}
Using the expression $\sum_{d|(n,q)}d\mu(q/d)$ for the Ramanujan sum $S(n,0;q)$, the character sum in
\eqref{character sum} is
\bna
\mathfrak{C}_r(n_1,n_2,m,q)
=\sum_{d | q }d\mu\left(\frac{q}{d}\right)
\sideset{}{^{\star}}\sum_{\alpha\bmod qr/n_1\atop n_1\alpha \equiv -m \bmod d}
e\left(\frac{ n_2\overline{\alpha}}{qr/n_1}\right).
\ena
Then we have
\bea\label{K-expression}
\mathfrak{K}(\tilde{n}_2)
=\sum_{d|q_1q_2}d\mu\left(\frac{q_1q_2}{d}\right)
\sum_{d'|q_1q_2'}
d'\mu\left(\frac{q_1q_2'}{d'}\right)
\mathop{\sideset{}{^\star}\sum_{\alpha\, \mathrm{mod} \,q_1q_2r/n_1 \atop n_1\alpha\equiv -m\bmod d}\;
\sideset{}{^\star}\sum_{\alpha'\, \mathrm{mod} \,q_1q_2'r/n_1 \atop n_1\alpha'\equiv -m'\bmod d'}}
_{q_2\overline{\alpha'}-q_2'\overline{\alpha}\equiv \tilde{n}_2 \bmod q_2q_2'q_1r/n_1}1.
\eea
Let $d=d_1d_2,d_1|q_1,d_2|q_2$ and $d'=d_1'd_2',d_1'|q_1,d_2'|q_2.$
We then factor $\mathfrak{K}_r(\tilde{n}_2)$ as
\bna
\mathfrak{K}(\tilde{n}_2)= \mathfrak{K}_1(\tilde{n}_2)\mathfrak{K}_2(\tilde{n}_2),
\ena
where
\bna
\mathfrak{K}_1(\tilde{n}_2)
&=&\sum_{d_1|q_1}\sum_{d_1'|q_1}d_1d_1'\mu\left(\frac{q_1}{d_1}\right)
\mu\left(\frac{q_1}{d_1'}\right)
\mathop{\mathop{\sideset{}{^\star}\sum}_{\substack{\alpha_1\bmod{q_1r/n_1}\\
n_1\alpha_1\equiv -m\bmod{d_1}}}\;\mathop{\sideset{}{^\star}\sum}_{\substack{\alpha_1'\bmod{q_1r/n_1}\\
n_1\alpha_1'\equiv -m'\bmod{d_1'}}}}_{q_2\bar{\alpha}'_1-q_2'\bar{\alpha}_1 \equiv  \tilde{n}_2\bmod{q_1r/n_1}}\; 1,\\
\mathfrak{K}_2(\tilde{n}_2)&=&\sum_{d_2|q_2}\sum_{d_2'|q_2'}d_2d_2'\mu\left(\frac{q_2}{d_2}\right)\mu\left(\frac{q_2'}{d_2'}\right)
\mathop{\mathop{\sideset{}{^\star}\sum}_{\substack{\alpha_2\bmod{q_2}\\
n_1\alpha_2\equiv -m\bmod{d_2}}}\;\mathop{\sideset{}{^\star}\sum}_{\substack{\alpha_2'\bmod{q_2'}\\
n_1\alpha_2'\equiv -m'\bmod{d_2'}}}}_{q_2\bar{\alpha}'_2-q_2'\bar{\alpha}_2 \equiv
 \tilde{n}_2\bmod{q_2q_2'}}\; 1.
\ena
For $\mathfrak{K}_2(\tilde{n}_2)$, since $(n_1,q_2q_2')=1$, we get
$\alpha\equiv -m\bar{n}_1\bmod{d_2}$ and $\alpha'\equiv -m'\bar{n}_1\bmod{d_2'}$.
Thus $d_2'|q_2n_1+m'\tilde{n}_2$ and $d_2|q_2'n_1- m\tilde{n}_2$. We are able to conclude that
\bna
\mathfrak{K}_2(\tilde{n}_2)\ll \mathop{\sum\sum}_{\substack{d_2|(q_2,q_2'n_1- m\tilde{n}_2)\\
d_2'|(q_2',q_2n_1+m'\tilde{n}_2)}}d_2d_2'.
\ena

For $\tilde{n}_2=0$, the claim follows from the congruence conditions in \eqref{K-expression}.
This completes the proof of the lemma.
\end{proof}

\subsection{Proof of Lemma \ref{integral:lemma}}
Recall from \eqref{I-integral},
\bna
\mathfrak{I}^{\pm}(X)=\int_{\mathbb{R}}
\phi\left(\xi\right)
\mathfrak{J}_r^{\pm}\left(N_1\xi,m,q_1q_2\right)
\overline{\mathfrak{J}_r^{\pm}\left(N_1\xi,m',q_1q_2'\right)}
\, e\left(-X\xi\right)\mathrm{d}\xi,
\ena
where the $\mathfrak{J}_r^{\pm}(N_1\xi,m,q)$, by making a change of variable $y=z^3$ in \eqref{J-definition}, is given by
\bea\label{J-integral}
\mathfrak{J}_r^{\pm}(N_1\xi,m,q)=3\int_0^\infty \widetilde{V}(z^3)z^2
e\left(t\varphi(z^3)\pm \frac{2}{q}(N_0m)^{1/2}z^{3/2}
+\frac{3}{q}\left(\frac{N_0N_1\xi}{r}\right)^{1/3}z\right)
\mathrm{d}z.
\eea
Before finer analysis for the integral $\mathfrak{I}^{\pm}(X)$, we record a direct estimate for $\mathfrak{J}_r^{+}(N_1\xi,m,q)$. Under the assumption \eqref{varphi condition}, the second derivative of the phase function in $\mathfrak{J}_r^{+}(N_1\xi,m,q)$
satisfies
\bna
t\left(\varphi(z^3)\right)''
+\frac{3}{2q}(N_0m)^{1/2}z^{-1/2}\gg \max\{t,\sqrt{N_0M}/C\}.
\ena
Then, by the second derivative test in Lemma \ref{lem: 2st derivative test, dim 1}, we have
\bea\label{J:second derivative estimate}
\mathfrak{J}_r^+(N_1\xi,m,q)\ll  \max\{t,\sqrt{N_0M}/C\}^{-1/2}.
\eea

To estimate the integral $\mathfrak{I}^{\pm}(X)$, we first prove a preliminary result.
Recall $\lambda:=\left(N_0N_1/C^3r\right)^{1/3}$ (see \eqref{lambda}).
Note that with our assumption on $N_1$ in \eqref{generic-N1}, $\lambda\gg N_0^\varepsilon$. Let $\phi\in  \mathcal{C}_c^{\infty}(2/3,3)$ be a smooth function.
\begin{lemma}\label{middle integral:lemma}
For $X,Y$ real with $|Y|\ll 1$, define
\bna
H\left(X,3\lambda Y\right)=\int_{\mathbb{R}} \xi^2\phi\left(\xi^3\right) e\left(-X\xi^3+3\lambda Y\xi\right)\mathrm{d}\xi.
\ena

(1) We have $H(X,3\lambda Y)=O\left(N_0^{-A}\right)$ if $X\gg \lambda^{1+\varepsilon}$.

(2) For $X\gg N_0^\varepsilon$, we have $H(X,3\lambda Y)=O\left(N_0^{-A}\right)$
unless $\sqrt[3]{4/9}-1/6<\lambda Y/X<\sqrt[3]{9}+1/6$.
For $1/3<\lambda Y/X<3$, if we let $\nu=2X\left(\lambda Y/X\right)^{3/2}$,
then
\bna
H(X,3\lambda Y)=|X|^{-1/2}e(\nu)H_1(\nu)+O_A\left(N_0^{-A}\right)
\ena
for any $A>0$, with $H_1(\nu)\in C_c^{\infty}(X/3,6X)$ satisfying
\bna
\nu^jH_1^{(j)}(\nu)\ll 1.
\ena

(3) $H(0,3\lambda Y)=\phi_0(-3\lambda Y)$ for some Schwartz function $\phi_0$.
\end{lemma}

\begin{proof}
(1) Let $h(\xi)=-X\xi^3+3\lambda Y\xi$. Then $h'(\xi)=-3X\xi^2+3\lambda Y$ and $h''(\xi)=-6X\xi$.
If $X\gg\lambda^{1+\varepsilon}$, we have
\bna
h'(\xi)\gg\lambda^{1+\varepsilon}\gg N_0^\varepsilon.
\ena
Thus we have $H(X,3\lambda Y)\ll N_0^{-A}$ for $X\gg\lambda^{1+\varepsilon}$,
by applying Lemma \ref{lem: upper bound}.

(2) For $X\gg N_0^\varepsilon$, we have
$|-3X\xi^2+3\lambda Y|\geq X/2$
if $\lambda Y/X\leq \xi^2-1/6$
or $\lambda Y/X\geq \xi^2+1/6$. Since $\phi\in  \mathcal{C}_c^{\infty}(2/3,3)$, $\xi^3\in[2/3,3]$.
Thus we have $H(X,3\lambda Y)\ll N_0^{-A}$ unless $\sqrt[3]{4/9}-1/6<\lambda Y/X<\sqrt[3]{9}+1/6$.

For $1/3<\lambda Y/X<3$, by the stationary phase analysis in Lemma \ref{lemma:exponentialintegral} (2),
we have
\bna
H\left(X,\lambda Y\right)=\frac{e\left(2X(\lambda Y/X)^{3/2}\right)}
{\sqrt{-12\pi (\lambda XY)^{1/2}}}
\mathcal{F}\left(\sqrt{\frac{\lambda Y}{X}}\right)+O_{A}(N_0^{-A})
\ena
for some $1$-inert function $\mathcal{F}(y)$ supported on $y\asymp 1$.
Set  $\nu=2X\left(\lambda Y/X\right)^{3/2}$ and
$$
H_1(\nu)=\frac{\sqrt{|X|}}
{\sqrt{-12\pi X(\nu/2X)^{1/3}}}
\mathcal{F}\left(\left(\frac{\nu}{2X}\right)^{1/3}\right).
$$
Then (2) follows.

(3) We have
\bna
H(0,3\lambda Y)=\int_{\mathbb{R}}\xi^2\phi\left(\xi^3\right) e(3\lambda Y\xi)\mathrm{d}\xi
\ena
which is the Fourier transform of $\xi^2\phi\left(\xi^3\right)$ at $-3\lambda Y$, so (3) is clear.
\end{proof}
Now we are ready to prove Lemma \ref{integral:lemma}.

Before proceeding, we recall $q, q_1q_2, q_1q'_2\asymp C$.
\subsection*{Proof of Lemma \ref{integral:lemma}}
(1) By \eqref{J-integral} and \eqref{I-integral}, we have
\bea\label{I integral-median}
\mathfrak{I}^{\pm}(X)
&=&9\int_\mathbb{R}\int_\mathbb{R}\widetilde{V}(z_1^3)
\widetilde{V}(z_2^3)z_1^2z_2^2
\int_{\mathbb{R}} \phi\left(\xi\right) e\left(3\left(\frac{N_0N_1}{r}\right)^{1/3}
\left(\frac{z_1}{q_1q_2}-\frac{z_2}{q_1q_2'}\right)\xi^{1/3}-X\xi\right)\mathrm{d}\xi\nonumber\\
   && \quad
   \times e\left(t\left(\varphi(z_1^3)-\varphi(z_2^3)\right)\pm 2N_0^{1/2}
      \left(\frac{m^{1/2}z_1^{3/2}}{q_1q_2}-
      \frac{m'^{1/2}z_2^{3/2}}{q_1q_2'}\right)\right)
   \mathrm{d}z_1\mathrm{d}z_2\nonumber\\
   &=&27\int_\mathbb{R}\int_\mathbb{R}\widetilde{V}(z_1^3)
\widetilde{V}(z_2^3)z_1^2z_2^2
H\left(X, 3\lambda\left(\frac{Cz_1}{q_1q_2}-\frac{Cz_2}{q_1q_2'}\right)\right)\nonumber\\
   && \quad
   \times e\left(t\left(\varphi(z_1^3)-\varphi(z_2^3)\right)\pm 2N_0^{1/2}
      \left(\frac{m^{1/2}z_1^{3/2}}{q_1q_2}-
      \frac{m'^{1/2}z_2^{3/2}}{q_1q_2'}\right)\right)
   \mathrm{d}z_1\mathrm{d}z_2.
\eea
Thus the statement in (1) is clear in view of Lemma \ref{middle integral:lemma} (1).

(2) For $X\ll \lambda^{3+\varepsilon}\max\left\{t,\sqrt{N_0M}/C\right\}^{-2}$, the bound follows by plugging the second derivative bound \eqref{J:second derivative estimate}
into \eqref{I-integral}.

(3) In this case, we derive similarly as in Munshi \cite{Mun6}. From the definition \eqref{J-definition},
\bna
\mathfrak{J}_r^+(N_1\xi,m,q)
=\int_0^\infty \widetilde{V}(y)
e\left(t\varphi(y)+2Dy^{1/2}
+3By^{1/3}\right)
\mathrm{d}y,
\ena
where $\xi\in [2/3,3]$,
\bea\label{AB}
D=\frac{1}{q}(N_0m)^{1/2}\asymp \sqrt{N_0M}/C, \qquad
B=\frac{1}{q}\left(\frac{N_0N_1\xi}{r}\right)^{1/3}\asymp \lambda.
\eea
Recall the range of $N_1$ in \eqref{M-N1-range}: $1\ll N_1\ll N_0^{2+\varepsilon}r/Q^3$
and $q\sim C$. For $C$ in \eqref{C range-large modulus}, we have
\bea\label{B upper bound}
B\ll \frac{N_0^{1+\varepsilon}}{CQ}\ll N_0^{-\varepsilon}t.
\eea
Therefore, the integral $\mathfrak{J}_r^+(N_1\xi,m,q)$ is negligibly small unless
$D\asymp t$.

Assume
\bea\label{phi assumption}
\varphi'(y)=-cy^{-\beta} \qquad \text{with}\quad \beta>0, \quad \beta\neq 1/2,
\eea
where $c>0$ is an absolute constant, that is,
\bea\label{phi assumption-2}
\varphi(y)=-c\log y+c_0\qquad \text{or}\qquad \varphi(y)=-\frac{c}{1-\beta}y^{1-\beta}+c_0
\quad \text{with}\; \beta\neq 1/2, 1,
\eea
where $c_0\in \mathbb{R}$ is an absolute constant. Without loss of generality, we further assume
$c_0=0$.
Let $g(y)=t\varphi(y)+2Dy^{1/2}+3By^{1/3}$.
Then
\bna
g'(y)&=&-cty^{-\beta}+Dy^{-1/2}+By^{-2/3},\\
g''(y)&=&c\beta ty^{-\beta-1}-\frac{D}{2}y^{-3/2}-\frac{2B}{3}y^{-5/3},\\
g^{(j)}(y)&=&t\varphi^{(j)}(y)+2D\cdot \frac{1}{2}\cdots\left(\frac{3}{2}-j\right)
y^{\frac{1}{2}-j}
+3B\cdot \frac{1}{3}\cdots\left(\frac{4}{3}-j\right)y^{\frac{1}{3}-j}.
\ena
Denote
\bna
C_{\alpha}^j=\frac{\alpha(\alpha-1)\cdot\cdot\cdot (\alpha-j+1)}{j!}.
\ena
By an iterative argument, the stationary point $y_*$ which is the solution to
the equation $g'(y)=-cty^{-\beta}+Dy^{-1/2}+By^{-2/3}=0$ can be written as
\bea\label{stationary point}
y_*&=&\left(\frac{ct}{D}\right)^{\frac{2}{2\beta-1}}\left(1-\frac{B}{ct}
y_*^{\beta-2/3}\right)^{\frac{2}{2\beta-1}}\nonumber\\
&=&\left(\frac{ct}{D}\right)^{\frac{2}{2\beta-1}}+
\left(\frac{ct}{D}\right)^{\frac{2}{2\beta-1}}  \bigg( \sum_{j=1}^{K_1}
C_{2/(2\beta-1)}^{j}\left(\frac{-B}{ct}\right)^jy_*^{\left(\beta-2/3\right)j}
+O_{c,\beta,K_1}\left(\frac{B^{K_1+1}}{t^{K_1+1}}\right) \bigg)  \nonumber\\
&=&y_0\left(1+\sum_{j=1}^{K_2}\frac{y_{j}}{y_0}
+O_{c,\beta,K_2}\left(\frac{B^{K_2+1}}{t^{K_2+1}}\right)\right),
\eea
where here and after, $K_j\geq 1$, $j=1,2,3\ldots,$ denote integers, and
\bna
y_0&=&\left(\frac{ct}{D}\right)^{\frac{2}{2\beta-1}} \asymp 1,\\
y_1&=&C_{2/(2\beta-1)}^{1}y_0^{\beta+1/3}
\left(\frac{-B}{ct}\right)\asymp \frac{B}{t},\\
y_{j}&=&f_{\beta,j}\left(y_0\right)\left(\frac{-B}{ct}\right)^{j}
\asymp \left(\frac{B}{t}\right)^{j}
\ena
for some function $f_{\beta,j}(x)$ of polynomially growth, depending only on
$\beta,j$, and supported on $x\asymp 1$. By \eqref{B upper bound}, the
$O$-term in \eqref{stationary point} is $O(N_0^{-\varepsilon K_2 })$, which
can be arbitrarily small by taking $K_2$ sufficiently large.

Note that $g''(y)\asymp\max\{t,\sqrt{N_0M}/C\}$ and $g^{(j)}(y)\ll \max\{t,\sqrt{N_0M}/C\}$
for $j\geq 1$. Recall $\widetilde{V}^{(j)}(y)\ll_j \triangle^j$, where $\triangle<t^{1-\varepsilon}$ (see \eqref{derivative-of-V} and also one line after \eqref{J-definition}). To make sure that the stationary phase analysis is applicable to the integral $\mathfrak{J}_r^+(N_1\xi,m,q)$, we assume $\triangle$ satisfies
\bea\label{assumption-on-Delta}
\triangle<t^{1/2-\varepsilon}
\eea
Now applying Lemma \ref{lemma:exponentialintegral} with $Z=1$, $X=\triangle$,
$Y=\max\{t,\sqrt{N_0M}/C\}$ and $R=Y/X^2\gg t^{\varepsilon}$, we have
\bna
\mathfrak{J}_r^+(N_1\xi,m,q)
=\frac{e(g(y_*))}{\sqrt{2\pi g''(y_*)}}
  \mathcal{G}_A(y_*) + O_{A}(  t^{-A}),
\ena
for any $A>0$, where $\mathcal{G}_A(y)$ is some inert function
supported on $y\asymp 1$. From \eqref{phi assumption-2} and \eqref{stationary point} and using Taylor series approximation, we have
\bna
g(y_*)&=&t\varphi(y_*)+2Dy_*^{1/2}+3By_*^{1/3}\\
&=&t\varphi(y_0)+2Dy_0^{1/2}+B\sum_{j=0}^{K_3}g_{c,\beta,j}\left(y_0\right)
\left(\frac{B}{t}\right)^j+O_{c,\beta,K_3}\left(\frac{B^{K_3+2}}{t^{K_3+1}}\right)
\ena
and
\bna
g''(y_*)&=&c\beta ty_*^{-\beta-1}-\frac{D}{2}y_*^{-3/2}-\frac{2B}{3}y_*^{-5/3}\\
&=&c\beta ty_0^{-\beta-1}-\frac{D}{2}y_0^{-3/2}+
B\sum_{j=0}^{K_4}h_{c,\beta,j}\left(y_0\right)
\left(\frac{B}{t}\right)^j+O_{c,\beta,K_3}\left(\frac{B^{K_4+2}}{t^{K_4+1}}\right)
\ena
for some functions $g_{c,\beta, j}(x)$, $h_{c,\beta, j}(x)$ of polynomially growth,
depending only on $c,\beta,j$, and supported on $x\asymp 1$. In particular,
$g_{c,\beta, 0}(x)=3x^{1/3}$,
$h_{c,\beta, 0}(x)=\left(\frac{2\beta(\beta+1)}{2\beta-1}-\frac{2}{3}\right)x^{-\frac{5}{3}}$.
Note that $g''(y_*)\asymp t$ since $Ay_0^{-1/2}=cty_0^{-\beta}$ and $\beta \neq 1/2$.
Hence, by taking $K_3$ sufficiently large,
\bea\label{J-stationary phase-2}
\mathfrak{J}_r^+(N_1\xi,m,q)
&=&\frac{1}{\sqrt{t}}\mathcal{G}_{\natural}(y_*)
e\left(t\varphi(y_0)+2Dy_0^{1/2}\right)\nonumber\\
&&\times e\left(B\sum_{j=0}^{K_3}g_{c,\beta,j}\left(y_0\right)
\left(\frac{B}{t}\right)^j\right) + O_{A}(  t^{-A}),
\eea
where $ \mathcal{G}_{\natural}(y)=\left(t/(2\pi g''(y))\right)^{1/2}
  \mathcal{G}_A(y)$ satisfies $\mathcal{G}_{\natural}^{(j)}(y)\ll_j 1$.

Plugging \eqref{J-stationary phase-2} into \eqref{I-integral}, we obtain
\bna
&&\mathfrak{I}^+(X)=\frac{1}{t}e\left(t\varphi(y_0)+2Dy_0^{1/2}-t\varphi(y_0')-2D'y_0'^{1/2}\right)
\int_{\mathbb{R}}
\phi\left(\xi\right)\mathcal{G}_{\natural}(y_*)\overline{\mathcal{G}_{\natural}(y_*')}\\
&&\times e\left(-X\xi+\xi^{1/3}\widetilde{B}\sum_{j=0}^{K_3}g_{c,\beta,j}(y_0)
\left(\frac{\widetilde{B}}{t}\right)^j\xi^{j/3}
-\xi^{1/3}\widetilde{B'}\sum_{\jmath=0}^{K_3}g_{c,\beta,\jmath}(y_0')
\left(\frac{\widetilde{B'}}{t}\right)^{\jmath}\xi^{\jmath/3}\right)
\mathrm{d}\xi + O_{A}( t^{-A}),
\ena
where $\widetilde{B}=B\xi^{-1/3}=\left(N_0N_1/(q_1^3q_2^3r)\right)^{1/3}\asymp \lambda$, $D$ is defined in \eqref{AB}, $y_0, y_*$ are as in \eqref{stationary point},
and $D', \widetilde{B}', y_0', y_*'$ are defined in the same way but with $m$, $q_1q_2$ replaced by $m'$, $q_1q_2'$
respectively. Making a change of variable $\xi\rightarrow \xi^3$, the phase function of the
exponential function in the above integral equals
\bna
-X\xi^3+\widetilde{B}\sum_{j=0}^{K_3}g_{c,\beta,j}(y_0)
\left(\frac{\widetilde{B}}{t}\right)^j\xi^{j+1}
-\widetilde{B'}\sum_{\jmath=0}^{K_3}g_{c,\beta,\jmath}(y_0')
\left(\frac{\widetilde{B'}}{t}\right)^{\jmath}\xi^{\jmath+1},
\ena
whose third derivative
\bna
\begin{split}
=&-6X+\widetilde{B}\sum_{j\geq 2}^{K_3}g_{c,\beta,j}(y_0)
\left(\frac{\widetilde{B}}{t}\right)^j(j+1)j(j-1)\xi^{j-2}
-\widetilde{B'}\sum_{\jmath\geq 2}^{K_3}g_{c,\beta,\jmath}(y_0')
\left(\frac{\widetilde{B'}}{t}\right)^{\jmath}(\jmath+1)\jmath(\jmath-1)\xi^{\jmath-2}\\
\gg&\, X+O\left(\frac{\widetilde{B}^3}{t^2}+\frac{\widetilde{B'}^3}{t^2}\right).
\end{split}
\ena
In particular, this is $\gg X$ if we assume $X\gg \lambda^{3+\varepsilon}\max\{t, \sqrt{N_0M}/C\}^{-2}$, since $\widetilde{B}, \widetilde{B'}\asymp \lambda$
and $t\asymp D\asymp \sqrt{N_0M}/C$.
Now by applying the third derivative test
in Lemma \ref{lem: 2st derivative test, dim 1}, we infer that
\bna
\mathfrak{I}^+(X)\ll \max\{t, \sqrt{N_0M}/C\}^{-1} |X|^{-1/3}.
\ena
This proves (3).

(4)
By \eqref{I integral-median} and Lemma \ref{middle integral:lemma} (2), we may write the integral $\mathfrak{I}^+(X)$
in \eqref{I-integral} as
\bna
\mathfrak{I}^+(X)=\frac{27}{\sqrt{|X|}}
\mathop{\int_1^2\int_1^2}_{\frac{Cz_1}{q_1q_2}-\frac{Cz_2}{q_1q_2'}\asymp \frac{X}{\lambda}}
w(z_1,z_2)e\left(f(z_1,z_2)\right)\mathrm{d}z_1\mathrm{d}z_2+O(N_0^{-A}),
\ena
where
\bna
w(z_1,z_2)=\widetilde{V}(z_1^3)\widetilde{V}(z_2^3)z_1^2z_2^2
H_{\natural}\left(\frac{\lambda}{X}\left(\frac{Cz_1}{q_1q_2}-\frac{Cz_2}{q_1q_2'}\right)\right)
\ena
and
\bna
f(z_1,z_2)=t\left(\varphi(z_1^3)-\varphi(z_2^3)\right)+
     2N_0^{1/2}\left(\frac{m^{1/2}z_1^{3/2}}{q_1q_2}-
      \frac{m'^{1/2}z_2^{3/2}}{q_1q_2'}\right)
      +2X\left(\frac{\lambda}{X}
    \left(\frac{Cz_1}{q_1q_2}-\frac{Cz_2}{q_1q_2'}\right)\right)^{3/2}
\ena
with $H_{\natural}(y)=H_1(2Xy^{3/2})H_2(y)$ satisfying
$H_{\natural}^{(j)}(y)\ll_j 1$. Here $H_2(y)$ is a
smooth function supported in $[1/3,3]$, $H_2(y)=1$ on $[\sqrt[3]{4/9}-1/6,\sqrt[3]{9}+1/6]$
and $H_2^{(j)}(y)\ll_j 1$.

By Fourier inversion, we write
\bna
H_{\natural}(y)=\int_\mathbb{R} \widehat{H_{\natural}}(z)e(zy)\mathrm{d}z,
\ena
where $\widehat{H_{\natural}}$ is the Fourier transform of $H_{\natural}$,
satisfying $\widehat{H_{\natural}}(z)\ll_A (1+|z|)^{-A}$.
Thus we may further write
\bna
\mathfrak{I}^+(X)=\frac{27}{\sqrt{|X|}}\int_{-N^\varepsilon}^{N^\varepsilon}
\widehat{H_{\natural}}(z)
\mathop{\int_1^2\int_1^2}_{\frac{Cz_1}{q_1q_2}-\frac{Cz_2}{q_1q_2'}\asymp \frac{X}{\lambda}}
     \widetilde{V}(z_1^3)
\widetilde{V}(z_2^3)z_1^2z_2^2
     e\left(f(z_1,z_2; z)\right)\mathrm{d}z_1\mathrm{d}z_2\mathrm{d}z+O\left(N_0^{-A}\right),
\ena
where
\bna
f(z_1,z_2; z)=f(z_1,z_2)
      +\frac{\lambda z}{X}\left(\frac{Cz_1}{q_1q_2}-\frac{Cz_2}{q_1q_2'}\right).
\ena

Note that
\bna
\frac{\partial f(z_1,z_2; z)}{\partial z_1}
     &=&t\left(\varphi(z_1^3)\right)'+\frac{3N_0^{1/2}
     m^{1/2}z_1^{1/2}}{q_1q_2}+\frac{3\lambda C}{q_1q_2}\left(\frac{\lambda}{X}
    \left(\frac{Cz_1}{q_1q_2}-\frac{Cz_2}{q_1q_2'}\right)\right)^{1/2}
     +\frac{\lambda zC}{q_1q_2X},\\
\frac{\partial f(z_1,z_2; z)}{\partial z_2}
     &=&-t\left(\varphi(z_2^3)\right)'-\frac{N_0^{1/2}m'^{1/2}z_2^{1/2}}{q_1q_2'}
     -\frac{3\lambda C}{q_1q_2'}\left(\frac{\lambda}{X}
    \left(\frac{Cz_1}{q_1q_2}-\frac{Cz_2}{q_1q_2'}\right)\right)^{1/2}
     -\frac{\lambda zC}{q_1q_2'X}.
\ena
Assume $X\gg \lambda^{2+\varepsilon}\max\left\{t,\sqrt{N_0M}/C\right\}^{-1}$.
By the assumption in \eqref{varphi condition}, one has
\bna
\frac{\partial^2 f(z_1,z_2; z)}{\partial z_1^2}
     &=&t\left(\varphi(z_1^3)\right)''+\frac{3N_0^{1/2}m^{1/2}}{2q_1q_2z_1^{1/2}}
     +\frac{3\lambda C^2}{2q_1^2q_2^2}\left(\frac{\lambda}{X}\right)^{1/2}
    \left(\frac{Cz_1}{q_1q_2}-\frac{Cz_2}{q_1q_2'}\right)^{-1/2}\\
     &&\gg \max\left\{t,\sqrt{N_0M}/C\right\},\\
\frac{\partial^2 f(z_1,z_2; z)}{\partial z_2^2}
     &=&-t\left(\varphi(z_2^3)\right)''-\frac{3N_0^{1/2}m'^{1/2}}{2q_1q_2'z_2^{1/2}}
     +\frac{3\lambda C^2}{2q_1^2q_2'^2}\left(\frac{\lambda}{X}\right)^{1/2}
    \left(\frac{Cz_1}{q_1q_2}-\frac{Cz_2}{q_1q_2'}\right)^{-1/2}\\
     &&\gg \max\left\{t,\sqrt{N_0M}/C\right\},\\
\frac{\partial^2 f(z_1,z_2; z)}{\partial z_1\partial z_2}&=& -\frac{3\lambda C^2}{2q_1^2q_2q_2'}\left(\frac{\lambda}{X}\right)^{1/2}
    \left(\frac{Cz_1}{q_1q_2}-\frac{Cz_2}{q_1q_2'}\right)^{-1/2}
    \asymp \frac{\lambda^2}{X}.
\ena
This implies that
\bna
\left|\mathrm{det}f''\right|=\left|\frac{\partial^2 f}{\partial z_1^2}
\frac{\partial^2 f}{\partial z_2^2}
   -\left(\frac{\partial^2 f}{\partial z_1\partial z_2}\right)^2\right|
   \gg \left(\max\left\{t,\sqrt{N_0M}/C\right\}\right)^2,
\ena
for $1\leq z_1, z_2\leq 2$, $\frac{Cz_1}{q_1q_2}-\frac{Cz_2}{q_1q_2'}\asymp \frac{X}{\lambda}$ and $|z|\leq N^\varepsilon$.

By applying the two dimensional second derivative test in
Lemma \ref{lem: 2nd derivative test, dim 2} with
$\rho_1=\rho_2=\max\{t,\sqrt{N_0M}/C\}$,
we have
\bna
\mathfrak{I}^+(X)\ll\max\{t,\sqrt{N_0M}/C\}^{-1}|X|^{-\frac{1}{2}}.
\ena
This completes the proof of (4).

(5) For $X=0$, from \eqref{L square bound}, we have
\bna
\mathfrak{I}^{\pm}(0)\ll t^{-1}.
\ena
To prove the second estimate, from Lemma \ref{middle integral:lemma} (3) we have
\bna
H(0,3\lambda Y)=\int_{\mathbb{R}}\xi^2\phi(\xi^3)e(3\lambda Y\xi)\mathrm{d}\xi
=\phi_0(-3\lambda Y)
\ena
for some Schwartz function $\phi_0$.
For $q_2=q_2'$, to simplify notation we write $q_1q_2=q$. Then
\bna
\mathfrak{I}^{\pm}(0)&=&12\int_{\mathbb{R}}\int_{\mathbb{R}}
\widetilde{V}(z_1^2)
\widetilde{V}(z_2^2)z_1z_2
\phi_0\left(\frac{-3\lambda C(z_1^{2/3}-z_2^{2/3})}{q}\right)\\&&\quad\times
e\left(t\left(\varphi(z_1^2)-\varphi(z_2^2)\right)\pm\frac{2N_0^{1/2}}{q}
\left(m^{1/2}z_1-m'^{1/2}z_2\right)\right)
\mathrm{d}z_1\mathrm{d}z_2;
\ena
see \eqref{I integral-median}.
We make a change of variable $z_1\rightarrow \tau+z_2$.
Then $\mathrm{d}\tau=\mathrm{d}z_1$. From the rapid decay of the Schwartz function $\phi_0(x)$ when $|x|\gg \lambda^{\varepsilon}$, we know $|\tau|\leq \lambda^{-1+\varepsilon}$.
Thus,
%\bna
%\mathfrak{I}(0)&=&27\int_{-\lambda^{-1}N^\varepsilon}^{\lambda^{-1}N^\varepsilon}
%\phi_0\left(\frac{3\lambda C\tau}{q_1q_2}\right)
%\int_1^2\widetilde{V}\left((\tau+z_2)^3\right)\widetilde{V}(z_2^3)
%(\tau+z_2)^2z_2^2\\
%&&\quad\times e\left(t\left(\varphi(N(\tau+z_2)^3)-\varphi(Nz_2^3)\right)+
%\frac{2N^{1/2}m^{1/2}}{q_1q_2}
%(\tau+z_2)^{3/2}-\frac{2N^{1/2}m'^{1/2}}{q_1q_2}z_2^{3/2}\right)
%\mathrm{d}z_2\mathrm{d}\tau\\
%&:=&27\int_{-\lambda^{-1}N^\varepsilon}^{\lambda^{-1}N^\varepsilon}
%\phi_0\left(\frac{3\lambda C\tau}{q_1q_2}\right)
%\int_1^2V_0(\tau,z_2)e\left(f_0(\tau,z_2)\right)\mathrm{d}z_2\mathrm{d}\tau,
%\ena
\bna
\mathfrak{I}^{\pm}(0)
:=\int_{-\lambda^{-1+\varepsilon}}^{\lambda^{-1+\varepsilon}}
\int_1^2V_0(\tau,z_2)
e\left(f_0(\tau,z_2)\right)\mathrm{d}z_2\mathrm{d}\tau+O_{A}(N_0^{-A})
\ena
for any $A>0$, where
$$V_0(\tau,z_2)=12\widetilde{V}\left((\tau+z_2)^2\right)\widetilde{V}(z_2^2)
(\tau+z_2)^2z_2^2\, \phi_0\left(\frac{-3\lambda C}{q}\left((\tau+z_2)^{2/3}-z_2^{2/3}\right)\right)$$ satisfying
\bna
\mathrm{Var}(V_0(\tau,\cdot))=
\int_1^2\left|\frac{\partial V_0(\tau,z_2)}{\partial z_2}\right|\mathrm{d}z_2\ll
\lambda^{\varepsilon}
\ena
and
\bna
f_0(\tau,z_2):=t\left(\varphi((\tau+z_2)^2)-\varphi(z_2^2)\right)\pm
\frac{2N_0^{1/2}}{q}\left(m^{1/2}
(\tau+z_2)-m'^{1/2}z_2\right).
\ena
For $|\tau|<\lambda^{-1+\varepsilon}$, we have (for some $\varsigma$ with
$z_2\leq \varsigma\leq \tau+z_2$)
\bna
\frac{\partial f_0(\tau,z_2)}{\partial z_2}&=&t\tau
\left.\left(\varphi(z^2)\right)''\right|_{z=\varsigma}\pm\frac{2N_0^{1/2}}{q}\left(m^{1/2}
-m'^{1/2}\right)\\
&=&\pm\frac{2N_0^{1/2}}{q}
\left(m^{1/2}-m'^{1/2}\right)
+O\left(\lambda^{-1+\varepsilon}t\right)
\ena
and (for some $z_2\leq \varsigma'\leq \tau+z_2$)
\bna
\frac{\partial^2f_0(\tau,z_2)}{\partial z_2^2}=
t\tau \left.\left(\varphi(z^2)\right)'''\right|_{z=\varsigma'}\ll
\lambda^{-1+\varepsilon}t.
\ena

Now we assume $\left|m^{\frac{1}{2}}-m'^{\frac{1}{2}}\right|
\geq C\lambda^{-1}N_0^{-1/2+\varepsilon} t$ (as otherwise $\mathfrak{I}^{\pm}(0)\ll N_0^{\varepsilon}t^{-1}$).
By using integration by parts once, one knows that
\bna
\int_1^2V_0(\tau,z_2)e(f_0(\tau,z_2))\mathrm{d}z_2
\ll CN_0^{-1/2+\varepsilon}\big|m^{1/2}-m'^{1/2}\big|^{-1}.
\ena
Therefore,
\bna
\mathfrak{I}^{\pm}(0)\ll C\lambda^{-1}N_0^{-1/2+\varepsilon}
\big|m^{1/2}-m'^{1/2}\big|^{-1}.
\ena

%\begin{remark}\label{remark for small C}
%We would like to point out that the claims (1), (5)  for
%the integral $\mathfrak{I}^{\pm}(X)$
%which we proved for large $C$ also hold for
%small $C$'s .
%\end{remark}

\section{Proofs of the corollaries}\label{proofs-of-corollaries}
In this section, we prove Corollaries \ref{sharp-cut-sum}, \ref{subconvexity}, \ref{sum-Fourier-coefficients} in Section \ref{introduction}.
\subsection{Proof of Corollary \ref{sharp-cut-sum}}
Note that from \eqref{GL3-Rankin--Selberg}, we have
\bna
\sum_{N<r^2n\leq N+N/\triangle}|\lambda_\pi(n,r)|^2\ll N/\triangle+N^{4/5}.
\ena
In particular, if $\triangle\leq N^{1/5}$, one has
\bna
\sum_{N<r^2n\leq N+N/\triangle}|\lambda_\pi(n,r)|^2\ll N/\triangle.
\ena
Similarly, under the same assumption $\triangle\leq N^{1/5}$
and from \eqref{GL2: Rankin Selberg}, one obtains
\bna
\sum_{N<n\leq N+N/\triangle}|\lambda_f(n)|^2\ll N/\triangle.
\ena

We choose the smooth function $V$ in \eqref{aim-sum} to be supported on
$[1,2]$ and $V(x)=1$ on $[1+1/\triangle,2-1/\triangle]$. Then, Theorem \ref{main-theorem} yields
\bna
\begin{split}
\sum_{N\leq r^2n\leq 2N}\lambda_{\pi}(n,r)
\lambda_f(n)e\left(t\varphi\bigg(\frac{r^2n}{N}\bigg)\right)\ll&\,
t^{3/5}N^{3/4+\varepsilon}+t^{-9/10}N^{5/4+\varepsilon}+t^{-9\gamma/14}N^{1+\varepsilon}\\
&+
\sum_{N<r^2n\leq N+N/\triangle}|\lambda_\pi(n,r)\lambda_f(n)|+\sum_{2N-N/\triangle<r^2n\leq 2N}|\lambda_\pi(n,r)\lambda_f(n)|\\
\ll&\, t^{3/5}N^{3/4+\varepsilon}+t^{-9/10}N^{5/4+\varepsilon}+t^{-9\gamma/14}N^{1+\varepsilon}\\
&+\bigg(\sum_{N<r^2n\leq N+N/\triangle}|\lambda_\pi(n,r)|^2\bigg)^{1/2}\bigg(\sum_{N<r^2n\leq N+N/\triangle}|\lambda_f(n)|^2\bigg)^{1/2}\\
\ll&\, t^{3/5}N^{3/4+\varepsilon}+t^{-9/10}N^{5/4+\varepsilon}+t^{-9\gamma/14}N^{1+\varepsilon}+N/\triangle,
\end{split}
\ena
as long as $\triangle\leq N^{1/5}$ and $\max\{t^{8/5+2\gamma},t^{12/5+\varepsilon}\}<N<t^{16/5}$.
Corollary \ref{sharp-cut-sum}
then follows by choosing $\triangle=t^{1/2-\varepsilon}$ and
by noting that $t^{1/2-\varepsilon}\leq N^{1/5}$
if and only if $t^{5/2-\varepsilon}\leq N$.

\subsection{Proof of Corollary \ref{subconvexity}}

Applying a standard approximate functional equation argument (see \cite[Theorem 5.3]{IK}), we have
\bea\label{condition1}
L\left(\frac{1}{2}+it,\pi\otimes f\right)\ll_{\pi,f,\varepsilon} t^{\varepsilon}
\sup_{t^{3-\eta}\leq N\leq t^{3+\varepsilon}}
\frac{\left|\mathscr{F}(N)\right|}{\sqrt{N}}+t^{(3-\eta)/2},
\eea
where $\eta>0$ is some parameter to be chosen later, and
\bna
\mathscr{F}(N):=\sum_{r,n\geq 1}\lambda_{\pi}(n,r)\lambda_{f}(n)
\bigg(\frac{r^2n}{N}\bigg)^{-it}V\bigg(\frac{r^2n}{N}\bigg)
\ena
for some smooth function $V$ supported in $[1,2]$ and satisfying $V^{(j)}(x)\ll_j 1$.
Applying Theorem \ref{main-theorem} with $\varphi(x)=-(\log x)/2\pi$ we get
\bea\label{condition2}
\frac{\mathscr{F}(N)}{\sqrt{N}} \ll_{\pi,f,\varepsilon} t^{3/5}N^{1/4+\varepsilon}
\eea
as long as $t^{3-3/10}\leq N\leq t^{3+\varepsilon}$ (we choose $\gamma=11/20$ in Theorem \ref{main-theorem}).
By \eqref{condition1} and \eqref{condition2}, one has
\bna
L\left(\frac{1}{2}+it,\pi\otimes f\right)\ll_{\pi,f,\varepsilon}
t^{3/2-3/20+\varepsilon}+t^{(3-\eta)/2}
\ena
for $\eta\leq 3/10$. Hence we have proved Corollary~\ref{subconvexity} by taking $\eta=3/10$.

\subsection{Proof of Corollary \ref{sum-Fourier-coefficients}}
We assume $N<t^3$, then Corollary \ref{sharp-cut-sum} is simplified to
\bna
\sum_{N\leq r^2n\leq 2N}\lambda_{\pi}(n,r)
\lambda_f(n)e\left(t\varphi\bigg(\frac{r^2n}{N}\bigg)\right)
\ll_{\pi,f,\varepsilon} t^{3/5}N^{3/4+\varepsilon}+t^{-9\gamma/14}N^{1+\varepsilon}
\ena
provided $\max\{t^{8/5+2\gamma}, t^{5/2-\varepsilon}\}<N<t^{3}$.

Specifying to $\varphi(n)=\pm 6n^{1/6}$ (with $t=(xN)^{1/6}$), this implies
$$
\sum_{N\leq r^2n\leq 2N}\lambda_\pi(n,r)\lambda_f(n)e(\pm 6(r^2nx)^{1/6})
\ll x^{1/10}N^{17/20+\varepsilon}+x^{-3\gamma/28}N^{1-3\gamma/28+\varepsilon}
$$
for $\max\{x^{(4+5\gamma)/(11-5\gamma)}, x^{5/7-\varepsilon}\}<N<x$, which simplifies to
$x^{5/7-\varepsilon}<N<x$ if we assume $\gamma<9/20$, for instance.
Consequently,
\bna
\begin{split}
B_6(x,N)=&\sum_{r^2n\leq N}\lambda_\pi(n,r)\lambda_f(n)(r^2n)^{-7/12}\cos\left(12\pi(r^2nx)^{1/6}\right)\\
\ll&\, x^{1/10}N^{4/15+\varepsilon}+x^{-3\gamma/28}N^{5/12-3\gamma/28+\varepsilon}.
\end{split}
\ena
Substituting this into \eqref{FI-Functional-eq} yields
$$\sum_{r^2n\leq x}\lambda_\pi(r,n)\lambda_f(n)\ll x^{31/60}N^{4/15+\varepsilon}
+(xN)^{5/12-3\gamma/28+\varepsilon}
+x^{5/6}N^{-1/6+\varepsilon},$$
which upon choosing $N=x^{19/26}$ to balance the first and third factors implies
%remark: the chose $N=x^{\frac{19}{26}}$ is admissible since $N=x^{\frac{19}{26}}>t^{12/5}=(xN)^{\frac{1}{6}\times \frac{12}{5}}=x^{18/26}$
\bna
\begin{split}
\sum_{r^2n\leq x}\lambda_\pi(r,n)\lambda_f(n)\ll&\, x^{37/52+\varepsilon}+
x^{37/52+(7-135\gamma)/728+\varepsilon}\\
\ll&\, x^{5/7-1/364+\varepsilon},
\end{split}
\ena
as long as $\gamma>7/135$.
%Note that with the choice $N=x^{\frac{19}{26}}$, the lower bound
%assumption $((xN)^{1/6})^{\frac{12}{5}}\ll N$ in Theorem \ref{main-theorem} is justified.
This completes the proof of Corollary \ref{sum-Fourier-coefficients}. Note that the Ramanujan--Petersson assumption arises in applying the identity \eqref{FI-Functional-eq}.

\appendix
	
\section{Estimates for exponential integrals}
Let
\begin{equation*}
 I = \int_{\mathbb{R}} w(y) e^{i \varrho(y)} \mathrm{d}y.
\end{equation*}
Firstly, we have the following estimates for exponential integrals
(see \cite[Lemma 8.1]{BKY}  and \cite[Lemma A.1]{AHLQ}).
	
	\begin{lemma}\label{lem: upper bound}
		Let $w(x)$ be a smooth function    supported on $[ a, b]$ and
        $\varrho(x)$ be a real smooth function on  $[a, b]$. Suppose that there
		are   parameters $Q, U,   Y, Z,  R > 0$ such that
		\begin{align*}
		\varrho^{(i)} (x) \ll_i Y / Q^{i}, \qquad w^{(j)} (x) \ll_{j } Z / U^{j},
		\end{align*}
		for  $i \geqslant 2$ and $j \geqslant 0$, and
		\begin{align*}
		| \varrho' (x) | \geqslant R.
		\end{align*}
		Then for any $A \geqslant 0$ we have
		\begin{align*}
		I \ll_{ A} (b - a)
Z \bigg( \frac {Y} {R^2Q^2} + \frac 1 {RQ} + \frac 1 {RU} \bigg)^A .
		\end{align*}
			\end{lemma}

Next, we need the following evaluation for exponential integrals
which are
 Lemma 8.1 and Proposition 8.2 of \cite{BKY} in the language of inert functions
 (see \cite[Lemma 3.1]{KPY}).

Let $\mathcal{F}$ be an index set, $X: \mathcal{F}\rightarrow\mathbb{R}_{\geq 1}$ and under this map
$T\mapsto X_T$
be a function of $T \in \mathcal{F}$.
\begin{definition}\label{inert}
A family $\{w_T\}_{T\in \mathcal{F}}$ of smooth
functions supported on a product of dyadic intervals in $\mathbb{R}_{>0}^d$
is called $X$-inert if for each $j=(j_1,\ldots,j_d) \in \mathbb{Z}_{\geq 0}^d$
we have
\bna
C(j_1,\ldots,j_d)
= \sup_{T \in \mathcal{F} } \sup_{(y_1, \ldots, y_d) \in \mathbb{R}_{>0}^d}
X_T^{-j_1- \cdots -j_d}\left| y_1^{j_1} \cdots y_d^{j_d}
w_T^{(j_1,\ldots,j_d)}(y_1,\ldots,y_d) \right| < \infty.
\ena
\end{definition}

\begin{lemma}
\label{lemma:exponentialintegral}
 Suppose that $w = w_T(y)$ is a family of $X$-inert functions,
 with compact support on $[Z, 2Z]$, so that
$w^{(j)}(y) \ll (Z/X)^{-j}$.  Also suppose that $\varrho$ is
smooth and satisfies $\varrho^{(j)}(y) \ll Y/Z^j$ for some
$Y/X^2 \geq R \geq 1$ and all $y$ in the support of $w$.
\begin{enumerate}
 \item
 If $|\varrho'(y)| \gg Y/Z$ for all $y$ in the support of $w$, then
 $I \ll_A Z R^{-A}$ for $A$ arbitrarily large.
 \item If $\varrho''(y) \gg Y/Z^2$ for all $y$ in the support of $w$,
 and there exists $y_0 \in \mathbb{R}$ such that $\varrho'(y_0) = 0$ (note $y_0$ is
 necessarily unique), then
 \begin{equation}
  I = \frac{e^{i \varrho(y_0)}}{\sqrt{\varrho''(y_0)}}
 F_{\natural}(y_0) + O_{A}(  Z R^{-A}),
 \end{equation}
where $F_{\natural}(y_0)$ is an $X$-inert function (depending on $A$)  supported
on $y_0 \asymp Z$.
\end{enumerate}
\end{lemma}

We also quote that the following are results well known as the
second derivative tests (see \cite[Lemma 5.1.3]{Hux2} and \cite[Lemma 4]{Mun1}).
	
\begin{lemma}\label{lem: 2st derivative test, dim 1}
		Let $\varrho(x)$ be a real smooth function on  $[a, b]$. Let $w(x)$
be a real smooth function supported on $[ a, b]$ and let $V_0$ be its total
variation.
If $ \varrho^{(r)} (x) \gg \lambda_0$  on $[a, b]$, then
		\begin{align*}
	I\ll \frac {V_0} {\sqrt[r]{\lambda_0}}.
		\end{align*}
	\end{lemma}
	
	\begin{lemma}\label{lem: 2nd derivative test, dim 2}
		Let  $f (x, y   )$ be a real smooth function on  $[a, b] \times [c, d]$ with
		\begin{align*}
		& \left|\partial^2 f / \partial x^2 \right| \gg \rho_1 > 0, \hskip 15pt
\left|\partial^2 f / \partial y   ^2 \right| \gg \rho_2 > 0, \\
		& |\det f''|  = \left|\partial^2 f / \partial x^2 \cdot
\partial^2 f / \partial y   ^2 - (  \partial^2 f / \partial x \partial y    )^2 \right|
\gg \rho_1 \rho_2,
		\end{align*}
		on the rectangle $[a, b] \times [c, d]$.   Let $w (x, y   )$ be
a real smooth function supported on $[ a, b]  \times [c, d]$ and let
		\begin{align*}
		\text{Var} := \int_a^b \int_c^d \left|
\frac {\partial^2 w(x, y   )} {\partial x \partial y   } \right| \mathrm{d} x
 \mathrm{d} y   .
		\end{align*}
		Then
		\begin{align*}
		\int_a^b \int_c^d e (f(x, y   )) w (x, y   ) \mathrm{d} x
 \mathrm{d} y     \ll \frac { \text{Var} } {\sqrt { \rho_1  \rho_2}},
		\end{align*}
		with an absolute implied constant.
	\end{lemma}

Let $U$ be a smooth real
valued function supported on the interval $[a, b] \subset (0, \infty)$
and satisfying $U^{(j)}\ll_{a, b, j} 1$. Let $\xi\in \mathbb{R}$ and
$s= \beta + i \tau\in \mathbb{C}$. We consider the following integral transform
\bna
U^\dagger(\xi, s) := \int_0^{\infty} U(y) e(-\xi y) y^{s-1} \mathrm{d}y.
\ena
By Lemma 5 in \cite{Mun1},
\bea\label{$U$ bound}
U^\dagger (\xi, \beta + i \tau) \ll_{a, b, \beta, j}
\min \left\{ \left(\frac{1+|\tau|}{|\xi|} \right)^j ,
\left(\frac{1+|\xi|}{|\tau|} \right)^j \right\}.
\eea
On the other hand, by the second derivative test,
\bea\label{$U$ bound-2}
U^\dagger (\xi, \beta + i \tau) \ll_{a, b, \beta }
|\tau|^{-1/2}.
\eea

\end{document}